\documentclass{amsart}
\usepackage{all2017} 

\usepackage{tree-commands2} 
\usepackage[colorlinks, citecolor={black}, linkcolor={{black}}, urlcolor={blue}]{hyperref} 
\usepackage{adjustbox}

\usepackage[T1]{fontenc}
\usepackage[utf8]{inputenc}
\usepackage[english]{babel}
 
\makeatletter
\newcommand{\defPartILabel}[2]{\@namedef{r@I-#1}{{#2}{1}{}{cite.dv20}{}}}
\defPartILabel{section-proof-main-result}{5.4}
\defPartILabel{lemma-formula-rphi}{13}
\defPartILabel{lemma-riemann-hurwitz}{15}
\defPartILabel{lemma-combinatorial-type-tmor}{24}
\defPartILabel{lemma-dangling-no-glue}{53}
\defPartILabel{lem-rphi-nd}{54}
\defPartILabel{cor-no-return}{56}
\defPartILabel{prop-local}{57}
\defPartILabel{lemma-pass-once}{60}
\defPartILabel{prop-adm-matrix}{61}
\defPartILabel{lemma-limit-dangling-no-glue}{66}
\defPartILabel{lemma-limit-matrix-change}{69}
\defPartILabel{lemma-limit-full-rank}{70}
\defPartILabel{lemma-limit-pass-once}{71}
\defPartILabel{lemma-limit-no-return}{73}
\defPartILabel{lemma-loop-bridge}{79}
\defPartILabel{lemma-vertices-in-GqA0}{87}
\defPartILabel{subsection-case-work}{6.7}
\defPartILabel{sub-deg2nd3}{6.7.4}
\defPartILabel{eq-star}{{$\star$}}
\makeatother

\usepackage{csquotes}

\DeclareMathAlphabet{\mathpzc}{OT1}{pzc}{m}{it}

\calclayout

\title[Catalan-many tropical morphisms to trees]{Catalan-many tropical morphisms to trees;
	\\ \small{Part II: A space and a count}}
\author{Alejandro Vargas}

\begin{document}

\begin{abstract}
In their work on Brill--Noether theory, Eisenbud and Harris established the geometry of the universal parameter space of linear series over curves, proving that for even genus $g$ and degree $d = g/2 + 1$, the projection to the moduli space of curves is a finite cover of degree equal to the Catalan number $C_{g/2} = \frac{1}{g/2+1}\binom{g}{g/2}$. In this paper, we construct the tropical counterpart of this universal family: a polyhedral cone complex $\TM d g$ parametrizing degree-$d$ tropical morphisms from genus-$g$ metric graphs to metric trees. For even $g$ and $d = g/2 + 1$, we prove that the forgetful projection $\Pi \colon \TM d g \to \MTrop g$ is a branched cover of degree $C_{g/2}$ equipped with natural determinantal multiplicities. We compute this degree by showing that on caterpillars of loops the morphisms are in bijection with ballot sequences, and we establish its global invariance across $\MTrop g$ via a tropical balancing condition across codimension-$1$ walls. Via deformation and path lifting, this yields an effective method to construct Catalan-many gonality-witnessing maps for any generic metric graph, establishing that the tree gonality of any genus-$g$ metric graph is at most $\lceil g/2 \rceil + 1$.
\end{abstract}

\maketitle

\section*{Introduction} 
We investigate \emph{tree gonality} of metric graphs. We envision the outcome as a rank-1 tropical version of Theorem~1 in \cite{eh87}. Fix natural numbers $g$, $r$, and $d$, and define the \emph{Brill--Noether number} as $\rho(g, r, d) = g - (r+1)(g-d+r)$. By the Brill--Noether theorem~\cite{gh80}, the variety $G^r_d(X)$ of linear series of degree $d$ and dimension~$r$ on a smooth genus-$g$ curve $X$ with general moduli is smooth of dimension~$\rho$. Let $C(g,r,d)$ be the number of points in $G^r_d(X)$ when $\rho = 0$. Castelnuovo~\cite{cas89} predicted via a classical count that for $r=1$ and even genus $g=2g'$, this number is the $g'$-th Catalan number $C_{g'} = \frac{1}{g'+1}\binom{2g'}{g'}$. More generally, Kempf~\cite{kem71} and Kleiman--Laksov~\cite{kl72} established the determinantal degree formula
\begin{align} \label{eq:kempf-kleiman-laksov}
C(g,r,d) = g! \prod_{i=0}^r \dfrac{i!}{(g-d+r+i)!}.
\end{align}
Theorem~1 in \cite{eh87} asserts the existence of smooth families $\pi \colon \mathfrak X \to B$ of smooth genus-$g$ curves such that $\mathcal{G}^r_d(\mathfrak X/B)$ is smooth and connected, and the fibers of the natural map $\mathcal{G}^r_d(\mathfrak X/B) \to B$ have dimension $\rho$, with degree given by $C(g,r,d)$ when $\rho = 0$. Recall that basepoint-free linear series of degree $d$ and dimension~$r$ on a smooth curve~$X$ correspond to morphisms $X \to \PP^r$. In essence, this classical result describes the geometry of a universal parameter space of morphisms.

Here we construct the tropical analogue of this universal family for pencils of degree $d = g' + 1$ on metric graphs of even genus $g = 2g'$: a generalized cone complex $\TM d g$ parametrizing tropical morphisms to metric trees. The tropical objects playing the role of algebraic curves and of the projective line are metric graphs and metric trees, respectively. The metric graph morphisms that we use have properties which make them discrete analogues of algebro-geometric morphisms; such \emph{tropical morphisms} have been studied in \cite{bn09, cap14, mik07, bbm11, cha13, abbr15, cmr16}, among others. In this setting, tropical morphisms to metric trees lead naturally to the concept of \emph{tree gonality} of a metric graph $\mG$, defined as the minimum degree of a tropical morphism from any tropical modification of $\mG$ to any metric tree. A tropical modification is an operation which retracts or attaches edges ending in a valency-1 vertex.
	 
While in Part~I~\cite{dv20} we addressed a question of Cools and Draisma~\cite{cd18} by constructing morphisms within top-dimensional trivalent cells of $\MTrop g$, our goal here is to construct the global parameter space $\TM d g$ across all of $\MTrop g$. As a set, $\TM d g$ consists of pairs of a genus-$g$ metric graph $\mH$ and a tropical morphism of degree~$d$ from a tropical modification of $\mH$ to a metric tree. Topologically, $\TM d g$ is constructed by gluing together finitely many cones of dimension $3g-3$ that are specified by combinatorial data, and identifying points that describe the same map. Informally, the morphisms parametrized by $\TM d g$ move in families of dimension $3g-3$, the dimension of the moduli space of metric graphs $\MTrop g$. We show that for even $g=2g'$, $\TM {g'+1} g$ surjects onto the moduli space of metric graphs $\MTrop g$. Our main result is:

\begin{thm} \label{thm}
	Let $g=2g'$ be even and $d = g' + 1$. The canonical projection
	\[
	\Pi \colon \TM d g \longrightarrow \MTrop g, \qquad (\mH, \tmor) \longmapsto \mH,
	\]
	is a surjective branched cover of generalized cone complexes of degree $\deg \Pi = C_{g'}$, the $g'$-th Catalan number. Moreover, $\TM d g$ is connected through codimension~1.
\end{thm}

In particular, for every $\mH \in \MTrop g$, the fiber over $\mH$ yields Catalan-many tropical morphisms of degree $g'+1$ to metric trees when counted with natural multiplicity:
\begin{equation} \label{eq:fiber-count}
\sum_{\tmor \in \inv \Pi(\mH)} \absMult \tmor = C_{g'} = \dfrac{1}{g'+1}\binom{2g'}{g'}.
\end{equation}
Furthermore, on a dense open locus of $\MTrop g$, these fibers capture all degree-$d$ tropical morphisms to metric trees. 

Because $\Pi$ is surjective, every genus-$g$ metric graph admits a tropical morphism to a metric tree of degree at most $\lceil g/2 \rceil + 1$, yielding an effective proof of the tree gonality bound:
\begin{cor} \label{theorem-gonality}
	The tree gonality of any genus-$g$ metric graph $\mG$ is at most $\lceil g/2 \rceil + 1$.
\end{cor}

While Theorem~\ref{thm} is formulated for even genus, Corollary~\ref{theorem-gonality} holds for arbitrary $g$: as in Part~I~\cite{dv20}, attaching a bridge and a self-loop reduces the odd case $g$ to $g+1$. Furthermore, because the proof of Theorem~\ref{thm} does not depend on algebro-geometric machinery and is purely effective, it yields a concrete combinatorial algorithm to construct degree $\lceil g/2 \rceil + 1$ rank-1 divisors on any metric graph by walking along deformation paths in $\TM d g$ from the caterpillar of loops (Proposition~\ref{prop-divisors-on-chain}).
		
It is an open question, for $d \le \lceil g /2 \rceil$, whether $\TM {d} g$ is the tropicalization of the Hurwitz space of degree-$d$ covers of $\PP^1$ (whose image in $\MTrop g$ would tropicalize the $d$-gonal locus of $\calM g$). For even genus $g=2g'$ and $d = g'+1$, the question asks whether $\TM d g$ tropicalizes the universal parameter space of linear series studied by Eisenbud and Harris~\cite{eh87}, and whether the algebraic cycle multiplicities coincide with our determinantal multiplicities $\absMult \dtmor$. That $\Pi \colon \TM d g \to \MTrop g$ is a branched cover of generalized cone complexes provides compelling evidence for such an underlying non-Archimedean tropicalization.

\subsection*{Proof idea} 
The proof of Theorem~\ref{thm} proceeds via a continuous deformation argument across the moduli space:
\begin{enumerate}
	\item \textbf{Base count on caterpillars of loops:} We introduce a distinguished class of metric graphs $\HCL g$, termed \emph{caterpillars of loops}, consisting of a caterpillar graph that at the leaves has loops. For a generic caterpillar of loops $\mH_L$, the edge-length matrix $A_\dtmor$ (which expresses source edge lengths in terms of target edge lengths) is a diagonal matrix, each multiplicity $\absMult \dtmor$ equals~1, and the tropical morphisms in $\TM d g$ over $\mH_L$ are in bijection with length-$g$ Dyck paths (or ballot sequences). This establishes the base count of $C_{g'}$ tropical morphisms on the caterpillar locus.
	\item \textbf{Invariance within trivalent cells (Part~I):} Within a top-dimensional cell of $\MTrop g$, deformations vary target edge lengths while the combinatorial type of the source remains fixed. Across codimension-1 trivalent limits $\dtmor_0$, the signed multiplicities satisfy the local balancing relation
	\[ \sum_{\dtmor \in \cstar{\dtmor_0}} \Mult \dtmor = 0, \]
	which ensures that the total count is constant within each cell.
	\item \textbf{Chamber transitions across non-trivalent walls (\emph{Walking Through II}):} To deform between different trivalent cells of $\MTrop g$, edge lengths shrink to zero, producing non-trivalent limits with vertices of valency $\ge 4$ or contracted loops. Resolving the 8 boundary casework scenarios shows that every codimension-1 limit extends into adjacent full-dimensional chambers with equal multiplicity ($\Mult \dtmor = \Mult{\dtmor'}$), preserving the count across cell boundaries.
	\item \textbf{Global codimension-1 connectivity:} Because the moduli space $\MTrop g$ is connected through codimension~1, any generic metric graph $\mH$ can be joined to a generic caterpillar of loops $\mH_L$ by a continuous piecewise-linear path that avoids codimension-2 strata. The invariance of the count along this path implies that $\sum \absMult \tmor = C_{g'}$ holds for all generic metric graphs.
\end{enumerate}

\subsection*{Organization of the paper} 
In~Section~\ref{section-background} we review metric graphs, tropical modifications, and the moduli space~$\MTrop g$. 
In~Section~\ref{section-properties-of-tropical-morphisms} we summarize the properties of discrete tropical morphisms, edge-length matrices, and full-dimensional cones from~\cite{dv20}. 
In~Section~\ref{section-a-space-of-tropical-morphisms} we construct the moduli space~$\TM d g$, analyze its symmetries, define the projection~$\Pi$, and compute the Catalan base count on caterpillars of loops. 
In~Section~\ref{sec-deformation-invariance} we establish the invariance of the count via continuous deformation, prove the tropical balancing condition, and complete the proof of Theorem~\ref{thm}. 
Section~\ref{sec-constructions} contains the exhaustive casework resolving the boundary limits. 
Finally, Appendix~\ref{appendix-trivalent-deformation} summarizes the trivalent deformation~cases.

\subsection*{Related work}
In a recent pair of articles, Robayo~\cite{rob24, rob25} develops a sophisticated tropical intersection theory tailored to tropical moduli spaces. As an application, he proves that loci of curves of fixed gonality form a tropical cycle. However, establishing the balancing condition in that framework relies on input from algebraic geometry (via Jun Li's degeneration formula), and is therefore neither purely combinatorial nor constructively effective. In degree $d = g' + 1$, he extends our caterpillar of loops construction to accommodate marked legs, thereby recovering and extending our Catalan count. In contrast, our approach provides a self-contained, constructive deformation theory directly on the parameter space $\TM d g$, establishing the count and the tree gonality bound through purely tropical and combinatorial methods.

\section*{Acknowledgements} 
I express my deep admiration and gratitude to Jan Draisma for years of mentorship and mathematics. I also warmly thank Erwan Brugall\'e, Melody Chan, Karl Christ, Livio Ferretti, Hannah Markwig, Dhruv Ranganathan, Diego Robayo, and Sam Payne for enlightening conversations and working out examples together. A portion of this research was conducted during the semester ``Tropical Geometry, Amoebas and Polytopes'' at the Institut Mittag-Leffler (Djursholm, Sweden), whose hospitality and stimulating research environment are gratefully acknowledged. This research was supported by the Swiss National Science Foundation, grant number 200142, and by the UK Engineering and Physical Sciences Research Council (EPSRC) grant number EP/X02752X/1.

\renewcommand{\thesubsection}{\thesection.\arabic{subsection}}
\numberwithin{thm}{section}
\setcounter{thm}{0}

\section{Metric graphs and moduli} \label{section-background}

Definitions and conventions follow \cite{dv20}. We briefly fix notation for metric graphs, tropical modifications, and the moduli space $\MTrop g$.

\subsection{Graphs} \label{subsection-graphs}
All graphs $G = (V(G), E(G))$ are finite and connected, with loops and multiple edges permitted. The \emph{genus} of $G$ is $g(G) = |E(G)| - |V(G)| + 1$; a \emph{tree} is a graph of genus 0, and its monovalent vertices are \emph{leaves}. For a vertex $A \in V(G)$, we write $\vale A$ for its valency (loops counting twice), and $\Neigh A \subset E(G)$ for the edges incident to $A$. A graph with minimum valency at least 3 is called a \emph{combinatorial type}; if every vertex has valency 3, it is \emph{trivalent}. Edge deletions, contractions, and subdivisions are defined in the standard way. A \emph{graph morphism} $\gamma: G \to G'$ preserves incidences and may contract edges: $\gamma(V(G)) \subseteq V(G')$, and for each edge $e$ with endpoints $A, B$, either $\gamma(e) \in E(G')$ with endpoints $\gamma(A), \gamma(B)$, or $\gamma(e) = \gamma(A) = \gamma(B) \in V(G')$. Fixing an edge ordering identifies the real vector space $\RR^{E(G)}$ of edge functions with $\RR^{|E(G)|}$.

\subsection{Metric graphs} \label{section-metric-graphs}
A \emph{metric graph} $\mG$ is a compact, connected metric space obtained from a graph $G$ by assigning positive lengths $\ell: E(G) \to \RRo$ to its edges and equipping the resulting one-dimensional CW-complex with the shortest-path metric. We call $(G, \ell)$ a \emph{realization} and $G$ a \emph{model} of $\mG$. The genus is $g(\mG) = b_1(\mG) = g(G)$.

A metric graph possesses infinitely many models via edge subdivisions. A finite subset $S \subset \mG$ is called a \emph{vertex set} if $\mG \setminus S$ is a disjoint union of open intervals; each $S$ induces a realization $(G_S, \ell_S)$ whose vertices are $S$ and whose edges are the connected components of $\mG \setminus S$. A point $x \in \mG$ is an \emph{essential vertex} if no open neighborhood of $x$ is isometric to an open interval $(-\varepsilon, \varepsilon)$, i.e., $x$ has valency $\vale x \ne 2$.

\begin{lm} \label{lemma-model-metric-graph}
	Let $\mG$ be a metric graph of genus $g \ge 2$, and let $\calE_\mG \subset \mG$ be its set of essential vertices. Then $\calE_\mG$ is non-empty and finite, and a finite subset $S \subset \mG$ is a vertex set if and only if $\calE_\mG \subseteq S$.\qed
\end{lm}

The realization $(\calE_\mG, \ell_{\calE_\mG})$ induced by $\calE_\mG$ is the \emph{essential realization}, and the graph $\calE_\mG$ is the \emph{essential model}. The only metric graphs without essential vertices are metric loops ($g = 1$, i.e., cycles), which we exclude. Identifying vertices of $\calE_\mG$ with their underlying points in $\mG$, we also write $\calE_\mG \subset \mG$ for the set of essential vertices when no ambiguity arises. The essential model is minimal: every model of $\mG$ is obtained from $\calE_\mG$ by edge subdivisions.

\subsection{Tropical modification} \label{subsection-tropical-modification}
\emph{Tropical modification} is the equivalence relation on metric graphs generated by attaching or removing trees. To treat this operation combinatorially, we use the notion of \emph{dangling elements} introduced in \cite{dv20}:

\begin{de} \label{de-dangling-char}
	Let $G$ be a graph, and $\mG = (G, \ell)$ a metric graph.
	An edge $e \in E(G)$ is \emph{dangling} if $G \setminus \{e\}$ has a connected component that is a tree; a vertex $A \in V(G)$ is \emph{dangling} if every edge incident to $A$ is dangling. A point $x \in \mG$ is \emph{dangling} if it corresponds to a dangling vertex or lies in the interior of a dangling edge. We denote by $\Neighnd A$ the \emph{non-dangling edge neighbourhood} of $A$ (the set of edges incident to $A$ that are not dangling), and by $\ndval A = |\Neighnd A|$ its \emph{non-dangling valency}.
\end{de}

Deleting all dangling points of $\mG$ produces a canonical metric graph $\mH$ with $g(\mH) = g(\mG)$. Its essential model $H$ has minimum valency at least 3, and is called the \emph{combinatorial type} of $\mG$, while $\mH$ is the \emph{deletion of dangling trees} of $\mG$.

\begin{lm}[\cite{dv20}] \label{lemma-combinatorial-type}
	Let $\mG$ be a metric graph of genus $g \ge 2$. If $\mG'$ is equivalent to $\mG$ under tropical modification and has a model $G'$ with minimum valency at least 3, then $\mG'$ is isometric to $\mH$ and $G'$ is isomorphic to $H$.\qed
\end{lm}

\subsection{Moduli space of metric graphs} \label{section-moduli-space-of-metric-graphs}
We denote by $\mathbb G_g$ the finite set of genus-$g$ trivalent combinatorial types. 

\begin{lm} \label{lemma-trivalent}
	Let $g \ge 2$. For every genus-$g$ combinatorial type $H_0$, there exists $H \in \mathbb G_g$ and a sequence of edge contractions of $H$ yielding a graph isomorphic to $H_0$.\qed
\end{lm}

For a graph $G$, the cone of length functions is the positive orthant $C_G = \RRo^{E(G)}$. We partially compactify $C_G$ to $\overline{C}_G \subset \RRgo^{E(G)}$ by allowing edge lengths to be zero: given $\ell \in \overline{C}_G$, we contract all edges with $\ell(e) = 0$, requiring that all cycles of $G$ retain strictly positive length so that the genus remains $g$.

\begin{de}
	The moduli space of genus-$g$ metric graphs is the topological space
	\[\MTrop g = \left(\bigsqcup_{H \in \mathbb G_g} \overline{C}_H \right) \Big / \cong, \]
	where $\cong$ identifies points that define isometric metric graphs, equipped with the quotient topology.
\end{de}

By Lemmas~\ref{lemma-combinatorial-type} and \ref{lemma-trivalent}, every genus-$g$ metric graph has a unique representative in $\MTrop g$ up to tropical modification and isometry. Since trivalent graphs have $|E(H)| = 3g - 3$, $\MTrop g$ has pure dimension $3g - 3$. Unlike the extended moduli spaces of \cite{bmv11, cap12, acp15}, $\MTrop g$ here carries no marked points, infinite legs, or vertex weights. While one could consider vertex weights to fully compactify the space by contracting cycles, this would introduce unnecessary technical overhead: our continuous deformation argument operates entirely within the locus where all $g$ cycles retain positive length.

\subsection{Symmetry in \texorpdfstring{$\MTrop g$}{MTrop g}} \label{subsection-symmetry-in-mtropg}
For a combinatorial type $H$, the automorphism group $\Aut H$ acts on the open cone $C_H$ by length pullback: $\gamma \cdot \ell = \ell \circ \inv \gamma$ for $\gamma \in \Aut H$.

\begin{lm} \label{lemma-isometry-to-isomorphism}
	Let $\Aq \mG 1, \Aq \mG 2$ be metric graphs, $S$ a vertex set of $\Aq \mG 1$, and $\Psi : \Aq \mG 1 \to \Aq \mG 2$ a continuous map. Then $\Psi$ is an isometry if and only if $\Psi(S)$ is a vertex set of $\Aq \mG 2$, the induced map $\gamma_\Psi$ is a graph isomorphism, and $\Aq \ell 2 = \Aq \ell 1 \circ \inv \gamma_\Psi$.\qed
\end{lm}

Applying Lemma~\ref{lemma-isometry-to-isomorphism} to essential vertex sets $\calE_{\Aq \mG 1}, \calE_{\Aq \mG 2}$ shows that two points in $C_H$ define isometric metric graphs if and only if they lie in the same $\Aut H$-orbit. Consequently, a metric graph $\mG$ with combinatorial type $H$ and trivial isometry group is represented exactly $\#\Aut H$ times in $C_H$.

\subsection{Notation} \label{sec-notation}
Throughout, $\mG$ denotes a metric graph of genus $g \ge 2$, $\mT$ a metric tree of genus $0$, and $\mH$ the deletion of dangling elements of $\mG$. For combinatorial models, $G$ denotes a graph of genus $g \ge 2$, $T$ a tree, and $H$ a combinatorial type (minimum valency $\ge 3$). The essential model of $\mG$ (and its set of essential vertices) is denoted by $\calE_\mG$, $\Neighnd A$ denotes the non-dangling edge neighbourhood of $A$, and $\ndval A = |\Neighnd A|$ its non-dangling valency. Discrete graph morphisms are denoted by~$\dtmor$, and tropical morphisms of metric graphs by~$\tmor$. We write $[d] = \{1, 2, \dots, d\}$, and $x \in G$ means $x \in V(G) \cup E(G)$.


\section{Properties of tropical morphisms} \label{section-properties-of-tropical-morphisms}

In this section we review the properties of discrete tropical morphisms and tropical morphisms to trees from \cite{dv20}. We introduce the cone of sources and its edge-length matrix, characterize full-rank morphisms, and identify the local combinatorial structures that forbid a morphism from being full-rank.

\subsection{Tropical morphisms} \label{subsection-tropical-morphisms}
A tropical morphism is a continuous piecewise linear map between metric graphs with integer slopes that satisfies the harmonic balancing condition at every vertex, tropicalizing ramified covering maps of algebraic curves \cite{cap14}. Its underlying combinatorial structure is encoded by a discrete tropical morphism. Following \cite{dv20}, all graphs considered here are unweighted.

\begin{de}[discrete tropical morphism] \label{def-discrete-tropical-morphism}
	Let $\dtmor \colon G \to G'$ be a morphism of finite connected graphs, and let $m_\dtmor \colon G \to \ZZgo$ be a map.
	\begin{enumerate}[(a)] 
		\item $m_\dtmor$ is an \emph{index map} for $\dtmor$ if for every $e \in E(G)$, $\dtmor(e) \in V(G')$ if and only if $m_\dtmor(e) = 0$.
		\item $\dtmor$ is \emph{harmonic} with index map $m_\dtmor$ if for every $A \in V(G)$ and every edge $e' \in E(G')$ incident to $\dtmor(A)$, the \emph{balancing condition} holds:
		\[ m_\dtmor(A) = \sum_{\substack{e \in \Neigh A \\ \dtmor(e) = e'}} m_\dtmor(e).\]
		In particular, this sum is independent of the choice of $e'$.
		\item $\dtmor$ is \emph{non-degenerate} if $m_\dtmor(x) \ge 1$ for all $x \in G$ (so no edges are contracted).
		\item $\dtmor$ satisfies the \emph{Riemann--Hurwitz condition} (RH-condition) if for every $A \in V(G)$:
		\[ r_\dtmor(A) = (\vale A - 2) - m_\dtmor(A) (\vale \dtmor(A) - 2) \ge 0. \]
		\item A \emph{discrete tropical morphism} (\DTmorp) is a pair $(\dtmor, m_\dtmor)$ consisting of a non-degenerate harmonic morphism $\dtmor$ with index map $m_\dtmor$ satisfying the RH-condition. We abbreviate $m_\dtmor$ to $m$ when no confusion arises.
	\end{enumerate}
\end{de}

\begin{lm}[{\cite[Lemma~2.4]{cap14}}, degree of $\dtmor$] \label{lm:degree-of-a-dtmor}
	Let $\dtmor \colon G \to G'$ be a \DTmorp. The fiber count with multiplicities,
	\[\deg \dtmor = \sum_{x \in \inv \dtmor(x')} m(x),\]
	is independent of the choice of $x' \in G'$ and is called the \emph{degree} of $\dtmor$. \qed
\end{lm}

\begin{lm}[{\cite[Lemma~\ref{I-lemma-formula-rphi}]{dv20}}, formula for $r_\varphi$] \label{lm:formula-rphi} 
	Let $\dtmor \colon G \to G'$ be a \DTmorp. Then for every $A \in V(G)$,
	\[\pushQED{\qed}
	r_\dtmor(A) = 2(m(A) - 1) - \sum_{e \in \Neigh A} (m(e) - 1).\qedhere\popQED
	\]
\end{lm}

\begin{lm}[{\cite[Lemma~\ref{I-lemma-riemann-hurwitz}]{dv20}}, Riemann--Hurwitz formula] \label{lm:riemann-hurwitz} 
	Let $\dtmor \colon G \to G'$ be a \DTmorp. Then
	\[\pushQED{\qed}
	2g(G) - 2 = \deg \dtmor \cdot (2g(G')-2) + \sum_{A \in V(G)} r_\dtmor(A).\qedhere\popQED
	\] 
\end{lm}

Since $r_\dtmor(A) \ge 0$ for all $A$, Lemma~\ref{lm:riemann-hurwitz} implies $g(G) \ge g(G')$. We call $G$ the \emph{source} and $G'$ the \emph{target} of $\dtmor$. Choosing metric data on the target yields a tropical morphism:

\begin{de}[tropical morphism] \label{definition-tropical-morphism}
	Let $\dtmor \colon G \to G'$ be a \DTmor and $\ell_{G'} \in C_{G'}$. The corresponding \emph{tropical morphism} is the unique continuous map $\tmor \colon \mG \to \mG'$ from the metric graph $\mG = (G, e \mapsto \ell_{G'}(\varphi(e))/m_\dtmor(e))$ to $\mG' = (G', \ell_{G'})$ whose restriction to each edge segment $\mG_e$ is an affine bijection onto $\mG'_{\dtmor(e)}$ of slope $m_\dtmor(e)$, and that agrees with $\dtmor$ on vertices.
\end{de}

We call $\dtmor$ a \emph{model} and $(\dtmor, \ell_{G'})$ a \emph{realization} of $\tmor$. The fiber count with slopes and vertex multiplicities equals $\deg \dtmor$ everywhere by Lemma~\ref{lm:degree-of-a-dtmor}, so $\tmor$ is a tropical morphism in the sense of \cite{cd18}. Recalling that $\calE_\mG$ and $\calE_{\mG'}$ denote the essential models of $\mG$ and $\mG'$ (and their sets of essential vertices), every model of a tropical morphism $\tmor \colon \mG \to \mG'$ is determined by choosing a vertex set $\mathcal{S}' \subset \mG'$ containing $\tmor(\mathcal{E}_\mG) \cup \mathcal{E}_{\mG'}$ and setting $\mathcal{S} = \inv\tmor(\mathcal{S}')$.

\begin{cons} \label{construction-canonical-model}
	Let $\tmor \colon \mG \to \mG'$ be a tropical morphism, and let $\mathcal{S}' \subset \mG'$ be a finite subset such that $\tmor(\mathcal{E}_{\mG}) \cup \mathcal{E}_{\mG'} \subseteq \mathcal{S}'$. By Lemma~\ref{lemma-model-metric-graph}, the sets $\mathcal{S} = \inv \tmor(\mathcal{S}')$ and $\mathcal{S}'$ are vertex sets inducing realizations $\mG = (G_\mathcal{S}, \ell_\mathcal{S})$ and $\mG' = (G_{\mathcal{S}'}, \ell_{\mathcal{S}'})$.
	Define $\dtmor_{\mathcal{S}'} \colon G_\mathcal{S} \to G_{\mathcal{S}'}$ by sending each vertex $x \in \mathcal{S}$ to $\tmor(x)$, and each edge $e \in E(G_\mathcal{S})$ to the edge $e' \in E(G_{\mathcal{S}'})$ containing $\tmor(e \setminus \mathcal{S})$, with index map $m_{\dtmor_{\mathcal{S}'}}(e) = \ell_{\mathcal{S}'}(\dtmor_{\mathcal{S}'}(e)) / \ell_\mathcal{S}(e)$. 
	By construction, $\dtmor_{\mathcal{S}'}$ is a \DTmor and $(\dtmor_{\mathcal{S}'}, \ell_{\mathcal{S}'})$ realizes $\tmor$. When $\mathcal{S}' = \tmor(\mathcal{E}_{\mG}) \cup \mathcal{E}_{\mG'}$, the induced map is denoted $\dtmor_\ess$ and called the \emph{essential model}; it is minimal in the sense that every model $\dtmor_{\mathcal{S}'}$ arises as an edge subdivision of $\dtmor_\ess$.
\end{cons}

\begin{de}[isomorphism of morphisms] \label{definition-tropical-iso} 
	Let $\dtmor_1 \colon G_1 \to G'_1$ and $\dtmor_2 \colon G_2 \to G'_2$ be graph morphisms. An isomorphism from $\dtmor_1$ to $\dtmor_2$ is a pair of graph isomorphisms $(\gamma, \gamma')$ making Diagram~\ref{diagram-iso}(a) commute. If $\dtmor_1, \dtmor_2$ are \DTmorsp, $(\gamma,\gamma')$ is an isomorphism of \DTmors if the index map $m_{\dtmor_2}$ pulls back to $m_{\dtmor_1}$ via $\gamma$ (Diagram~\ref{diagram-iso}(b)). An isomorphism between tropical morphisms $\tmor_1, \tmor_2$ is a pair of isometries $(\Psi, \Psi')$ making Diagram~\ref{diagram-iso}(c) commute.
	
	\begin{minipage}[b]{0.95\textwidth}
		\centering
		\begin{minipage}{0.3\textwidth}
			\centering
			\[
			\begin{tikzcd}
			G_1 \arrow[r,"\dtmor_1"] \arrow[d,swap,"\gamma"] &
			G'_1 \arrow[d,"\gamma'"] \\
			G_2 \arrow[r,"\dtmor_2"] & G_2'
			\end{tikzcd}	
			\]
			(a)
		\end{minipage}
		\begin{minipage}{0.3\textwidth}
			\centering
			\[ 
			\begin{tikzcd}
			G_1 \arrow[d,swap,"\gamma"] \arrow{dr}{m_{\dtmor_1}} \\
			G_2 \arrow[r,swap,"m_{\dtmor_2}"] & \mathbb Z_{\ge 0}
			\end{tikzcd} 
			\]
			(b)
		\end{minipage}
		\begin{minipage}{0.3\textwidth}
			\centering
			\[ 
			\begin{tikzcd}
			\mG_1 \arrow[r,"\tmor_1"] \arrow[d,swap,"\Psi"] &
			\mG'_1 \arrow[d,"\Psi'"] \\
			\mG_2 \arrow[r,"\tmor_2"] & \mG_2'
			\end{tikzcd} 
			\]
			(c)
		\end{minipage}
		\newcounter{diagram}
		\refstepcounter{thm}
		\label{diagram-iso}
		Diagram \ref{diagram-iso}
	\end{minipage}
\end{de}

For $g(G) \ge 2$, a fibre $\inv \dtmor(x')$ is a \emph{dangling fibre} if all its elements are dangling in $G$. Let $\hat E' \subset E(G')$ be the set of edges whose fibre is not dangling, $\hat G' \subseteq G'$ the subgraph induced by $\hat E'$, and $\hat G$ the unique connected component of $\inv \dtmor(\hat G')$ of positive genus. The \emph{deletion of dangling fibres} of $\dtmor$ is the restriction $\hat \dtmor = \dtmor|_{\hat G} \colon \hat G \to \hat G'$. Under tropical modification, we identify $\dtmor$ with the essential model $\hat \dtmor_\ess$, and $\tmor = (\dtmor, \ell)$ with $\hat \tmor = (\hat \dtmor, \ell|_{\hat G'})$.

\begin{de} \label{definition-combinatorial-type-tmor}
	Let $\tmor$ be a tropical morphism and $\hat \tmor$ its deletion of dangling fibres. The \emph{combinatorial type} of $\tmor$ is the essential model $\hat \dtmor_\ess$ of $\hat \tmor$.
\end{de}

\begin{de} \label{definition-combinatorial-type}
	A \emph{combinatorial type of \DTmorsp} is a \DTmor $\dtmor \colon G \to G'$ without dangling fibres such that $\sum_{A \in \inv \dtmor(v)} r_\dtmor(A) \ge 1$ for every divalent vertex $v \in V(G')$. By \cite[Lemma~\ref{I-lemma-combinatorial-type-tmor}]{dv20}, if a tropical modification $\bar \tmor \colon \bar \mG \to \bar \mG'$ of $\tmor$ has a model $\bar \dtmor$ that is a combinatorial type, then $\bar \tmor \cong \hat \tmor$ and $\bar \dtmor \cong \hat \dtmor_\ess$, so the combinatorial type is invariant under tropical modifications.
\end{de}

Given a non-loop edge $e' \in E(G')$, contracting $e'$ in $G'$ and the edges $\inv \dtmor(e')$ in $G$ yields graphs $G'_0$ and $G_0$, and an induced morphism $(\dtmor_0, m_{\dtmor_0}) \colon G_0 \to G'_0$, called the \emph{limit at} $e'$. It satisfies conditions (a)--(c) of Definition~\ref{def-discrete-tropical-morphism}.

\begin{prop}[$r_\varphi$ under contraction] \label{prop-rphi-under-contraction}
	Let $e' \in E(G')$ be a non-loop edge, let $\dtmor_0 \colon G_0 \to G'_0$ be the limit at $e'$, and let $w_0 \in V(G'_0)$ be the vertex to which $e'$ contracts. For each $A_0 \in \inv \dtmor_0(w_0)$, let $G_{A_0} \subseteq G$ be the subgraph contracting to $A_0$. Then:
	\[\pushQED{\qed}
	2g(G_{A_0}) + r_{\dtmor_0}(A_0) = \sum_{A \in V(G_{A_0})} r_\dtmor(A). \qedhere\popQED
	\]
\end{prop}

If $g(G_{A_0}) = 0$ for all $A_0 \in \inv \dtmor_0(w_0)$, no cycles of $G$ are contracted, which for tree targets is equivalent to $g(H(\dtmor)) = g(H(\dtmor_0))$. In this case, Proposition~\ref{prop-rphi-under-contraction} ensures that $\dtmor_0$ is again a \DTmorp.

\subsection{Cone of sources} \label{sub-cone-of-sources} 
Given a \DTmor $\dtmor \colon G \to G'$, let $C_\dtmor$ denote the parameter cone of sources of all tropical morphisms with model $\dtmor$. For a vertex $A \in V(G)$, its \emph{non-dangling edge neighbourhood} $\Neighnd A \subseteq \Neigh A$ is the set of edges incident to $A$ that are not dangling, and its \emph{non-dangling valency} is $\ndval A = |\Neighnd A|$. To describe $C_\dtmor$ as a rational polyhedral cone, we realize the combinatorial type of $G$ via the following graph:

\begin{cons} \label{construction-Hphi}
	The graph $H(\dtmor)$ has vertex set $\{A \in V(G) \mid \ndval A \ge 3\}$; its edges are the maximal paths of $G$ whose endpoints lie in $V(H(\dtmor))$ and whose internal vertices have non-dangling valency $\ndval A = 2$.
\end{cons}

For a length function $\ell_{G'} \in C_{G'}$, the induced length function on $H(\dtmor)$ is
\[ y(h) = \sum_{e \in h} \frac{\ell_{G'}(\dtmor(e))}{m_\dtmor(e)}. \]
This defines a linear map $A_\dtmor \colon \RR^{E(G')} \to \RR^{E(H(\dtmor))}$, called the \emph{edge-length map}. In standard coordinates, $A_\dtmor$ is represented by the \emph{edge-length matrix}, with rows indexed by $E(H(\dtmor))$ and columns by $E(G')$. Its entries are:
\begin{align} \label{eq-aht} \tag{aht}
a_{ht} = \sum_{\substack{e \in h \\ \dtmor(e) = t}} \frac{1}{m_\dtmor(e)},
\end{align}
so that
\begin{align} \label{eq-yh} \tag{yh}
y(h) = \sum_{t \in E(G')} a_{ht} \ell_{G'}(t).
\end{align}
By construction, $(H(\dtmor), A_\dtmor(\ell_{G'}))$ is the deletion of dangling trees of the source of $(\dtmor, \ell_{G'})$. Since $H(\dtmor)$ has minimum valency at least 3, Lemma~\ref{lemma-combinatorial-type} identifies $H(\dtmor)$ with the combinatorial type of $G$. Thus, $C_\dtmor$ is parametrized by the rational polyhedral cone $A_\dtmor(\RRo^{E(G')})$, with which we identify it.

\begin{prop}[dimension formula] \label{proposition-dim-formula} 
	Let $\dtmor \colon G \to G'$ be a \DTmorp. Then
	\[|E(G')| + \sum_{v\in V(G')} (\ch v + \vale v - 3) = 2g(G) - g(G')(2d-3) + 2d - 5,\]
	where $\ch v = \sum_{A \in \inv \dtmor(v)} r_\dtmor(A)$. \qed
\end{prop}

The quantity $\ch v$ detects ramification above $v$. In particular, if $\ch v = 0$ at a divalent vertex $v$ with incident edges $t_1, t_2 \in E(G')$, then $a_{ht_1} = a_{ht_2}$ for every $h \in E(H(\dtmor))$, so the columns of $A_\dtmor$ incident to $v$ coincide, preventing full rank. Consequently, combinatorial types satisfy:

\begin{lm} \label{lm:combi-type-greater-than3}
	If $\dtmor$ is a combinatorial type, then $\ch v + \vale v \ge 3$ for all $v \in V(G')$. \qed
\end{lm}

We say that $\dtmor$ (or $C_\dtmor$) is \emph{full-rank} if the edge-length matrix $A_\dtmor$ is injective. In this case, $\dim C_\dtmor = |E(G')|$, and a metric graph $\mG \in C_\dtmor$ uniquely determines the realization $(\dtmor, \inv A_\dtmor(\ell_H))$.

\subsection{Maps to trees} \label{subsection-maps-to-trees}
When the target is a tree $T$, so that $g(T) = 0$, the dimension formula (Proposition~\ref{proposition-dim-formula}) simplifies to $2g + 2d - 5$. While full-rankness is a global condition on $A_\dtmor$, it imposes necessary local conditions that can be tracked during edge contractions and deformations.

\begin{de} \label{def-auxiliary-conditions} Let $\dtmor \colon G \to T$ be a \DTmor to a tree.
	\begin{itemize}
		\item A vertex $v \in V(T)$ is \emph{change-minimal} if $\ch v + \vale v = 3$. 
		\item A vertex $A \in V(G)$ satisfies the \emph{no-return} condition if $|\dtmor(\Neighnd A)| \ge 2$.
		\item An edge $h = \langle A_0, e_1, \dots, e_\mu, A_\mu \rangle \in E(H(\dtmor))$ satisfies the \emph{pass-once} condition if $\dtmor$ restricted to $\{e_i \in h \mid \dtmor(e_i) \text{ is not incident to a leaf}\}$ is injective.
	\end{itemize}
\end{de} 

\begin{de} \label{def-conditions} Let $\dtmor \colon G \to T$ be a \DTmorp.
	\begin{itemize}
		\item $\dtmor$ is \emph{change-minimal} if every vertex of $T$ is change-minimal.
		\item $\dtmor$ satisfies the \emph{dangling-no-glue} condition if every dangling element has multiplicity 1, so that $m(x) = 1$ for all dangling $x \in G$.
		\item $\dtmor$ satisfies the \emph{no-return} condition if every non-dangling vertex $A \in V(G)$ with $\dtmor(A)$ not a leaf satisfies no-return.
		\item $\dtmor$ satisfies the \emph{pass-once} condition if all edges of $H(\dtmor)$ satisfy pass-once.
	\end{itemize}
\end{de} 

If $\dtmor_0$ is the limit of $\dtmor$ under contracting an edge $t \in E(T)$ with ends $u, v$ such that $g(H(\dtmor)) = g(H(\dtmor_0))$, then by Lemma~\ref{lm:combi-type-greater-than3}:
\[ \ch w_0 + \vale w_0 - 3 = (\ch u + \vale u - 3) + (\ch v + \vale v - 3) + 1 \ge 1. \]
Hence, change-minimal \DTmors cannot arise as limits of one another.

\begin{lm}[{\cite[Lemmas~\ref{I-lemma-dangling-no-glue}, \ref{I-cor-no-return}, \ref{I-lemma-pass-once}]{dv20}}] \label{lm:properties}
	If $\dtmor$ is full-rank and change-minimal, then it satisfies the dangling-no-glue, no-return, and pass-once conditions. \qed
\end{lm} 

\begin{de}[top-dimensional cones] \label{definition-top-dimensional}
	A \DTmor $\dtmor \colon G \to T$ is \emph{top-dimensional} if its cone attains the maximum dimension given by Proposition~\ref{proposition-dim-formula}: $\dim C_\dtmor = 2g(G) + 2\deg \dtmor - 5$.
\end{de}

\begin{de} \label{def-possibly-full-rank}
	Motivated by Lemma~\ref{lm:properties}, a \DTmor $\dtmor$ is \emph{possibly full-rank} if it satisfies the dangling-no-glue and no-return conditions.
\end{de}

\begin{lm}[{\cite[Lemma~\ref{I-lem-rphi-nd}]{dv20}}, non-dangling $r_\dtmor$ formula] \label{lem-rphi-nd} 
	If $\dtmor$ satisfies dangling-no-glue, then for every $A \in V(G)$:
	\[\pushQED{\qed}
	r_\varphi(A) = \nddeg A - 2 + 2m(A) - \sum_{e \in \Neighnd A} m(e).\qedhere\popQED
	\]
\end{lm}

\begin{prop}[{\cite[Proposition~\ref{I-prop-local}]{dv20}}, local properties] \label{prop-local} 
	Let $\dtmor$ be a change-minimal \DTmor satisfying dangling-no-glue, and let $A \in V(G)$ satisfy $\nddeg A \le 3$. Then exactly one of the following holds:  
	\begin{itemize}
		\item[(r0)] If $r_\dtmor(A) = 0$, then $\dtmor|_{\Neighnd A}$ is injective, so $\nddeg A \le \vale \dtmor(A)$.
		\item[(r1)] If $r_\dtmor(A) = 1$, then $\dtmor(A)$ is divalent, $\nddeg A \in \{2, 3\}$, and $\vale A = 3$. If $\Neigh A = \{e, e', e''\}$ with $m(e) \ge m(e') \ge m(e'')$, then $m(e) = m(A) = m(e') + m(e'')$.
		\item[(r2)] If $r_\dtmor(A) = 2$, then $\dtmor(A)$ is a leaf, $\nddeg A = 2$, and $\vale A = 2$. If $\Neigh A = \{e, e'\}$, then $m(e) = m(e') = 1$.
		\item[(d)] If $G_{\mathrm{d}} \subseteq G$ is a connected subgraph whose edges are all dangling, then $\dtmor|_{G_{\mathrm{d}}}$ is injective. \qed
	\end{itemize}
\end{prop}

\begin{lm}[inheritance at limits] \label{lm:properties-inherit}
	Full-rankness, dangling-no-glue, no-return, and the pass-once condition are inherited by limits $\dtmor_0$ satisfying $g(H(\dtmor)) = g(H(\dtmor_0))$ {\cite[Lemmas~\ref{I-lemma-limit-dangling-no-glue}, \ref{I-lemma-limit-full-rank}, \ref{I-lemma-limit-no-return}, \ref{I-lemma-limit-pass-once}]{dv20}}. \qed
\end{lm} 

\begin{de}[labelling] \label{def-labelling}
	For $q=1,2$, let $\Aq \dtmor q \colon \Aq G q \to \Aq T q$ be \DTmors with edge labellings $\Aq {\lambda_T} q \colon E(\Aq T q) \to [|E(\Aq T q)|]$ and $\Aq {\lambda_H} q \colon E(H(\Aq \dtmor q)) \to [|E(H(\Aq \dtmor q))|]$. Let $\Aq {\dtmor_0} q$ be the limit at $t_k = \inv{(\Aq {\lambda_T} q)}(k)$, with induced labellings $\Aq {\lambda_{T,0}} q$ and $\Aq {\lambda_{H,0}} q$. If $|E(H(\Aq {\dtmor_0} q))| \ge |E(H(\Aq \dtmor q))| - 1$, the labellings are \emph{compatible at $t_k$} if there is an isomorphism $(\gamma, \tau) \colon \Aq \dtmor 1 _0 \to \Aq \dtmor 2 _0$ such that $\Aq {\lambda_{T,0}} 2 = \Aq {\lambda_{T,0}} 1 \circ \inv \tau$ and $\Aq {\lambda_{H,0}} 2 = \Aq {\lambda_{H,0}} 1 \circ \inv \wgamma$, where $\wgamma \colon H(\Aq {\dtmor_0} 1) \to H(\Aq {\dtmor_0} 2)$ is induced by $\gamma$.
\end{de}

\begin{lm}[{\cite[Lemma~\ref{I-lemma-limit-matrix-change}]{dv20}}, limit matrix] \label{lm:limit-matrix-change} 
	Let $\Aq \dtmor 1, \Aq \dtmor 2$ be \DTmors with labellings compatible at $t_k$, and edge-length matrices $(a_{ij}), (b_{ij})$ respectively. Then $a_{ij} = b_{ij}$ for all $j \ne k$. \qed
\end{lm}

\subsection{Graphical representation}
Graphically, we represent a \DTmor $\dtmor \colon G \to T$ following the graphical representation of \cite{dv20}: $G$ is viewed as the result of taking $\deg \dtmor$-many copies of $T$ and identifying vertices and edges. The index map $m_\dtmor(e)$ records the number of copies glued along $e$, with identifications indicated by dashed lines (see Example~\ref{ex-genus2}).

\subsection{Forbidden local parts} \label{subsection-forbidden-local-parts}
We now analyze local configurations that force the edge-length matrix $A_\dtmor$ to drop rank. This yields Proposition~\ref{proposition-at-most-two-weights}, which controls the variation of edge indices along paths in $H(\dtmor)$ and will later be used to compute lattice multiplicities for the cones $C_\dtmor$. Here $\F d g$ denotes the family of full-rank change-minimal degree-$d$ \DTmors with genus-$g$ source (defined formally in Section~\ref{sec-space}).

\begin{prop} \label{proposition-at-most-two-weights}
	Let $\dtmor$ arise from a sequence of cycle-preserving limits starting from an element of $\F d g$, and let $h = \langle A_0, e_1, A_1, \dots, A_{\nu - 1}, e_\nu, A_\nu \rangle \in E(H(\dtmor))$.
	\begin{itemize}
		\item If $h$ contains an edge mapping to an edge incident to a leaf of $T$, then $m(e_i) = 1$ for all $1 \le i \le \nu$, and either $\dtmor(A_0)$ or $\dtmor(A_\nu)$ is incident to an edge leading to a leaf.
		\item Otherwise, the edge indices along $h$ take at most two distinct values, which differ by at most~$1$ and vary monotonically: there exist integers $k \ge 1$ and $1 \le \mu \le \nu$ such that, up to reversing the orientation of $h$, $m(e_i) = k$ for $1 \le i \le \mu$ and $m(e_j) = k+1$ for $\mu < j \le \nu$ (with $\mu = \nu$ when the index is constant).
	\end{itemize}
\end{prop}

Before going into the proof of Proposition~\ref{proposition-at-most-two-weights}, we catalogue the forbidden local parts grouped by common limits, retaining the case notation from \cite[Section~\ref{I-subsection-case-work}]{dv20}. In each case, let $\bbv_i$ denote the column of $A_\dtmor$ corresponding to $t_i \in E(T)$.

\begin{itemize}[leftmargin=*] 
	\item Case \{v2-r2-nd2-ka-1\}: Let $\Aq \dtmor 1$ and $\Aq \dtmor 2$ have local parts as shown below, with limit $\dtmor_0$ obtained by contracting $t_1$.
	\vspace{0.5em}
	
	\noindent\begin{minipage}[t]{.21\textwidth}
		\centering
		\begin{overpic}[scale=0.7]{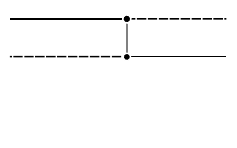} 
			\put (49,58) {\scalebox{1}{$A$}}
			\put (33,31) {\scalebox{1}{$t_2$}}
			\put (67,31) {\scalebox{1}{$t_3$}}
			\put (48,31) {\scalebox{1}{$w_0$}}	
			\put (-3,53) {\scalebox{1}{$e$}}	
			\put (97,36) {\scalebox{1}{$e'$}}	 
		\end{overpic}
		
		$\dtmor_0$
	\end{minipage}\hspace{0.8em}
	\begin{minipage}[t]{.21\textwidth} 
		\centering
		\begin{overpic}[scale=0.7]{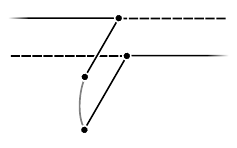}     
			\put (47,15) {\scalebox{1}{$t_1$}}
			\put (39,4) {\scalebox{1}{$u$}} 
			\put (53,31) {\scalebox{1}{$v$}} 
		\end{overpic} 
		
		$\Aq \dtmor 1$
	\end{minipage}\hspace{0.3em}
	\begin{minipage}[t]{.21\textwidth} 
		\centering
		\begin{overpic}[scale=0.7]{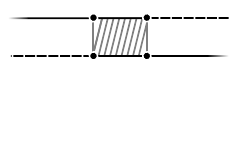}       
			\put (48,30) {\scalebox{1}{$t_1$}}
			\put (37,30) {\scalebox{1}{$u$}} 
			\put (59,30) {\scalebox{1}{$v$}} 
		\end{overpic}
		
		$\Aq \dtmor 2$
	\end{minipage}
	\begin{minipage}[b]{.29\textwidth} 
		\centering
		\begin{align*}
		\begin{matrix}
		\bbv_2 	& \bbv_3 	& \Aq {\bbv_1} 1 	& \Aq {\bbv_1} 2 \\
		1 		& 1			& 2					& \frac{1}{2}  \\
		a_i		& a_i		& 0				  	& a_i			
		\end{matrix}
		\end{align*}
	\end{minipage}
	\vspace{0.5em}
	
	Since columns $\bbv_2$ and $\bbv_3$ are identical, $\Aq \dtmor q$ is not full-rank.
	
	\item Case \{v2-r2-nd2-ka-2\}: Let $\Aq \dtmor 1$ and $\Aq \dtmor 2$ have local parts as shown below, with limit $\dtmor_0$ obtained by contracting $t_1$.
	\vspace{0.5em}
	
	\noindent\begin{minipage}[t]{.21\textwidth}
		\centering
		\begin{overpic}[scale=0.7]{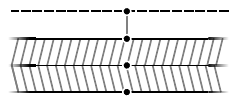} 
			\put (50,41) {\scalebox{1}{$A$}}
			\put (33,-4) {\scalebox{1}{$t_2$}}
			\put (68,-4) {\scalebox{1}{$t_3$}}
			\put (48,-4) {\scalebox{1}{$w_0$}}	
			\put (-3,13) {\scalebox{1}{$e_1$}}	
			\put (97,13) {\scalebox{1}{$e_2$}}	 
		\end{overpic}
		
		\vspace{1em}
		$\dtmor_0$
	\end{minipage}\hspace{0.3em}
	\begin{minipage}[t]{.21\textwidth} 
		\centering
		\begin{overpic}[scale=0.7]{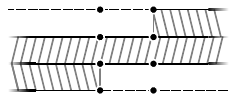}     
			\put (51,-4) {\scalebox{1}{$t_1$}}
			\put (40,-4) {\scalebox{1}{$u$}} 
			\put (62,-4) {\scalebox{1}{$v$}} 
		\end{overpic} 
		
		\vspace{1em}
		$\Aq \dtmor 1$
	\end{minipage}\hspace{0.3em}
	\begin{minipage}[t]{.21\textwidth}  
		\centering
		\begin{overpic}[scale=0.7]{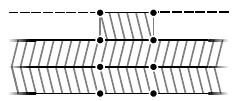}       
			\put (51,-4) {\scalebox{1}{$t_1$}}
			\put (40,-4) {\scalebox{1}{$u$}} 
			\put (62,-4) {\scalebox{1}{$v$}} 
		\end{overpic}
		
		\vspace{1em}
		$\Aq \dtmor 2$
	\end{minipage}
	\begin{minipage}[b]{.29\textwidth} 
		\centering
		\begin{align*}
		\begin{matrix}
		\bbv_2 			& \bbv_3 			& \Aq {\bbv_1} 1 	& \Aq {\bbv_1} 2 \\
		\frac{1}{k} 	& \frac{1}{k}		& \frac{1}{k-1}		& \frac{1}{k+1}  \\
		a_i				& a_i				& a_i				& a_i			
		\end{matrix}
		\end{align*}
	\end{minipage}
	\vspace{0.5em}
	
	Again $\bbv_2 = \bbv_3$, so $\Aq \dtmor q$ is not full-rank.
	
	\item Case \{v2-r2-nd2-kb\}: Let $\Aq \dtmor 1$ have the local part shown below, with limit $\dtmor_0$ obtained by contracting $t_1$.
	\vspace{0.5em}
	
	\noindent\begin{minipage}[t]{.30\textwidth}
		\centering
		\begin{overpic}[scale=1]{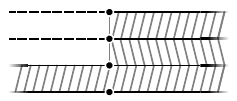} 
			\put (43,41) {\scalebox{1}{$A$}}
			\put (28,-4) {\scalebox{1}{$t_2$}}
			\put (61,-4) {\scalebox{1}{$t_3$}}
			\put (43,-4) {\scalebox{1}{$w_0$}}	
			\put (-3,7) {\scalebox{1}{$e_1$}}	
			\put (97,18) {\scalebox{1}{$e_2$}}	 
		\end{overpic}
		
		\vspace{1em}
		$\dtmor_0$
	\end{minipage}\hspace{0.8em}
	\begin{minipage}[t]{.30\textwidth} 
		\centering
		\begin{overpic}[scale=1]{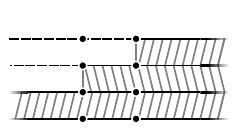}     
			\put (45,-4) {\scalebox{1}{$t_1$}}
			\put (32,-4) {\scalebox{1}{$u$}} 
			\put (56,-4) {\scalebox{1}{$v$}} 
		\end{overpic} 
		
		\vspace{1em}
		$\Aq \dtmor 1$
	\end{minipage}
	\begin{minipage}[b]{.29\textwidth} 
		\centering
		\begin{align*}
		\begin{matrix}
		\bbv_2 			& \bbv_3 				& \Aq {\bbv_1} 1 	\\
		\frac{1}{k} 	& \frac{1}{k+2}			& \frac{1}{k+1}		\\
		a_i				& a_i					& a_i						
		\end{matrix}
		\end{align*}
	\end{minipage}
	\vspace{0.5em}
	
	Since $2(k+1)\Aq {\bbv_1} 1 = k\bbv_2 + (k+2)\bbv_3$, these columns are linearly dependent, so $\Aq \dtmor 1$ is not full-rank.

	\item Case \{v2-r2-nd2-M\}: If two non-dangling edges $e_2, e_3$ incident to $A$ have equal index $m(e_2) = m(e_3)$ while neither leads to a leaf, their corresponding columns $\bbv_2, \bbv_3$ in $A_\dtmor$ are identical across all rows, forcing $A_\dtmor$ to drop rank (see \cite[Section~\ref{I-subsection-case-work}]{dv20}).
\end{itemize} 

\begin{proof}[Proof of Proposition~\ref{proposition-at-most-two-weights}]
	Edge contractions in $T$ preserve the multiplicities $m(e)$ of all uncontracted edges, and every path $h \in E(H(\dtmor))$ contracts to a path in the limit whose sequence of edge multiplicities is a subsequence of the original one. Therefore, if the asserted condition failed for a limit morphism, it would already fail for any change-minimal ancestor $\dtmor \in \F d g$ from which it specializes. It thus suffices to prove the proposition for full-rank change-minimal $\dtmor$.
	
	Let $e_i, e_{i+1}$ be adjacent edges in $h$ meeting at the internal divalent vertex $A_i$. By Cases~(r0), (r1), and (r2) of Proposition~\ref{prop-local}, the edge indices on adjacent edges differ by at most~1: $|m(e_i) - m(e_{i+1})| \le 1$.
	
	\medskip\noindent\textbf{Case 1: $h$ passes above a leaf.}
	Suppose $h$ contains an edge $e_*$ mapped to an edge $t_* \in E(T)$ incident to a leaf $v \in V(T)$, with other endpoint $u$. By dangling-no-glue, $m(e_*) = 1$. If $\vale u = 2$, then $\dtmor(A_0)$ or $\dtmor(A_\nu)$ leads to a leaf, and the claim follows by \cite[Lemma~\ref{I-lemma-loop-bridge}]{dv20}. If $\vale u = 3$, suppose for contradiction that some edge in $h$ has index strictly greater than~1. Since $m(e_*) = 1$ and adjacent indices differ by at most~1, there exists a pair of adjacent edges $e_2, e_3 \in h$ with $m(e_2) \ne m(e_3)$, say with common vertex $A$. By Proposition~\ref{prop-local}, $\vale \dtmor(A) = 2$, so $t_* = \dtmor(e_*)$, $t_2 = \dtmor(e_2)$, and $t_3 = \dtmor(e_3)$ are three distinct edges of $T$. 
	
	By change-minimality and Proposition~\ref{prop-local}, the leaf column $\bbv_*$ has a single non-zero entry (corresponding to the row of $h$), while the columns $\bbv_2$ and $\bbv_3$ coincide outside the row of $h$. Hence the difference $\bbv_2 - \bbv_3$ is supported solely on the row of $h$, with non-zero value $1/m(e_2) - 1/m(e_3) \ne 0$. This makes $\bbv_2 - \bbv_3$ a non-zero scalar multiple of $\bbv_*$, contradicting that $A_\dtmor$ has full column rank. Furthermore, if neither $\dtmor(A_0)$ nor $\dtmor(A_\nu)$ leads to a leaf, $\dtmor$ would contain the configuration of Case~\{v2-r2-nd2-M\}, where columns $\bbv_2$ and $\bbv_3$ are equal, again contradicting full rank. Thus $m(e_i) = 1$ for all $1 \le i \le \nu$.
	
	\medskip\noindent\textbf{Case 2: $h$ does not pass above a leaf.}
	Suppose for contradiction that the sequence of indices along $h$ is not monotonic with values in $\{k, k+1\}$. Since adjacent indices differ by at most~1, this failure implies the existence of three indices along $h$ exhibiting either:
	\begin{enumerate}[(i)]
		\item a \emph{direction change} (peak or valley): indices $(k, k+1, k)$ or $(k+1, k, k+1)$; or
		\item a \emph{double-step}: indices $(k, k+1, k+2)$ (or $(k+2, k+1, k)$).
	\end{enumerate}
	By contracting all intermediate edges in $T$ along which the index remains constant, we obtain a limit morphism $\dtmor_0$ containing three consecutive edges $e_\alpha, e_{\alpha+1}, e_{\alpha+2} \in h$ realizing one of these profiles.
	
	Contracting the target edge of the middle edge $e_{\alpha+1}$ and inspecting the regrowing possibilities via the trivalent deformation cases (see Appendix~\ref{appendix-trivalent-deformation} and \cite[Section~\ref{I-subsection-case-work}]{dv20}) yields:
	\begin{itemize}
		\item In the peak/valley case $(k, k+1, k)$, contracting $e_{\alpha+1}$ leaves the outer edges with equal indices $k$ and $k$. Regrowing the contracted edge produces a full-rank change-minimal morphism containing the local part of \textbf{Case~\{v2-r2-nd2-ka-1\}} (if $k=1$) or \textbf{Case~\{v2-r2-nd2-ka-2\}} (if $k \ge 2$). In both cases, columns $\bbv_2$ and $\bbv_3$ are identical ($\bbv_2 = \bbv_3$), so $A_\dtmor$ drops rank.
		\item In the double-step case $(k, k+1, k+2)$, contracting $e_{\alpha+1}$ produces the local part of \textbf{Case~\{v2-r2-nd2-kb\}}, where $2(k+1)\Aq{\bbv_1}1 = k\bbv_2 + (k+2)\bbv_3$, so columns $\bbv_2, \bbv_3, \Aq{\bbv_1}1$ are linearly dependent and $A_\dtmor$ drops rank.
	\end{itemize}
	In every case, the resulting morphism is not full-rank, contradicting that it arises as a limit of an element of $\F d g$. Hence edge indices along $h$ take at most two distinct values, which differ by at most~1 and vary monotonically.
\end{proof}


\section{A space of full-rank tropical morphisms} \label{section-a-space-of-tropical-morphisms}
Throughout this section, let $g \ge 2$ and $d \le \lceil g/2\rceil + 1$ be integers. We construct the parameter space $\TM d g$ of full-rank tropical morphisms of degree~$d$ from genus-$g$ metric graphs to metric trees, equipped with the canonical projection $\Pi \colon \TM d g \to \MTrop g$ sending $(\mH, \tmor) \mapsto \mH$. We characterise the symmetries of $\TM d g$, establish the finiteness of the fibers of $\Pi$, and determine the fiber $\inv \Pi(\mH)$ over generic caterpillars of loops, where it consists of exactly $C_{g/2}$ points connected through codimension-1 cells.

\subsection{The space} \label{sec-space}
We assemble the space of tropical morphisms by glueing together the cones of full-rank \DTmors of degree $d$ whose source has genus $g$. As established in Subsection~\ref{subsection-maps-to-trees}, change-minimal \DTmors are the natural candidates for the top-dimensional cells, since they cannot be obtained by edge contractions. We denote by $\F d g$ the set of full-rank change-minimal degree-$d$ \DTmors $\dtmor$ with $g(H(\dtmor)) = g$; by Lemma~\ref{lm:properties}, every element of $\F d g$ automatically satisfies the dangling-no-glue, no-return, and pass-once conditions. When $g$ is even and $d = g/2 + 1$, we write $\FD g$ and call its elements \emph{full-dimensional} \DTmorsp{}.

For $\dtmor \in \F d g$, the open cone $C_\dtmor$ parameterises tropical morphisms because $\dtmor$ has full rank. By Lemma~\ref{lm:limit-matrix-change}, if $\dtmor_0$ is a limit of $\dtmor$, then $C_{\dtmor_0}$ is a relatively open, codimension-1 face of the closure of $C_\dtmor$ in $\RR^{E(H(\dtmor))}$. To enable glueing across boundaries, we add these limit faces to $C_\dtmor$.

We adopt the following convention for length vectors $\ell_T \in \RRgo^{E(T)}$ that can have zero lengths: contracting all edges of $T$ with $\ell_T(e) = 0$ produces a limit morphism $\dtmor_0 \colon G_0 \to T_0$, whose isomorphism class is independent of the contraction order. The pair $(\dtmor, \ell_T)$ then designates the tropical morphism $(\dtmor_0, \ell_T|_{T_0})$. We define the \emph{completed cone} $\overline{C}_\dtmor$ as the image under $A_\dtmor$ of all points $\ell_T \in \RRgo^{E(T)}$ such that the source of $(\dtmor, \ell_T)$ has genus $g$. Because $\dtmor$ is full-rank, the matrix $A_\dtmor$ is injective, so each point $\ell_H \in \overline{C}_\dtmor$ corresponds to a unique length vector $\ell_T = \inv A_\dtmor(\ell_H) \in \RRgo^{E(T)}$. This vector determines a unique pair $(\mH, \tmor)$ consisting of the metric graph $\mH$ obtained from $(H, \ell_H)$ by contracting zero-length edges and the tropical morphism $\tmor = (\dtmor, \ell_T)$ from a tropical modification $\mG$ of $\mH$ to the metric tree $\mT = (T, \ell_T)$.

\begin{de}[space of full-rank tropical morphisms] \label{definition-TMSpace} 
	Let $d \le \lceil g/2 \rceil + 1$. The \emph{space of full-rank tropical morphisms} is the topological space
	\[
	\TM d g = \left( \bigsqcup_{\dtmor \in \F d g} \overline{C}_\dtmor \right) \Big/ \cong,
	\]
	equipped with the quotient topology, where $\cong$ identifies points corresponding to isomorphic tropical morphisms. The canonical projection is $\Pi \colon \TM d g \to \MTrop g$, given by $(\mH, \tmor) \mapsto \mH$.
\end{de}

Just as in the construction of $\MTrop g$, the geometry of $\TM d g$ involves self-identifications: a single cone $\overline{C}_\dtmor$ may have distinct points identified under $\cong$, glueing the cone onto itself along symmetries in the group $\Aut_\dtmor H(\dtmor)$ characterised in Subsection~\ref{subsection-symmetries-in-tmdg}. Consequently, $\TM d g$ is a generalised cone complex in the sense of~\cite[Remark~2.6.1]{acp15}.

\begin{ex}[genus 2] \label{ex-genus2} 
	We determine the full-dimensional cones $\FD 2$ in $\TM 2 2$. Here $|E(T)| = 2d + 2g - 5 = 3$, which coincides with the dimension $3g - 3 = 3$ of $\MTrop 2$. There are two trees with three edges: the star $\Aq T 1$ whose single internal vertex is adjacent to three leaves, and the path $\Aq T 2$ of length~3.
	
	Change-minimality determines the fiber above each leaf and each leaf edge. Over $\Aq T 1$, this determines the morphism $\Aq \dtmor 1$ completely, whose source combinatorial type $\Aq H 1$ is the theta graph. Over $\Aq T 2$, the endpoints of the middle edge must have trivalent preimages in $\Aq H 2$, so the middle edge in $\Aq H 2$ has index $m(e) = 2$; the resulting source type $\Aq H 2$ is the dumbbell graph, which is the caterpillar of loops $\HCL 2$. Computing the edge-length matrices gives diagonal matrices:
	
	\vspace{0.5em}
	\noindent\begin{minipage}{0.24\textwidth}
		\centering
		\begin{overpic}[scale=1.4]{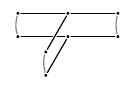}
		\end{overpic}  
	\end{minipage}
	\begin{minipage}{0.20\textwidth}
		\centering
		\begin{overpic}[scale=1]{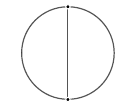}
		\end{overpic}
	\end{minipage}
	\begin{minipage}{0.25\textwidth}
		\centering 
		\begin{overpic}[scale=1.5]{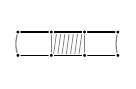}
		\end{overpic}  
	\end{minipage}
	\begin{minipage}{0.27\textwidth}
		\centering
		\begin{overpic}[scale=0.8]{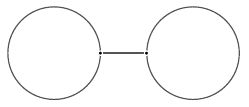}
		\end{overpic}
	\end{minipage}
	
	\noindent\begin{minipage}[t]{0.24\textwidth}
		\centering
		$\Aq \dtmor 1$
	\end{minipage}
	\begin{minipage}[t]{0.20\textwidth}
		\centering
		$\Aq H 1$
	\end{minipage}
	\begin{minipage}[t]{0.25\textwidth}
		\centering 
		$\Aq \dtmor 2$
	\end{minipage}
	\begin{minipage}[t]{0.27\textwidth}
		\centering
		$\Aq H 2$
	\end{minipage} 
	
	\vspace{0.5em}
	\noindent
	\begin{minipage}[t]{0.44\textwidth}
		\centering
		$\Aq A 1 = \begin{pmatrix}
		2 & 0 & 0\\
		0 & 2 & 0\\
		0 & 0 & 2
		\end{pmatrix}$
	\end{minipage}
	\begin{minipage}[t]{0.49\textwidth}
		\centering
		$\Aq A 2 = \begin{pmatrix}
		2 & 0 & 0\\
		0 & 1/2 & 0\\
		0 & 0 & 2
		\end{pmatrix}$
	\end{minipage}
	\vspace{0.5em}
	
	Since both matrices are invertible, $\Aq \dtmor 1$ and $\Aq \dtmor 2$ belong to $\FD 2$. The set $\mathbb G_2$ of genus-2 trivalent combinatorial types consists of only the theta graph $\Aq H 1$ and the dumbbell graph $\Aq H 2$. Because the completed cones $\overline{C}_{\Aq \dtmor 1}$ and $\overline{C}_{\Aq \dtmor 2}$ project onto the maximal cones $\overline{C}_{\Aq H 1}$ and $\overline{C}_{\Aq H 2}$ of $\MTrop 2$, their union surjects onto $\MTrop 2$, so $\Pi \colon \TM 2 2 \to \MTrop 2$ is surjective.
\end{ex}

\subsection{Symmetries in \texorpdfstring{$\TM d g$}{TM(d,g)}} \label{subsection-symmetries-in-tmdg}
Let $\dtmor$ be a \DTmorp. Parallel to the description of symmetries in $\MTrop g$ from Subsection~\ref{subsection-symmetry-in-mtropg}, we characterise the points in $C_\dtmor$ that define isomorphic tropical morphisms. Recall from Definition~\ref{definition-tropical-iso} that two tropical morphisms $\Aq \tmor 1, \Aq \tmor 2$ are isomorphic if there exist isometries $\Psi, \Upsilon$ such that $\Aq \tmor 2 \circ \Psi = \Upsilon \circ \Aq \tmor 1$.

Because $\Psi$ and $\Upsilon$ are isometries, they preserve essential vertex sets (Lemma~\ref{lemma-isometry-to-isomorphism}), hence restrict to isomorphisms of essential models $(\gamma, \tau)$ commuting with metric lengths and preserving the index map.

Every point of $\TM d g$ belongs to the quotient of an open cone $C_\dtmor / \cong$, where $\dtmor$ is a full-rank \DTmor obtained from an element of $\FD g$ (or $\F d g$) through a finite sequence of edge contractions. For any metric graph $\mH$, the fiber decomposes as
\begin{equation} \label{eq-S} \tag{S}
|\inv \Pi(\mH)| = \sum_\dtmor |\inv \Pi(\mH) \cap (C_\dtmor / \cong)|,
\end{equation}
where the sum ranges over all such full-rank $\dtmor$ that realize $\mH$.

Because $\dtmor$ has full rank, the matrix $A_\dtmor$ is injective (and invertible when $d = g/2 + 1$), so the metric on $T$ is uniquely determined by the metric on $H(\dtmor)$. An automorphism of $\dtmor \colon G \to T$ is a pair of graph isomorphisms $(\gamma_G, \tau) \in \Aut G \times \Aut T$ such that $\dtmor \circ \gamma_G = \tau \circ \dtmor$ and $m_\dtmor = m_\dtmor \circ \inv \gamma_G$. Because $\gamma_G$ preserves non-dangling and essential vertices, it descends by dangling deletion and contraction of divalent vertices to an automorphism of the combinatorial core $H(\dtmor)$. We denote by
\[
\Aut_\dtmor H(\dtmor) \subseteq \Aut H(\dtmor)
\]
the subgroup of automorphisms of $H(\dtmor)$ that arise from automorphisms of $\dtmor$. The group $\Aut_\dtmor H(\dtmor)$ acts on $C_\dtmor$ by length pullback: $\gamma \cdot y = y \circ \inv \gamma$. The number of ways that a fixed $\dtmor$ can realize $\mH$ is then determined by the orbit structure under this action:

\begin{lm} \label{lemma-congruence-iff-autphi-orbit}
	Let $\Aq y 1, \Aq y 2 \in C_\dtmor$. Then $\Aq y 1 \cong \Aq y 2$ if and only if $\Aq y 1$ and $\Aq y 2$ lie in the same $\Aut_\dtmor H(\dtmor)$-orbit.
\end{lm}

\begin{proof}
	An isomorphism $(\Psi, \Upsilon) \colon \Aq \tmor 1 \to \Aq \tmor 2$ preserves essential vertex sets and induces an isomorphism of essential models $(\gamma, \tau) \colon \dtmor \to \dtmor$ preserving lengths and slopes, which yields an element $\gamma \in \Aut_\dtmor H(\dtmor)$ sending $\Aq y 1$ to $\Aq y 2$. Conversely, any $\gamma \in \Aut_\dtmor H(\dtmor)$ extends to an automorphism of $\dtmor$, and if $\Aq y 2 = \gamma \cdot \Aq y 1$, the length-preserving maps on edge segments yield isometries $\Psi$ and $\Upsilon$ witnessing $\Aq \tmor 1 \cong \Aq \tmor 2$.
\end{proof}

\begin{lm} \label{lemma-count-realizations}
	Let $y \in C_\dtmor$, and let $\mH = (H(\dtmor), y)$ be the corresponding metric graph. Then
	\[
	|\inv \Pi(\mH) \cap (C_\dtmor / \cong)| = \frac{|(\orb {\Aut H} y) \cap C_\dtmor|}{|\orb {\Aut_\dtmor H(\dtmor)} y|}.
	\]
	In particular, if the edge lengths in $y$ are general, then
	\[
	|\inv \Pi(\mH) \cap (C_\dtmor / \cong)| = \frac{|(\orb {\Aut H} y) \cap C_\dtmor|}{|\Aut_\dtmor H(\dtmor)|}.
	\]
\end{lm}

\begin{proof}
	By Lemma~\ref{lemma-isometry-to-isomorphism}, the points in $C_\dtmor$ that define metric graphs isometric to $\mH = (H, y)$ are precisely the points in $(\orb {\Aut H} y) \cap C_\dtmor$. By Lemma~\ref{lemma-congruence-iff-autphi-orbit}, two such points define isomorphic tropical morphisms if and only if they belong to the same $\Aut_\dtmor H(\dtmor)$-orbit. Because $\Aut_\dtmor H(\dtmor) \subseteq \Aut H$, its action partitions $(\orb {\Aut H} y) \cap C_\dtmor$ into orbits of cardinality $|\orb {\Aut_\dtmor H(\dtmor)} y|$, proving the first formula. When the edge lengths in $y$ are general, the stabilizer of $y$ in $\Aut H$ is trivial, so $|\orb {\Aut_\dtmor H(\dtmor)} y| = |\Aut_\dtmor H(\dtmor)|$.
\end{proof}

Because the set $\F d g$ (and in particular $\FD g$) of combinatorial types is finite, and all automorphism groups are finite, equation~\eqref{eq-S} immediately implies that the fiber $\inv \Pi(\mH)$ is finite for every $\mH \in \MTrop g$.

\subsection{Caterpillars of loops} \label{subsection-caterpillar-of-loops}
Throughout this subsection, let $g$ be even, say $g = 2g'$, and let $d = g'+1$. We prove Theorem~\ref{thm} for the class of metric graphs called \emph{caterpillars of loops}.

\begin{de}[caterpillar of loops] \label{def-caterpillar-loops} 
	The genus-$g$ \emph{caterpillar of loops}, denoted $\HCL g$, is the graph obtained from a path of length $g-1$,
	\[
	\langle A_1, h_1, B_2, h_2, \dots, B_{g-2}, h_{g-2}, B_{g-1}, h_{g-1}, A_g \rangle,
	\]
	by attaching a self-loop at each of the two end vertices $A_1$ and $A_g$, and attaching a lollipop at each interior vertex $B_i$ for $i = 2, \dots, g-1$; that is, introducing a vertex $A_i$ with a self-loop and joining $A_i$ to $B_i$ by a bridge. See Figure~\ref{fig-chain}.
\end{de}

\begin{figure}[htbp]
	\centering
	\begin{overpic}[scale=1.4]{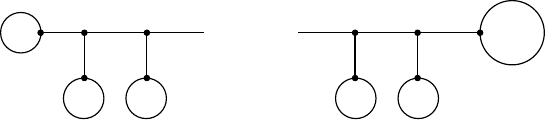} 
		\put (44,16) {\scalebox{1.3}{$\dots$}}
		\put (14.5,17) {\scalebox{1}{$B_2$}} 
		\put (26,17) {\scalebox{1}{$B_3$}} 
		\put (64,17) {\scalebox{1}{$B_{g-2}$}}
		\put (75.5,17) {\scalebox{1}{$B_{g-1}$}}
		\put (8,17) {\scalebox{1}{$A_1$}}
		\put (16.5,8) {\scalebox{1}{$A_2$}}
		\put (28.5,8) {\scalebox{1}{$A_3$}}
		\put (66,8) {\scalebox{1}{$A_{g-2}$}}
		\put (77.4,8) {\scalebox{1}{$A_{g-1}$}}
		\put (89,17) {\scalebox{1}{$A_{g}$}}
		\put (11,13) {\scalebox{1}{$h_1$}}
		\put (20.5,13) {\scalebox{1}{$h_2$}}
		\put (69,13) {\scalebox{1}{$h_{g-2}$}}
		\put (80,13) {\scalebox{1}{$h_{g-1}$}}
	\end{overpic}
	\caption{\label{fig-chain}A caterpillar of loops $\HCL g$.}
\end{figure}
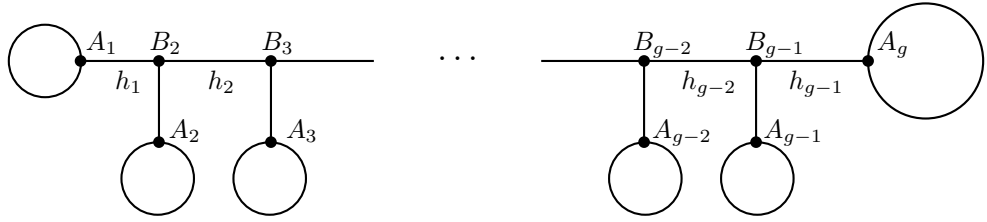 

We denote by $\FDCL g \subset \FD g$ the set of full-dimensional \DTmors $\dtmor$ such that $H(\dtmor) \cong \HCL g$. These morphisms enjoy optimal geometric properties: each cone $C_\dtmor$ is the positive orthant, the combinatorial types in $\FDCL g$ are in bijection with length-$g$ ballot sequences, and the cones glue along codimension-1 faces into a connected family in $\TM {g'+1} g$.

\begin{ex} \label{ex-chain-loops}
	For genus $g = 4$ and degree $d = 3$, Figure~\ref{fig:genus4-caterpillars} shows the two \DTmors in $\FDCL 4$. By Proposition~\ref{prop-caterpillar-ballot} and Lemma~\ref{lm:catalan-many}, these are the only two such morphisms, corresponding to the $C_2 = 2$ ballot sequences of length~4.
	
	\begin{figure}[htbp]
		\centering
		\begin{minipage}{0.48\textwidth}
			\centering
			\begin{overpic}[scale=1.3]{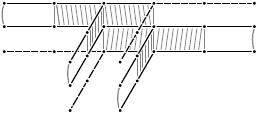}
			\end{overpic} 
			\[\dtmor_1\] 
		\end{minipage}
		\begin{minipage}{0.48\textwidth}
			\centering
			\begin{overpic}[scale=1.3]{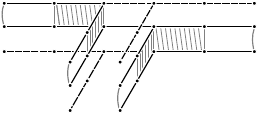}
			\end{overpic}  
			\[\dtmor_2\]
		\end{minipage}
		\caption{The two morphisms in $\FDCL 4$.}
		\label{fig:genus4-caterpillars}
	\end{figure}
\end{ex}

\begin{re} \label{re:dramatic-example}
	The divisorial gonality of a caterpillar of loops is~2 for every genus $g$, as witnessed by the rank-1 divisor $D = 2 A_1$. In~\cite{cha13}, Chan investigated tree gonality defined via harmonic morphisms from tropical modifications to trees without imposing the Riemann--Hurwitz condition. In our setting, every morphism in $\FDCL g$ has degree $g/2 + 1$. Caterpillars of loops thus illustrate that once the Riemann--Hurwitz condition is enforced, the gap between tree gonality and divisorial gonality can be arbitrarily large.
\end{re}

The clean combinatorial behaviour of caterpillars of loops originates from the local rigidity of lollipops: a lollipop in $H(\dtmor)$ maps to a length-2 path in $T$ ending at a leaf, and the fiber above this path is uniquely determined.

\begin{lm} \label{lm:bridge-and-loop}
	Let $\dtmor$ be change-minimal and full-rank. Let $A$ be a trivalent vertex of $H(\dtmor)$ incident to a bridge $h_b$ and a loop $h_l$. Then $h_b = \langle A, e_b, B \rangle$, $h_l = \langle A, e_1, C, e_2, A \rangle$, $\dtmor(C)$ is a leaf, $\dtmor(A)$ is divalent, $r_\dtmor(A) = 1$, and $m(e_b) = 2$. Furthermore, $e_b, A, e_1, e_2, C$ are the only non-dangling elements in the fibers of $\dtmor(e_b)$, $\dtmor(A)$, $\dtmor(e_1)$, and $\dtmor(C)$.
\end{lm}

\begin{proof}
	See~\cite[Lemma~\ref{I-lemma-loop-bridge}]{dv20}.
\end{proof}

\begin{lm} \label{lm:combinatorial-structure-caterpillar-of-loops} 
	Let $\dtmor \colon G \to T$ be in $\FDCL g$, and let $\tG$ be the deletion of the dangling elements of $G$. Then the restriction $\gamma^{CL}_g = \dtmor|_{\tG}$ depends only on $g$. See Figure~\ref{figure:gammaCLg}.
\end{lm}

\begin{proof}
	Fix an isomorphism $L \colon \HCL g \to H(\dtmor)$. By Lemma~\ref{lm:bridge-and-loop}, the image $\dtmor(\Neigh{L(A_i)})$ is a length-2 path with interior vertex $u_i$ and leaf endpoint $v_i$. These paths account for $2g$ distinct edges of $T$ whose fibers are completely determined. Since $\dtmor$ is change-minimal, $\chan u_i = 1$ and $\chan v_i = 2$. By the Riemann--Hurwitz formula (Lemma~\ref{lm:riemann-hurwitz}), these account for the total change $3g$ of $\dtmor$, so every vertex of $T$ other than $u_i$ and $v_i$ has change 0.
	
	Consider the unique path $P$ in $G$ with endpoints $L(B_2)$ and $L(B_{g-1})$. The path $P$ has length at least $g-3$, containing the internal vertices $L(B_i)$ for $i = 3, \dots, g-2$. The image $\dtmor(P)$ is disjoint from all $\dtmor(\Neigh{L(A_i)})$. Since $T$ has $3g-3$ edges, $\dtmor(P)$ contains at most $(3g-3) - 2g = g-3$ edges. Because the interior vertices of $P$ lie above vertices of change 0, they have $r_\dtmor = 0$. By Case~(r0) of Proposition~\ref{prop-local}, consecutive edges along $P$ map to distinct edges of $T$. Since $T$ is a tree, $\dtmor$ is injective on $P$.
	
	We conclude that $\dtmor(P)$ is a path of length $g-3$ in $T$, each path edge $h_i$ contains a single edge of $G$ for $i = 1, \dots, g-1$ (so the edge-length matrix $A_\dtmor$ is diagonal), and each path $\dtmor(\Neigh{L(A_i)})$ attaches to $\dtmor(P)$ to form the morphism $\gamma^{CL}_g \colon G^{CL}_g \to T^{CL}_g$ shown in Figure~\ref{figure:gammaCLg}.
\end{proof}

\begin{figure}[htbp]
	\centering
	\begin{overpic}[scale=1.6]{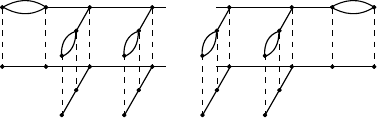} 
	\end{overpic}
	\caption{\label{figure:gammaCLg} The graph morphism $\gamma^{CL}_g \colon G^{CL}_g \to T^{CL}_g$.}
\end{figure} 

Let $\mH = (\HCL g, y)$ be a metric caterpillar of loops. Lemma~\ref{lm:combinatorial-structure-caterpillar-of-loops} implies that constructing a full-dimensional tropical morphism from a modification of $\mH$ to a metric tree reduces to choosing slopes $s_i$ on the path edges $h_i$. By the dangling-no-glue property and Case~(d) of Proposition~\ref{prop-local}, the dangling trees are then uniquely determined by the slope sequence. Moreover, the slopes satisfy a strict step condition:

\begin{prop}[caterpillar slopes and ballot sequences] \label{prop-caterpillar-ballot} \label{lm:one-trop-morph-for-each-ballot}
	Let $\mH = (\HCL g, y)$ be a metric caterpillar of loops.
	\begin{enumerate}[(1)]
		\item In every full-dimensional tropical morphism $\tmor \colon \mG \to \mT$ whose deletion of dangling trees is $\mH$, the slopes $s_i$ on the path edges satisfy $s_1 = s_{g-1} = 2$, $s_i \ge 1$, and $s_i - s_{i-1} = \pm 1$ for all $i = 2, \dots, g-1$.
		\item Conversely, every such slope sequence $(s_i)_{i=1}^{g-1}$ uniquely determines a tropical morphism $\tmor_s \colon \mG \to (T^{CL}_g, z)$ modifying $\mH$ of combinatorial type in $\FDCL g$.
	\end{enumerate}
\end{prop}

\begin{proof}
	For (1), let $\dtmor$ be the combinatorial type of $\tmor$. By Lemma~\ref{lm:bridge-and-loop}, the index at the bridge incident to $L(A_i)$ is~2. Because $r_\dtmor(L(B_i)) = 0$, the $r_\dtmor$ formula (Lemma~\ref{lm:formula-rphi}) gives
	\[
	0 = 2(m(L(B_i)) - 1) - \bigl((s_{i-1} - 1) + (s_i - 1) + (2 - 1)\bigr) = 2m(L(B_i)) - 1 - (s_{i-1} + s_i),
	\]
	so $s_{i-1} + s_i = 2m(L(B_i)) - 1$. Since $m(L(B_i)) \ge \max(s_{i-1}, s_i)$ by the balancing condition, it follows that $s_i - s_{i-1} = \pm 1$. The boundary conditions $s_1 = s_{g-1} = 2$ and positivity $s_i \ge 1$ follow from Lemma~\ref{lm:combinatorial-structure-caterpillar-of-loops}.
	
	For (2), constructing $\tmor_s$ consists of extending $\gamma^{CL}_g$ to a \DTmor $\dtmor_s$ by specifying the index map $m$ and attaching dangling trees (Lemma~\ref{lm:combinatorial-structure-caterpillar-of-loops}). The index map $m$ is determined on loops and their bridges by Lemma~\ref{lm:bridge-and-loop} and change-minimality; on path edges by $m(h_i) = s_i$; on dangling elements by dangling-no-glue ($m = 1$); and at vertices $B_i$ by $s_{i-1} + s_i = 2m(B_i) - 1$, forcing $m(B_i) = \max(s_i, s_{i-1})$. Balancing at $B_i$ is restored by attaching $m(B_i) - 2$ paths of length~2 mapping to the lollipop at $B_i$, oriented according to $s_i - s_{i-1} = \pm 1$. By Case~(d) of Proposition~\ref{prop-local}, this construction is unique, yielding a full-rank change-minimal \DTmor $\dtmor_s \in \FDCL g$.
\end{proof}

Every valid slope sequence $(s_i)_{i=1}^{g-1}$ defines a sequence $(b_i)_{i=1}^g$ by setting $b_1 = 1$, $b_g = -1$, and $b_i = s_i - s_{i-1}$ for $i = 2, \dots, g-1$. Each $b_i \in \{1, -1\}$, all partial sums $\sum_{j=1}^i b_j = s_i - 1$ are non-negative, and $\sum_{j=1}^g b_j = 0$ since $s_1 = s_{g-1} = 2$. These are precisely \emph{ballot sequences} of length $g$.

\begin{lm}[{\cite[Theorem~1.5.1]{sta15}}] \label{lm:catalan-many}
	The number of length-$g$ ballot sequences is the Catalan number $C_{g/2} = \frac{1}{g/2 + 1} \binom{g}{g/2}$. \qed
\end{lm}

\begin{prop} \label{prop-divisors-on-chain} 
	Let $g$ be even and let $\mH = (\HCL g, y)$ be a caterpillar of loops. If the edge lengths in $y$ are pairwise distinct, then the fiber $\inv \Pi(\mH)$ contains exactly $C_{g/2}$ points in $\TM d g$, and these points are pairwise connected by paths passing through codimension-1 cells in $\TM d g$.
\end{prop}

\begin{proof}
	For $q=1,2$, let $\Aq \tmor q \colon \Aq \mG q \to \Aq \mT q$ be a tropical morphism whose deletion of dangling trees is $\Aq \mH q$, with an isometry $\Aq L q \colon \mH \to \Aq \mH q$. Let $\Aq b q$ be the ballot sequence associated with $\Aq \tmor q$ by Proposition~\ref{prop-caterpillar-ballot}.
	
	If $\Aq b 1 = \Aq b 2$, then $\Aq \tmor 1$ and $\Aq \tmor 2$ are isomorphic by construction. Conversely, suppose $\Aq \tmor 1 \cong \Aq \tmor 2$, so there exist isometries $\Psi, \Upsilon$ with $\Aq \tmor 2 \circ \Psi = \Upsilon \circ \Aq \tmor 1$. The isometry $\Psi$ restricts to an isometry $\widetilde \Psi \colon \Aq \mH 1 \to \Aq \mH 2$, and $\inv {(\Aq L 2)} \circ \widetilde \Psi \circ \Aq L 1$ is an automorphism of $\mH$. Because the edge lengths in $y$ are pairwise distinct, $\Aut(\mH)$ is trivial, so $\Aq L 2 = \widetilde \Psi \circ \Aq L 1$. Writing $\ell(\cdot)$ for segment lengths, we find
	\[
	\Aq {s_i} 2 = \frac{\ell(\Aq L 2(h_i))}{\ell(\Aq \tmor 2 \circ \Aq L 2(h_i))} 
	= \frac{\ell(\widetilde \Psi \circ \Aq L 1(h_i))}{\ell(\Aq \tmor 2 \circ \widetilde \Psi \circ \Aq L 1(h_i))} 
	= \frac{\ell(\Aq L 1(h_i))}{\ell(\Upsilon \circ \Aq \tmor 1 \circ \Aq L 1(h_i))}
	= \Aq {s_i} 1.
	\]
	Hence $\Aq b 1 = \Aq b 2$. Lemma~\ref{lm:catalan-many} then shows that $\inv \Pi(\mH)$ consists of exactly $C_{g/2}$ points.
	
	To prove connectivity, suppose $\Aq b 2$ is obtained from $\Aq b 1$ by swapping adjacent entries $\Aq b 1_i = 1$ and $\Aq b 1_{i+1} = -1$ to $\Aq b 2_i = -1$ and $\Aq b 2_{i+1} = 1$, leaving all other entries unchanged. Then the slope on $h_i$ shifts from $s_{i-1}+1$ to $s_{i-1}-1$, while all other edges and indices agree. Contracting $h_i$ to length~0 in both cones yields isomorphic limit morphisms in their common boundary face. Since any ballot sequence can be transformed into any other by a finite sequence of such adjacent swaps, the completed cones $\overline{C}_\dtmor$ are connected through codimension-1 cells in $\TM d g$.
\end{proof}

\begin{re} \label{re:dual-graph-of-EH87}
	In~\cite{eh87}, Eisenbud and Harris established the Brill--Noether theorem by analysing a 1-parameter family of curves specialising to a nodal curve $C_\infty$ of genus $g$. The dual graph $G_\infty$ of $C_\infty$ is shown in Figure~\ref{fig:dual-graph-eh}:
	
	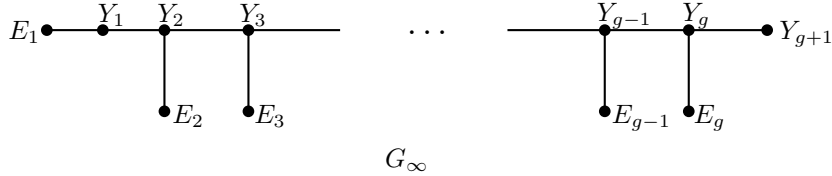
\begin{figure}[htbp]
		\centering
		\begin{overpic}[scale=1.4]{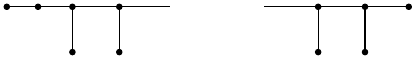}  	
			\put (-3.5,11.3) {\scalebox{1}{$E_1$}}
			\put (8.5,13.6) {\scalebox{1}{$Y_1$}}
			\put (16.5,13.6) {\scalebox{1}{$Y_2$}}
			\put (18.5,0) {\scalebox{1}{$E_2$}}
			\put (27.5,13.6) {\scalebox{1}{$Y_3$}}
			\put (50,12) {\scalebox{1.3}{$\dots$}}
			\put (29.5,0) {\scalebox{1}{$E_3$}}
			\put (75.5,13.8) {\scalebox{1}{$Y_{g-1}$}}
			\put (77.5,0) {\scalebox{1}{$E_{g-1}$}}
			\put (87,13.8) {\scalebox{1}{$Y_g$}}
			\put (88.5,0) {\scalebox{1}{$E_g$}}
			\put (100,11.3) {\scalebox{1}{$Y_{g+1}$}}
		\end{overpic}
		\[G_\infty\]
		\caption{The dual graph $G_\infty$ of the Eisenbud--Harris degenerate curve $C_\infty$.}
		\label{fig:dual-graph-eh}
	\end{figure}
	
	Here the components $Y_q$ are rational curves, so that $g(Y_q) = 0$, and the components $E_q$ are elliptic curves, with $g(E_q) = 1$. Replacing each elliptic component $E_q$ by a vertex with a self-loop yields a metric graph tropically equivalent to the caterpillar of loops $\HCL g$. Thus, the combinatorial construction underlying our Catalan count aligns closely with the limit linear series degenerations of Eisenbud and Harris~\cite{eh87}.
\end{re}


\section{Invariance of the count via continuous deformation} \label{sec-deformation-invariance}

We prove that the fibers of $\Pi \colon \TM d g \to \MTrop g$, counted with multiplicity, are invariant under continuous deformations. Deforming a generic metric graph to a caterpillar of loops extends the Catalan count $C_{g/2}$ across~$\MTrop g$.

\subsection{General paths} \label{subsec-general-paths}
We study the fiber count $|\inv \Pi(\mG)|$ along paths $\omega \colon [0,1] \to \MTrop g$, equipped with the quotient topology from its constituent cones. For $k \ge 0$, let $\mathbb L_k \FD g$ denote the \DTmors obtained by contracting $k$ edges from a morphism in $\FD g$ (which inherit full-rankness, dangling-no-glue, no-return, and pass-once by Lemma~\ref{lm:properties-inherit}), and write $\Pi(\mathbb L_k \FD g) = \bigcup_{\dtmor_0 \in \mathbb L_k \FD g} \Pi(\overline C_{\dtmor_0})$. Away from the codimension-1 walls $\Pi(\mathbb L_1 \FD g)$ and points with non-general edge lengths, the fiber count is locally constant:

\begin{prop}[general paths] \label{proposition-generic-path}
	Let $\omega \colon [0,1] \to \MTrop g \setminus \Pi(\mathbb L_1 \FD g)$ be a path. If $\omega(0)$ and $\omega(1)$ have general edge lengths, then
	\[
	|\inv \Pi(\omega(0))| = |\inv \Pi(\omega(1))|.
	\]
\end{prop} 

\begin{proof}
	Because $\omega$ avoids $\Pi(\mathbb L_1 \FD g)$, all graphs along it share a trivalent type $H$, so $\omega$ lifts to $C_H$. Since this path avoids all limits, where branching can occur, each morphism over $\omega(0)$ extends to a unique continuous path along $\omega$. General edge lengths at the endpoints ensure trivial $\Aut H$-stabilizers, so the fiber count is constant by Lemma~\ref{lemma-count-realizations}.
\end{proof}

To preserve the count across the codimension-1 walls $\Pi(\mathbb L_1 \FD g)$, morphisms must be counted with appropriate multiplicities.

\subsection{Multiplicity of morphisms} \label{subsec-multiplicity}
We define a determinantal multiplicity for full-dimensional morphisms satisfying a balancing condition across codimension-1 cells of $\TM d g$.

\begin{de}[multiplicities] \label{def-multiplicity}
	For $\dtmor \in \FD g$, let $A_\dtmor$ be its edge-length matrix, $l(T)$ the number of leaves of $T$, and $d_i$ the least common denominator of the $i$-th row of $A_\dtmor$. Setting $D_\dtmor = \prod_i d_i$, the \emph{signed multiplicity} of $\dtmor$ is
	\[
	\Mult \dtmor = \frac{D_\dtmor}{2^{l(T)}} \det A_\dtmor,
	\]
	and its \emph{multiplicity} is $\absMult \dtmor = |\Mult \dtmor|$, which is invariant under isomorphisms. The multiplicities of a point $\tmor \in \TM d g$ are defined via its combinatorial model.
\end{de}

\begin{lm}[toric multiplicity] \label{lemma-toric-multiplicity}
	For each $\dtmor \in \FD g$, the multiplicity $\absMult \dtmor$ equals the toric multiplicity of the polyhedral cone of the transpose matrix $A_\dtmor^T$. \qed
\end{lm}

\begin{re} \label{rem-sign-convention}
	The sign of $\Mult \dtmor$ depends on ordering the edges $E(H)$ and $E(T)$. Because $\TM d g$ is connected through codimension~1 (Theorem~\ref{thm}), choosing an orientation on a single cone determines signs globally across adjacent codimension-1 faces. By Lemma~\ref{lemma-toric-multiplicity}, the absolute multiplicity $\absMult \dtmor$ is independent of these conventions and admits an intrinsic interpretation as a toric lattice multiplicity.
\end{re}

We determine the denominators $d_i$ for full-dimensional morphisms:

\begin{lm}[edge denominators] \label{lemma-edge-deno}
	Let $\dtmor \in \FD g$, let $A_\dtmor$ be its edge-length matrix, and let $h_i \in E(H)$. Exactly one of the following holds:
	\begin{enumerate}[(a)]
		\item $h_i$ passes above a leaf and $d_i = 1$.
		\item $h_i$ avoids leaves, all edges in $h_i$ have constant index $k$, and $d_i = k$.
		\item $h_i$ avoids leaves, the edge indices in $h_i$ take values $k$ and $k+1$, and $d_i = k(k+1)$.
	\end{enumerate}   
\end{lm}

\begin{proof}
	By Cases~(r0) and (r1) of Proposition~\ref{prop-local}, indices on adjacent edges in $h_i$ differ by at most~1.
	
	Assume first that $h_i$ passes above a leaf $v \in V(T)$ via an edge $\tilde e \in h_i$, and suppose there exists an edge in $h_i$ with index $\ne 1$. By balancing at a leaf, $m(\tilde e) = 1$. Traversing $h_i$ from that edge to $\tilde e$, adjacent edge indices change by at most~1 at each step, so there exist adjacent edges $e, e' \in h_i$ with $m(e) = 2$ and $m(e') = 1$. Let $A \in V(G)$ be their common endpoint; because $\dtmor$ is full-dimensional, $r_\dtmor(A) = 1$. Perturbing the target lengths by $+z$ on $\dtmor(e)$, $-z$ on $\dtmor(e')$, and $+z/4$ on the leaf edge $t = \dtmor(\tilde e)$ alters no edge lengths in $H$ by Case~(r0) of Proposition~\ref{prop-local}. This produces a non-zero vector in $\ker A_\dtmor$, contradicting that $A_\dtmor$ has full rank. Hence every edge in $h_i$ has index~1, establishing item~(a).
	
	Now assume $h_i$ avoids leaves. By Proposition~\ref{proposition-at-most-two-weights}, edge indices along $h_i$ take at most two distinct values, differing by at most~1. If all edges have constant index $k$, every entry in row $i$ of $A_\dtmor$ is an integer multiple of $1/k$, so $d_i = k$, proving item~(b). If two distinct indices $k$ and $k+1$ occur, the least common multiple of denominators is $k(k+1)$, proving item~(c).
\end{proof}

\begin{re}[multisets of specializing morphisms] \label{re:multiset-specializing}
	When tracking tropical morphisms along a continuous path of metric graphs crossing a codimension-1 wall $C_{\dtmor_0}$, multiple geometric chambers in $\TM d g$ can specialize to the limit $\dtmor_0$. Crucially, two or more of these chambers may share the identical combinatorial type $\dtmor \in \FD g$, differing only by how lengths are assigned across symmetric features of the graph. Even though they represent the same combinatorial data, when following the metric path there is no collapse due to symmetries: the branches remain distinct in $\TM d g$. Each such chamber contributes an incident facet meeting at $C_{\dtmor_0}$, which forces $\cstar{\dtmor_0}$ to be defined as a multiset rather than a simple set.
\end{re}

\begin{ex}[wall crossing without collapse of symmetries] \label{ex-multiset-star}
	We illustrate this phenomenon using the deformation of a loop of three loops from~\cite[Section~4]{dv20} across the wall where the length $y_3$ of the central cycle satisfies $y_3 = y_1 + y_2$. Figure~\ref{fig:wall-crossing-multiset} displays the three regimes:
	
	\begin{figure}[htbp]
		\centering
		\begin{minipage}[b]{0.23\textwidth}
			\centering
			\begin{overpic}[scale=0.75]{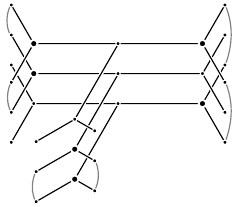}
				\put (15,58) {\scalebox{0.7}{$A$}}
				\put (25,24) {\scalebox{0.6}{$B$}}
				\put (25,11) {\scalebox{0.6}{$C$}}
				\put (80,45) {\scalebox{0.6}{$D$}}
				\put (80,71) {\scalebox{0.6}{$E$}}
				\put (15,71) {\scalebox{0.6}{$F$}}
			\end{overpic}
			\vspace{0.4em}
			
			\small{$\Aq \dtmor 1$ ($y_3 < y_1 + y_2$)}
		\end{minipage}\hfill
		\begin{minipage}[b]{0.23\textwidth}
			\centering
			\begin{overpic}[scale=0.75]{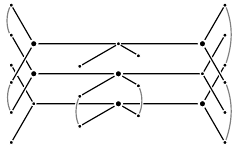}
				\put (15,33) {\scalebox{0.7}{$A$}}
				\put (47,33.5) {\scalebox{0.6}{$B$}}
				\put (47,20) {\scalebox{0.6}{$C$}}
				\put (80,20) {\scalebox{0.6}{$D$}}
				\put (80,46) {\scalebox{0.6}{$E$}}
				\put (15,46) {\scalebox{0.7}{$F$}}
			\end{overpic}
			\vspace{0.4em}
			
			\small{$\dtmor_0$ ($y_3 = y_1 + y_2$)}
		\end{minipage}\hfill
		\begin{minipage}[b]{0.23\textwidth}
			\centering
			\begin{overpic}[scale=0.75]{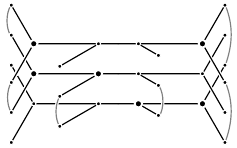}
				\put (15,33) {\scalebox{0.7}{$A$}}
				\put (40,33) {\scalebox{0.6}{$B$}}
				\put (55,20) {\scalebox{0.6}{$C$}}
				\put (80,20) {\scalebox{0.6}{$D$}}
				\put (80,46) {\scalebox{0.6}{$E$}}
				\put (15,46) {\scalebox{0.7}{$F$}}
			\end{overpic}
			\vspace{0.4em}
			
			\small{$\Aq \dtmor {3a}$ ($y_4 < y_5$)}
		\end{minipage}\hfill
		\begin{minipage}[b]{0.23\textwidth}
			\centering
			\begin{overpic}[scale=0.75]{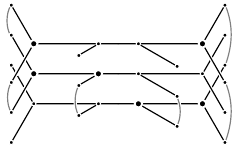}
				\put (15,33) {\scalebox{0.7}{$A$}}
				\put (40,33) {\scalebox{0.6}{$B$}}
				\put (55,20) {\scalebox{0.6}{$C$}}
				\put (80,20) {\scalebox{0.6}{$D$}}
				\put (80,46) {\scalebox{0.6}{$E$}}
				\put (15,46) {\scalebox{0.7}{$F$}}
			\end{overpic}
			\vspace{0.4em}
			
			\small{$\Aq \dtmor {3b}$ ($y_4 > y_5$)}
		\end{minipage}
		\caption{Deformation across a codimension-1 wall. On the left, $\Aq \dtmor 1$ realizes metric graphs with $y_3 < y_1 + y_2$. In the middle, contracting the target edge yields the codimension-1 limit $\dtmor_0$. On the right, regrowing into the region $y_3 > y_1 + y_2$ produces two morphisms $\Aq \dtmor {3a}$ and $\Aq \dtmor {3b}$ that share the same combinatorial type $\dtmor_3$, but differ by which edge between $B$ and $C$ is longer.}
		\label{fig:wall-crossing-multiset}
	\end{figure}
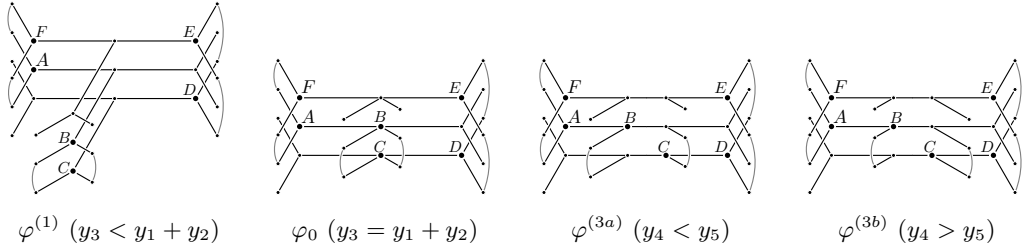
	
	Before the wall, $\mH$ is realized by the single morphism $\Aq \dtmor 1$. At the wall, contracting the target edge $t_1$ yields the codimension-1 limit $\dtmor_0$. Beyond the wall, regrowing $t_1$ gives two realizations $\Aq \tmor {3a} = (\dtmor_3, \Aq z a)$ and $\Aq \tmor {3b} = (\dtmor_3, \Aq z b)$. Combinatorially, $\Aq \dtmor {3a}$ and $\Aq \dtmor {3b}$ are isomorphic, defining the same combinatorial type $\dtmor_3 \in \FD g$. However, on any metric graph with general edge lengths ($y_4 \ne y_5$), the two morphisms are distinct points in $\inv \Pi(\mH)$, distinguished by which of the two edges between $B$ and $C$ contains the preimage of the target branching vertex $\dtmor(B)$. Symmetries do not collapse these two branches along the metric path. To balance the single cone of $\Aq \dtmor 1$ on the other side of the wall, the combinatorial type $\dtmor_3$ must appear with multiplicity~2 in $\cstar{\dtmor_0}$.
\end{ex}

The edge denominators transform the local linear relations among edge-length matrices into a balancing condition for multiplicities:

\begin{prop}[balancing condition] \label{prop-signed-mult}
	Let $\dtmor_0 \in \mathbb L_1 \FD g$ be a codimension-1 limit. Let $\cstar{\dtmor_0}$ denote the multiset of full-dimensional morphisms in $\FD g$ that specialize to $\dtmor_0$.
	\begin{enumerate}[(1)]
		\item If $H(\dtmor_0)$ is trivalent, then
		\[
		\sum_{\dtmor \in \cstar{\dtmor_0}} \Mult \dtmor = 0.
		\]
		\item If $H(\dtmor_0)$ is non-trivalent, then $\Mult \dtmor = \Mult{\dtmor'}$ for every pair $\dtmor, \dtmor' \in \cstar{\dtmor_0}$.
	\end{enumerate} 
\end{prop}

\begin{proof}
	For the first statement, we combine the determinant relations across adjacent cones with the edge denominators from Lemma~\ref{lemma-edge-deno}. We illustrate this on the representative trivalent case \{w2-r2-nd3-M-1k\} (Appendix~\ref{appendix-trivalent-deformation}). Let $\Aq \dtmor 1, \Aq \dtmor 2, \Aq \dtmor 3 \in \cstar{\dtmor_0}$, and write $\Aq c q = \det \Aq A q$. Here $l(\Aq T 2) = l(\Aq T 3) = l(T_0)$ while $l(\Aq T 1) = l(T_0) + 1$. For rows $h_1, h_2, h_3$ and columns $t_1, t_2, t_3$, the matrices $\Aq A q$ take the form:
	
	\vspace{0.5em}
	\noindent
	\begin{minipage}[t]{.3\textwidth} 
		\centering
		$\begin{pmatrix}
		2 & 1 & 0 &     \\
		0 & \frac{1}{k} & 0 & \dots\\
		0 & 0 & \frac{1}{k} &    \\
		0 & \Aq {a_{i2}} 1 & \dots\\
		&\vdots&& \ddots 
		\end{pmatrix}$
		\vspace{0.5em}
		
		$\Aq A 1$
	\end{minipage}\hspace{0.7em}
	\begin{minipage}[t]{.3\textwidth} 
		\centering
		$\begin{pmatrix}
		1 & 1 & 0 &     \\
		\frac{1}{k-1} & \frac{1}{k} & 0 & \dots\\
		0 & 0 & \frac{1}{k} &    \\
		\Aq {a_{i1}} 2 & \Aq {a_{i2}} 2 & \dots\\
		&\vdots&& \ddots 
		\end{pmatrix}$
		\vspace{0.5em}
		
		$\Aq A 2$ 
		
		\small{$\Aq {a_{i1}} 2 = \Aq {a_{i2}} 2$ for $i \ge 4$.}
	\end{minipage}\hspace{0.7em}
	\begin{minipage}[t]{.3\textwidth} 
		\centering
		$\begin{pmatrix}
		0 & 1 & 0 &     \\
		0 & \frac{1}{k} & 0 & \dots\\
		\frac{1}{k+1} & 0 & \frac{1}{k} &    \\
		\Aq {a_{i1}} 3 & \Aq {a_{i2}} 3 & \dots\\
		&\vdots&& \ddots 
		\end{pmatrix}$
		\vspace{0.5em}
		
		$\Aq A 3$
		
		\small{$\Aq {a_{i1}} 3 = \Aq {a_{i2}} 3$ for $i \ge 4$.}
	\end{minipage}
	\vspace{0.7em}
	
	Let $d_{0,i}$ be the least common denominator of row $i$ in $A_0$, and set $D_0 = \prod_{i=1}^{3g-3} d_{0,i}$. Since $\Aq {d_i} q = d_{0,i}$ for all $i \ge 4$, we have:
	\begin{align*}
	\Mult{\Aq \dtmor 1} + \Mult{\Aq \dtmor 2} + \Mult{\Aq \dtmor 3} 
	&= \frac{D_0}{2^{l(T_0)}} \Biggl( 
		\frac{\Aq {d_1} 1\Aq {d_2} 1\Aq {d_3} 1}{2 d_{0,1} d_{0,2} d_{0,3}} \Aq c 1 + 
		\frac{\Aq {d_1} 2\Aq {d_2} 2\Aq {d_3} 2}{d_{0,1} d_{0,2} d_{0,3}} \Aq c 2 \\
		&\quad + \frac{\Aq {d_1} 3\Aq {d_2} 3\Aq {d_3} 3}{d_{0,1} d_{0,2} d_{0,3}} \Aq c 3  \Biggr).
	\end{align*}
	Setting $\Aq b 1 = 1$, $\Aq b 2 = k-1$, and $\Aq b 3 = k+1$, the determinant identity across this limit states that $\frac{1}{2} \Aq b 1 \Aq c 1 + \Aq b 2 \Aq c 2 + \Aq b 3 \Aq c 3 = 0$. By Lemma~\ref{lemma-edge-deno}, the ratio $\prod_{j=1}^3 \Aq {d_j} q / \prod_{j=1}^3 d_{0,j}$ matches $\Aq b q$ whenever $\Aq c q \ne 0$, so the sum vanishes. The same cancellation holds for all other trivalent deformation cases in Appendix~\ref{appendix-trivalent-deformation}.
	
	The second statement follows directly from the local wall-crossing isomorphisms across non-trivalent limits constructed in Section~\ref{sec-constructions}.
\end{proof}

\subsection{Wall crossing} \label{subsec-walking-through-walls}
The balancing condition ensures that fiber counts with multiplicity are invariant across codimension-1 walls:

\begin{prop}[wall crossing] \label{proposition-walking-through-II}
	Let $\omega \colon [0,1] \to \MTrop g \setminus \Pi(\mathbb L_2 \FD g)$ be a path meeting $\Pi(\mathbb L_1 \FD g)$ in a single point $\mH_0 = \omega(s_0)$ with $s_0 \in (0,1)$. If $\omega(0)$ and $\omega(1)$ have general edge lengths, then
	\[
	\sum_{\tmor \in \inv \Pi(\omega(0))} \absMult \tmor = \sum_{\tmor \in \inv \Pi(\omega(1))} \absMult \tmor.
	\]
\end{prop}

\begin{proof}
	By Proposition~\ref{proposition-generic-path}, the count is constant on $\omega \setminus \{\mH_0\}$; it suffices to compare points $\mH_1, \mH_2$ on opposite sides arbitrarily close to $\mH_0$. Since $\mH_0 \notin \Pi(\mathbb L_2 \FD g)$, morphisms in cone interiors contribute identically to both sides.
	
	For each limit $\dtmor_0 \in \inv \Pi(\mH_0) \cap \mathbb L_1 \FD g$, if $H(\dtmor_0)$ is trivalent, $C_{\dtmor_0}$ divides the ambient cone into two half-spaces containing $\mH_1$ and $\mH_2$. Cones project to opposite sides according to the sign of $\Mult \dtmor$, and the balancing condition $\sum_{\dtmor \in \cstar{\dtmor_0}} \Mult \dtmor = 0$ (Proposition~\ref{prop-signed-mult}(1)) equates the sums of multiplicities on each side. If $H(\dtmor_0)$ is non-trivalent, $\mH_0$ separates two chambers of $\MTrop g$, and Proposition~\ref{prop-signed-mult}(2) matches morphisms across the wall with equal multiplicities. Summing over all limits preserves the total count.
\end{proof}

\subsection{Proofs of the main theorems} \label{sec-proof-main-theorems}
We assemble the deformation results to prove Theorem~\ref{thm} and Corollary~\ref{theorem-gonality}.

\begin{lm}[{\cite[Proposition~3.3.3]{cap12}}] \label{lemma-tropM-connected-co1}
	The moduli space $\MTrop g$ is connected through codimension~1. \qed
\end{lm}

\begin{proof}[Proof of Theorem~\ref{thm}]
	Let $g = 2g'$ be even and $d = g'+1$. Let $\mH \in \MTrop g$ be a metric graph with general edge lengths, and choose a generic caterpillar of loops $\mH_L \in \MTrop g$.
	
	The locus $\Pi(\mathbb L_2 \FD g)$ has codimension at least~2 in $\MTrop g$, and $\FD g$ is finite. By Lemma~\ref{lemma-tropM-connected-co1}, there exists a continuous piecewise-linear path $\omega \colon [0,1] \to \MTrop g \setminus \Pi(\mathbb L_2 \FD g)$ connecting $\mH_L$ to $\mH$ that intersects $\Pi(\mathbb L_1 \FD g)$ transversely in finitely many points $0 < s_1 < \dots < s_k < 1$.
	
	Subdividing $\omega$ at these wall crossings decomposes it into general segments and wall crossings. By Proposition~\ref{proposition-generic-path}, the fiber count with multiplicity is constant along each general segment; by Proposition~\ref{proposition-walking-through-II}, it is invariant across each wall crossing. Therefore,
	\[
	\sum_{\tmor \in \inv \Pi(\mH)} \absMult \tmor = \sum_{\tmor \in \inv \Pi(\mH_L)} \absMult \tmor.
	\]
	For the caterpillar of loops $\mH_L$, Proposition~\ref{prop-divisors-on-chain} shows that $\inv \Pi(\mH_L)$ consists of exactly $C_{g/2}$ points. For each $\dtmor \in \FDCL g$, the matrix $A_\dtmor$ has unit determinant and $D_\dtmor = 2^{l(T)}$, so $\absMult \dtmor = 1$. Consequently, the fiber over every generic metric graph contains exactly $C_{g/2}$ points counted with multiplicity, proving that $\Pi$ is a branched cover of degree $C_{g/2}$ branched along the codimension-1 discriminant locus $\Pi(\mathbb L_1 \FD g)$.
	
	Because generic metric graphs are dense and fibers over them are non-empty, the closed map of cone complexes $\Pi$ is surjective onto $\MTrop g$. Finally, lifting $\omega$ connects every full-dimensional cone of $\TM d g$ across codimension-1 faces to the cones over the caterpillar locus, which are themselves connected through codimension-1 cells by Proposition~\ref{prop-divisors-on-chain}. Hence $\TM d g$ is connected through codimension~1.
\end{proof}

\begin{proof}[Proof of Corollary~\ref{theorem-gonality}]
	If $g$ is even, surjectivity of $\Pi$ in Theorem~\ref{thm} provides a tropical morphism from a tropical modification of $\mG$ to a metric tree of degree $g/2 + 1 = \lceil g/2 \rceil + 1$.
	
	Suppose $g$ is odd. Choose an interior point $x$ on an edge of $\mG$, and attach a self-loop via a bridge to form a metric graph $\mG_b$ of even genus $g+1$. Since $g$ is odd, $(g+1)/2 + 1 = \lceil g/2 \rceil + 1$. By Theorem~\ref{thm}, there is a tropical morphism $\tmor_b \colon \mG_b' \to \mT_b$ from a tropical modification $\mG_b'$ of $\mG_b$ to a metric tree $\mT_b$ of degree $\lceil g/2 \rceil + 1$.
	
	By Lemma~\ref{lm:bridge-and-loop}, the bridge and loop attached at $x$ map to two distinct edges of $\mT_b$, and no other non-dangling elements of $\mG_b'$ map to these edges. Deleting these two edges from $\mT_b$ and restricting $\tmor_b$ yields a tropical morphism from a tropical modification of $\mG$ to a metric tree of degree $\lceil g/2 \rceil + 1$. Hence the tree gonality of any genus-$g$ metric graph is at most $\lceil g/2 \rceil + 1$.
\end{proof}


\section{Constructions: changing combinatorial type} \label{sec-constructions}

We construct the local deformations across codimension-1 walls where the combinatorial type of the source graph changes, proving the balancing condition of Proposition~\ref{prop-signed-mult}(2) for non-trivalent limits.

\subsection{Combinatorial setup and local determinants} \label{subsec-setup-determinants}
Let $\dtmor \colon G \to T$ be a full-dimensional \DTmor and let $t$ be an edge of $T$ whose contraction yields a codimension-1 limit $\dtmor_0$ with non-trivalent combinatorial type $H_0 = H(\dtmor_0)$.

Because $H_0$ is non-trivalent, $|E(H_0)| \le 3g - 4$. Since $\dtmor_0$ has codimension~1, its edge-length matrix satisfies $\rk A_{\dtmor_0} = 3g - 4$, which forces $|E(H_0)| = 3g - 4$. It follows that $H_0$ has a unique 4-valent vertex $A$, while all other vertices are trivalent. 

To analyze the full-dimensional morphisms specializing to $\dtmor_0$, we fix the following labelling conventions. Let $\Aq \dtmor q$ be any such morphism:
\begin{enumerate}[(1)]
	\item \textbf{The contracting edge in the source:} Exactly one edge of $\Aq H q$ contracts to the 4-valent vertex $A$; we label this edge $\Aq {h_1} q$, and denote its trivalent endpoints in $V(\Aq H q)$ by $A^{(q)}_1$ and $A^{(q)}_2$ (so $\nddeg \Aq A q _1 = \nddeg \Aq A q _2 = 3$). Aside from $\Aq {h_1} q$, the edges of $\Aq H q$ correspond bijectively to those of $H_0$, and for each $h \in E(H_0)$ we denote by $\Aq h q$ the corresponding edge in $E(\Aq H q)$.
	
	\item \textbf{The contracting edge in the target:} In the base tree $\Aq T q$, we label the edge contracting to $w_0 \in V(T_0)$ by $t_1$, with endpoints $u$ and $v$.
	
	\item \textbf{Edges incident to $A$:} In $G_0$, we label the four non-dangling edges incident to $A$ as $e_2, e_3, e_4, e_5$, and write $k_i = |e_i|$ for their cardinalities. The corresponding edges of $H_0$ incident to $A$ are labeled $h_2, h_3, h_4, h_5$.
	
	\item \textbf{Placement above $u$ and $v$:} When both $\vale u$ and $\vale v$ are distinct from~1, the no-return condition implies that $A^{(q)}_1$ and $A^{(q)}_2$ lie above distinct vertices of $\Aq T q$; we adopt the convention that $A^{(q)}_1$ lies above $u$ and $A^{(q)}_2$ lies above $v$. Under this convention, the edge $\Aq {h_1} q$ contains a unique edge class $\Aq e q _1$ above $t_1$, with cardinality $\Aq k q _1 = |\Aq e q _1|$.
\end{enumerate}
	
	\subsubsection{Vertices above \texorpdfstring{$w_0$}{w0}} \label{sub-above-w0}
	Before analyzing the combinatorial resolutions of $H_0$, we verify that the wall-crossing deformation is localized entirely at the 4-valent vertex $A$, while all other vertices in the fiber above $w_0$ remain rigid:

	\begin{lm}[rigidity above $w_0$] \label{lemma-above-w0}
		Let $\dtmor_0 \colon G_0 \to T_0$ be a codimension-1 limit with non-trivalent combinatorial type $H_0$. Let $A \in V(G_0)$ be the unique 4-valent vertex of $H_0$, and let $w_0 \in V(T_0)$ be the image of the contracting edge $t_1$. Then
		\[
		r_0(A) = \ch w_0 = 4 - \vale w_0,
		\]
		and $r_0(B) = 0$ for every other vertex $B \in V(G_0) \setminus \{A\}$ above $w_0$. In particular, all classes in $\Aq \dtmor q$ contracting to any vertex $B \ne A$ are uniquely determined by the base tree $\Aq T q$, so the deformation across $\dtmor_0$ is localized entirely at $A$.
	\end{lm}

	\begin{proof}
		Because any full-dimensional morphism $\Aq \dtmor q$ specializing to $\dtmor_0$ is change-minimal at both endpoints $u$ and $v$ of $t_1$, contracting $t_1$ gives $\ch w_0 = 4 - \vale w_0$. Since $\ch w_0 = \sum_{B \in \inv \dtmor_0(w_0)} r_0(B)$ with $r_0(B) \ge 0$ for all $B$, it suffices to show that $r_0(A) = \ch w_0$.
		
		We check the three possibilities for $\vale w_0$:
		\begin{itemize}
			\item If $\vale w_0 = 4$, then $\ch w_0 = 0$. Since $r_0(A) \ge 0$, we have $r_0(A) = 0$.
			\item If $\vale w_0 = 3$, then $\ch w_0 = 1$ and $\vale u = 2$. By the local properties on $\Aq A q _1$ (Proposition~\ref{prop-local}), we have $r_{\Aq \dtmor q}(\Aq A q _1) = 1$, which contracts to give $r_0(A) = 1$.
			\item If $\vale w_0 = 2$, then $\ch w_0 = 2$. If $\vale u = 1$ and $\vale v = 3$, the edge connecting $\Aq A q _1$ and $\Aq A q _2$ passes above $u$, so a vertex with $r$-value equal to~2 contracts to $A$. If $\vale u = \vale v = 2$, local properties imply $r_{\Aq \dtmor q}(\Aq A q _1) = r_{\Aq \dtmor q}(\Aq A q _2) = 1$, both contracting to $A$. In either case, $r_0(A) = 2$.
		\end{itemize}
		Thus $r_0(A) = \ch w_0$ in all cases, forcing $r_0(B) = 0$ for every $B \ne A$.
	\end{proof}  
	
	\subsubsection{Graphs contracting to \texorpdfstring{$H_0$}{H0}}
	To construct a trivalent graph that contracts to $H_0$, the 4-valent vertex $A$ must resolve into two trivalent vertices $A_1^{(q)}$ and $A_2^{(q)}$ connected by the edge $\Aq {h_1} q$. Choosing two of the incident edges $\{h_2, h_3, h_4, h_5\}$ to meet at $A_1^{(q)}$ and the remaining two at $A_2^{(q)}$ defines the combinatorial type $H_{\alpha,\beta}$ (or $H_S$ with $S = \{\alpha, \beta\}$). Assuming without loss of generality that $h_2$ is incident to $A_1^{(q)}$, there are three possible combinatorial resolutions:
	\begin{itemize}
		\item \textbf{Type I ($H_{2,3}$):} $h_2$ and $h_3$ are incident to $A_1^{(q)}$; $h_4$ and $h_5$ are incident to $A_2^{(q)}$.
		\item \textbf{Type II ($H_{2,4}$):} $h_2$ and $h_4$ are incident to $A_1^{(q)}$; $h_3$ and $h_5$ are incident to $A_2^{(q)}$.
		\item \textbf{Type III ($H_{2,5}$):} $h_2$ and $h_5$ are incident to $A_1^{(q)}$; $h_3$ and $h_4$ are incident to $A_2^{(q)}$.
	\end{itemize}
	See the diagrams below.
	
	\noindent \begin{minipage}{.23\textwidth}
		\centering
		\begin{overpic}[scale=1.2]{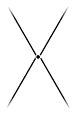} 
			\put (27,62) {\scalebox{0.8}{$A$}} 
			\put (4,75) {\scalebox{0.8}{$h_2$}}
			\put (6,34) {\scalebox{0.8}{$h_3$}}
			\put (52,75) {\scalebox{0.8}{$h_4$}}
			\put (48,34) {\scalebox{0.8}{$h_5$}}		 
		\end{overpic}
		
		$H_0$
	\end{minipage}
	\begin{minipage}{.23\textwidth}
		\centering
		\begin{overpic}[scale=1.2]{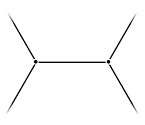}  
			\put (23.5,47) {\scalebox{0.8}{$A_1^{(q)}$}}
			\put (62,47) {\scalebox{0.8}{$A_2^{(q)}$}}
			\put (2,52) {\scalebox{0.8}{$\Aq {h_2} q$}}
			\put (2,26) {\scalebox{0.8}{$\Aq {h_3} q$}}
			\put (87,52) {\scalebox{0.8}{$\Aq {h_4} q$}}
			\put (87,26) {\scalebox{0.8}{$\Aq {h_5} q$}}
			\put (44,31) {\scalebox{0.8}{$\Aq {h_1} q$}}		
		\end{overpic} 
		
		$H_{2,3}$ \\ Type I
	\end{minipage}\hspace{1em}
	\begin{minipage}{.23\textwidth} 
		\centering
		\begin{overpic}[scale=1.2]{\figsdir/124.pdf}  
			\put (23.5,47) {\scalebox{0.8}{$A_1^{(q)}$}}
			\put (62,47) {\scalebox{0.8}{$A_2^{(q)}$}}
			\put (2,52) {\scalebox{0.8}{$\Aq {h_2} q$}}
			\put (2,26) {\scalebox{0.8}{$\Aq {h_4} q$}}
			\put (87,52) {\scalebox{0.8}{$\Aq {h_3} q$}}
			\put (87,26) {\scalebox{0.8}{$\Aq {h_5} q$}}		
			\put (44,31) {\scalebox{0.8}{$\Aq {h_1} q$}}
		\end{overpic}

		$H_{2,4}$ \\ Type II
	\end{minipage}\hspace{1em}
	\begin{minipage}{.23\textwidth} 
		\centering
		\begin{overpic}[scale=1.2]{\figsdir/124.pdf} 
			\put (23.5,47) {\scalebox{0.8}{$A_1^{(q)}$}}
			\put (62,47) {\scalebox{0.8}{$A_2^{(q)}$}}
			\put (2,52) {\scalebox{0.8}{$\Aq {h_2} q$}}
			\put (2,26) {\scalebox{0.8}{$\Aq {h_5} q$}}
			\put (87,52) {\scalebox{0.8}{$\Aq {h_4} q$}}
			\put (87,26) {\scalebox{0.8}{$\Aq {h_3} q$}}		
			\put (44,31) {\scalebox{0.8}{$\Aq {h_1} q$}}
		\end{overpic}  
		
		$H_{2,5}$ \\ Type III
	\end{minipage} \vspace{1em}
	
	\subsubsection{Counting and multiplicity preservation}
	To establish Proposition~\ref{prop-signed-mult}(2), we must show that across any non-trivalent limit $\dtmor_0$, the full-dimensional morphisms specializing to $\dtmor_0$ have equal multiplicities and partition equally among Types I, II, and III. We first relate their determinants and multiplicities across the contracting edge:
	
	\begin{lm}[determinants and multiplicities across non-trivalent limits] \label{lm:change-comb-type} \label{lemma-change-comb-type}
		Let $\Aq \dtmor q$ and $\Aq \dtmor {q'}$ be full-dimensional morphisms specializing to $\dtmor_0$ by contracting $t_1$. Then:
		\begin{enumerate}[(1)]
			\item Their edge-length determinants satisfy
			\[
			\Aq k q _1 \det(A_{\Aq \dtmor q}) = \Aq k {q'} _1 \det(A_{\Aq \dtmor {q'}}),
			\]
			where $\Aq k q _1$ is $\frac{1}{2}$ if $t_1$ is incident to a leaf, and otherwise the cardinality of the class $\Aq e q _1$ containing the contracting edge $\Aq {h_1} q$.
			\item Their signed multiplicities are equal:
			\[
			\Mult{\Aq \dtmor q} = \Mult{\Aq \dtmor {q'}}.
			\]
		\end{enumerate}
	\end{lm}
	
	\begin{proof}
		For (1), by Lemma~\ref{lm:limit-matrix-change}, the matrices $A_{\Aq \dtmor q}$ and $A_{\Aq \dtmor {q'}}$ differ only in the row and column corresponding to the contracting edges $\Aq {h_1} q$ and $t_1$. A cofactor expansion along the row of $\Aq {h_1} q$, where the only non-zero entry is $1/\Aq k q _1$, expresses $\det(A_{\Aq \dtmor q})$ as $(1/\Aq k q _1) \det(A_{\dtmor_0})$. The same expansion for $q'$ yields the equality.
		
		For (2), by Lemma~\ref{lemma-edge-deno}, the denominator of the row corresponding to $\Aq {h_1} q$ is $d_1^{(q)} = \Aq k q _1$, while the denominators of all other rows and the leaf count $l(T)$ are identical for $\Aq \dtmor q$ and $\Aq \dtmor {q'}$. Multiplying the determinant identity in (1) by $\prod_{i \ge 2} d_i / 2^{l(T)}$ immediately yields $\Mult{\Aq \dtmor q} = \Mult{\Aq \dtmor {q'}}$.
	\end{proof}
	
	It remains to prove that the full-dimensional morphisms specializing to $\dtmor_0$ partition equally among Types I, II, and III. We verify this case by case for each valency $\vale w_0 \in \{4, 3, 2\}$ in Subsections~\ref{subsec-case-v4}, \ref{subsec-case-v3}, and \ref{subsec-case-v2}.  
	
	\subsection{Valency-4 limits: Case \{v4-nd4\}} \label{subsec-case-v4}
	Let $\vale w_0 = 4$. By Lemma~\ref{lemma-above-w0}, $r_0(A) = 0$, so $r_{\Aq \dtmor q}(\Aq A q _1) = r_{\Aq \dtmor q}(\Aq A q _2) = 0$. Since the endpoints $u, v$ of $t_1$ are trivalent, Case~(r0-nd3) of Proposition~\ref{prop-local} implies that the four edges of $G_0$ incident to $A$ lie above distinct edges of $T_0$. Furthermore, $\Aq e q_1$ is the unique edge of $\Aq G q$ above $t_1$ that contracts to $A$.

	Label the edges of $G_0$ and $T_0$ such that $e_i$ is above $t_i$ and $|e_5| \ge |e_4| \ge |e_3| \ge |e_2|$. The three trees contracting to $T_0$ are $T_{2,3}, T_{2,4}$, and $T_{2,5}$. For each $S \in \{\{2,3\}, \{2,4\}, \{2,5\}\}$, setting $\Aq T q = T_S$ specifies the combinatorial type $H_S$. Applying the non-dangling excess formula (Lemma~\ref{lem-rphi-nd}) to $A$ gives identity~\eqref{eq-sqr}:
	\begin{align} \label{eq-sqr} \tag{$\square$}
	k_2 + k_3 + k_4 + k_5 = 2|A| + 2. 
	\end{align}

	The local part of $\Aq \dtmor q$ contracting to $A$ is uniquely determined by the values $|\Aq A q _1|$, $|\Aq A q _2|$, $\Aq k q _1 = |\Aq e q _1|$, and the base tree. For each $\Aq \dtmor q$, we introduce indices $\{\alpha, \beta, \gamma, \delta\} = \{2,3,4,5\}$ and relabel $\Aq A q _1, \Aq A q _2$ as $\Aq A q _-, \Aq A q _+$ such that:
	\begin{itemize}
		\item $\Aq e q _\alpha$ and $\Aq e q _\beta$ are incident to $\Aq A q _-$, while $\Aq e q _\gamma$ and $\Aq e q _\delta$ are incident to $\Aq A q _+$;
		\item $k_\alpha + k_\beta \le k_\gamma + k_\delta$, with $\alpha = 2$ if equality holds;
		\item $k_\alpha \le k_\beta$, with $\alpha = 2$ whenever $2 \in \{\alpha, \beta\}$;
		\item $k_\gamma \le k_\delta$, with $\delta = 5$ whenever $5 \in \{\gamma, \delta\}$.
	\end{itemize}
	These conditions uniquely specify the relabelling. Since both $\Aq A q _-$ and $\Aq A q _+$ are incident to $\Aq e q _1$, Lemma~\ref{lem-rphi-nd} implies $|\Aq A q _-| \le |\Aq A q _+|$, motivating the signs.

	We set $|\Aq A q _+| = |A| - K$ for an integer $K \ge 0$. Because $\max(k_\gamma, k_\delta) = k_\delta \le |\Aq A q _+| \le |A|$, fixing $K$ determines the sizes of $\Aq k q _1$ and $|\Aq A q _-|$ via Lemma~\ref{lem-rphi-nd}:
	\begin{align*}
	|\Aq A q _+| &= |A| - K,\\
	|\Aq k q _1| &= k_\alpha + k_\beta - 1 - 2K,\\
	|\Aq A q _-| &= k_\alpha + k_\beta - 1 - K.
	\end{align*}
	The upper bounds $|A| \ge |\Aq A q _+|, |\Aq A q _-|, \Aq k q _1$ hold for all $K \ge 0$ since $k_\alpha + k_\beta \le |A| + 1$ by \eqref{eq-sqr}. The lower bounds $|\Aq A q _+| \ge k_\delta$, $|\Aq A q _-| \ge k_\beta$, and $\Aq k q _1 \ge 1$ determine the admissible range for $K$ according to whether $k_5 + k_2 \le |A| + 1$ or $k_5 + k_2 > |A| + 1$:

	\paragraph*{Subcase \{v4-nd4-1\} ($k_5 + k_2 \le |A| + 1$):}
	Here $k_\alpha = k_2$. The second lower bound $|\Aq A q _-| = k_2 + k_\beta - 1 - K \ge k_\beta$ holds if and only if $K \le k_2 - 1$. Assuming $K \le k_2 - 1$, we have $|\Aq A q _+| = |A| - K \ge |A| - k_2 + 1 \ge k_\delta$ since $k_2 + k_\delta \le |A| + 1$, satisfying the first lower bound. Finally, $\Aq k q _1 = |\Aq A q _-| - K \ge k_2 - (k_2 - 1) = 1$.

	Taking branch swappings into account, there are exactly $k_2$ morphisms for each base tree. See the figures below for $K = 1$:

	\vspace{0.5em}
	\noindent  
	\begin{minipage}[t]{.23\textwidth}
		\centering
		\begin{overpic}{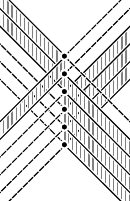} 
			\put (28.5,19) {\scalebox{0.8}{$w_0$}} 
			\put (29,76) {\scalebox{0.8}{$A$}} 
			\put (-8,34) {\scalebox{0.8}{$e_2$}}
			\put (-8,92) {\scalebox{0.8}{$e_3$}}
			\put (66,70) {\scalebox{0.8}{$e_4$}}
			\put (66,12) {\scalebox{0.8}{$e_5$}}		 
		\end{overpic}
		
		\vspace{1em}  
		$\dtmor_0$
	\end{minipage}	
	\begin{minipage}[t]{.23\textwidth} 
		\centering
		\begin{overpic}{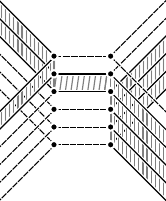}
			\put (47,15) {\scalebox{0.7}{$A_+^{(1)}$}}
			\put (25,75) {\scalebox{0.7}{$A_-^{(1)}$}}
		\end{overpic} 
		
		\vspace{1em}
		$\Aq \dtmor 1$\\Type I
	\end{minipage}\hspace{0.7em}
	\begin{minipage}[t]{.23\textwidth}
		\centering 
		\begin{overpic}{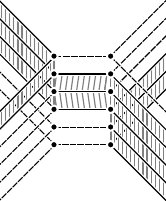} 
			\put (45,17) {\scalebox{0.7}{$A_+^{(2)}$}}
			\put (25,75) {\scalebox{0.7}{$A_-^{(2)}$}}
		\end{overpic} 
		
		\vspace{1em}
		$\Aq \dtmor 2$ \\Type II
	\end{minipage}\hspace{0.7em}
	\begin{minipage}[t]{.23\textwidth} 
		\centering
		\begin{overpic}{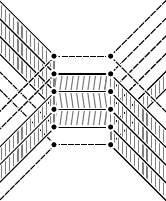} 
			\put (45,17) {\scalebox{0.7}{$A_-^{(3)}$}}
			\put (25,75) {\scalebox{0.7}{$A_+^{(3)}$}}
		\end{overpic}
		
		\vspace{1em}  
		$\Aq \dtmor 3$ \\Type III
	\end{minipage}

	\vspace{1em}

	\paragraph*{Subcase \{v4-nd4-2\} ($k_5 + k_2 > |A| + 1$):}
	Here $k_\delta = k_5$. The first lower bound $|\Aq A q _+| = |A| - K \ge k_5$ holds if and only if $K \le |A| - k_5$. This implies $|\Aq A q _-| = k_\alpha + k_\beta - 1 - K \ge k_\alpha + k_\beta + k_5 - 1 - |A| = |A| - k_\gamma$ by~\eqref{eq-sqr}. Since $k_\alpha + k_5 > |A| + 1$, we have $k_\beta + k_\gamma \le |A|$, so $|A| - k_\gamma \ge k_\beta$, satisfying the second lower bound. Lastly, $\Aq k q _1 = |\Aq A q _-| - K \ge k_5 - (|A| - k_5) > 1$ since $2k_5 \ge k_5 + k_2 > |A| + 1$.

	Taking branch swappings into account, there are exactly $|A| - k_5 + 1$ morphisms for each base tree. See the figures below for $K = 0$:

	\vspace{0.5em}
	\noindent  
	\begin{minipage}[t]{.23\textwidth}
		\centering
		\begin{overpic}{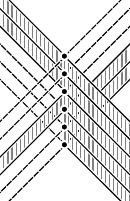} 
			\put (28.5,19) {\scalebox{0.8}{$w_0$}} 
			\put (29,76) {\scalebox{0.8}{$A$}} 
			\put (-8,34) {\scalebox{0.8}{$e_2$}}
			\put (-8,92) {\scalebox{0.8}{$e_3$}}
			\put (66,70) {\scalebox{0.8}{$e_4$}}
			\put (66,12) {\scalebox{0.8}{$e_5$}}		 
		\end{overpic}
		
		\vspace{1em}  
		$\dtmor_0$
	\end{minipage}	
	\begin{minipage}[t]{.23\textwidth} 
		\centering
		\begin{overpic}{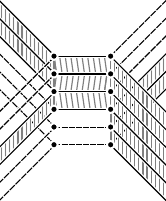}
			\put (45,75) {\scalebox{0.7}{$A_+^{(1)}$}}
			\put (25,75) {\scalebox{0.7}{$A_-^{(1)}$}}
		\end{overpic} 
		
		\vspace{1em}
		$\Aq \dtmor 1$\\ Type I
	\end{minipage}\hspace{0.7em}
	\begin{minipage}[t]{.23\textwidth}
		\centering 
		\begin{overpic}{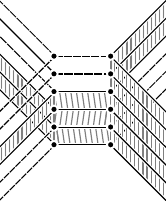}  
			\put (45,75) {\scalebox{0.7}{$A_+^{(2)}$}}
			\put (25,75) {\scalebox{0.7}{$A_-^{(2)}$}}
		\end{overpic} 
		
		\vspace{1em}
		$\Aq \dtmor 2$ \\ Type II 
	\end{minipage}\hspace{0.7em}
	\begin{minipage}[t]{.23\textwidth} 
		\centering
		\begin{overpic}{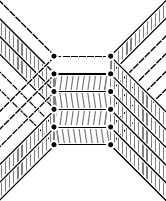} 
			\put (45,75) {\scalebox{0.7}{$A_+^{(3)}$}}
			\put (24,18) {\scalebox{0.7}{$A_-^{(3)}$}}
		\end{overpic}
		
		\vspace{1em}  
		$\Aq \dtmor 3$ \\Type III
	\end{minipage}

	\vspace{1em}

	In both subcases, the range of admissible values is $0 \le K \le \min(k_2 - 1, \, |A| - k_5)$, and each value of $K$ produces a distinct morphism for each base tree $T_S$. Since the three trees $T_{2,3}, T_{2,4}, T_{2,5}$ yield the three distinct combinatorial types $H_{2,3}, H_{2,4}, H_{2,5}$, there are exactly $\min(k_2 - 1, \, |A| - k_5) + 1$ full-dimensional morphisms of Type~I, Type~II, and Type~III, respectively.

	\subsection{Valency-3 limits: Case \{v3-nd4\}} \label{subsec-case-v3}
	Let $\vale w_0 = 3$. The contracting edge $t_1$ has endpoints $u$ and $v$ with $\vale u + \vale v = 5$. Since $t_1$ is an internal edge of $\Aq T q$, we have $\vale u = 2$ and $\vale v = 3$. For clarity, we denote the vertices $A^{(q)}_1$ and $A^{(q)}_2$ above $u$ and $v$ by $\Aq A q _u$ and $\Aq A q _v$, respectively. 
	
	By Lemma~\ref{lemma-above-w0}, $r_0(A) = 4 - \vale w_0 = 1$, with $r_{\Aq \dtmor q}(\Aq A q _u) = 1$ and $r_{\Aq \dtmor q}(\Aq A q _v) = 0$. Applying the non-dangling excess formula (Lemma~\ref{lem-rphi-nd}) to $A$ in $G_0$ gives
	\[
	r_0(A) = 2|A| + 4 - 2 - (k_2 + k_3 + k_4 + k_5) = 2|A| + 2 - (k_2 + k_3 + k_4 + k_5).
	\]
	Since $r_0(A) = 1$, we obtain identity~\eqref{eq-tri}:
	\begin{align} \label{eq-tri} \tag{$\boxplus$} %
	k_2 + k_3 + k_4 + k_5 = 2|A| + 1. 
	\end{align} 
	By the no-return condition, the four non-dangling edges $e_i$ incident to $A$ in $G_0$ must lie above at least two distinct edges of $T_0$. If they lay above at most two edges, the sum of their cardinalities would satisfy $\sum k_i \le 2|A|$, contradicting~\eqref{eq-tri}. Thus, the four edges lie above all three edges of $T_0$ incident to $w_0$: exactly two edges lie above the same edge of $T_0$, and the other two lie above distinct edges.
	We label the edges of $T_0$ and $G_0$ such that:
	$e_2$ and $e_5$ are above $t_2$ with $k_2 \le k_5$;
	$e_3$ is above $t_3$;
	$e_4$ is above $t_4$ with $k_3 \le k_4$.
	
	Since $e_2$ and $e_5$ are above the same edge $t_2$, we have $|A| \ge k_2 + k_5$. Moreover, since $k_5 \ge k_2$ and $k_4 \ge k_3$, identity~\eqref{eq-tri} implies
	\[
	2(k_4 + k_5) \ge (k_2 + k_3) + (k_4 + k_5) = 2|A| + 1,
	\]
	which forces $k_4 + k_5 \ge |A| + 1$.
	
	There are three choices for the base tree $\Aq T q$, determined by which edge incident to $w_0$ attaches to the divalent vertex $u$: namely $T_2, T_3$, or $T_4$. We show that the base tree and the cardinalities $k_i$ completely determine $|\Aq A q _u|$, $|\Aq A q _v|$, and $\Aq k q _1$:

	\paragraph*{Base tree $\Aq T q = T_2$:}
	Here $t_2$ is incident to $u$, so $e_2$ and $e_5$ lie above the edge incident to $u$. By Case~(r1) of Proposition~\ref{prop-local} applied to $\Aq A q _u$, we have $|\Aq A q _u| = \Aq k q _1 = k_2 + k_5$. The vertex $\Aq A q _v$ is incident to $\Aq e q _1$ (above $t_1$), $e_3$ (above $t_3$), and $e_4$ (above $t_4$). Applying Case~(r0) of Proposition~\ref{prop-local} to $\Aq A q _v$ gives
	\[
	2|\Aq A q _v| + 1 = k_3 + k_4 + \Aq k q _1 = k_3 + k_4 + k_2 + k_5 = 2|A| + 1,
	\]
	so $|\Aq A q _v| = |A|$. The edges meeting at $\Aq A q _u$ are $h_2$ and $h_5$, so this morphism has Type~III ($H_{2,5}$).

	\paragraph*{Base trees $\Aq T q = T_\alpha$ ($\alpha \in \{3, 4\}$):}
	Here $t_\alpha$ is incident to $u$, while $t_2$ and the remaining edge $t_\beta$ (where $\{\alpha, \beta\} = \{3, 4\}$) are incident to $v$. Applying Case~(r1) to $\Aq A q _u$ shows that $\Aq A q _u$ is incident to $\Aq e q _\alpha$ (above $t_\alpha$), $\Aq e q _1$ (above $t_1$), and a second edge $e'$ above $t_1$, satisfying $k_\alpha = \Aq k q _1 + |e'|$. The edge $e'$ connects $\Aq A q _u$ to a vertex $A'$ above $v$ with $\nddeg A' = 2$. By Case~(r0), $A'$ is incident to an edge $\Aq e q _\delta$ above $t_2$ (with $\delta \in \{2, 5\}$), so $|e'| = |\Aq e q _\delta| = k_\delta$. This gives
	\[
	\Aq k q _1 = k_\alpha - k_\delta, \quad \text{requiring} \quad k_\alpha > k_\delta.
	\]
	Let $\gamma \in \{2, 5\} \setminus \{\delta\}$, so $\{\alpha, \beta, \gamma, \delta\} = \{2, 3, 4, 5\}$. The vertex $\Aq A q _v$ is incident to $\Aq e q _1$, $\Aq e q _\beta$, and $\Aq e q _\gamma$. Case~(r0) applied to $\Aq A q _v$ gives
	\[
	2|\Aq A q _v| + 1 = k_\beta + k_\gamma + \Aq k q _1 = k_\beta + k_\gamma + k_\alpha - k_\delta = (2|A| + 1) - 2k_\delta,
	\]
	so $|\Aq A q _v| = |A| - k_\delta$.

	For $\Aq A q _v$ to be a valid vertex, its cardinality must be at least that of each incident edge: $|\Aq A q _v| \ge \Aq k q _1 = k_\alpha - k_\delta$ (giving $|A| \ge k_\alpha$), $|\Aq A q _v| \ge k_\gamma$ (giving $|A| \ge k_\gamma + k_\delta = k_2 + k_5$), and $|\Aq A q _v| \ge k_\beta$ (giving $|A| \ge k_\beta + k_\delta$). The first two inequalities always hold. By~\eqref{eq-tri}, the third inequality $|A| \ge k_\beta + k_\delta$ is equivalent to
	\[
	k_\alpha + k_\gamma \ge |A| + 1.
	\]
	Whenever this inequality holds, we also have $k_\alpha \ge |A| + 1 - k_\gamma \ge (k_\gamma + k_\delta) + 1 - k_\gamma = k_\delta + 1 > k_\delta$, so the condition $k_\alpha > k_\delta$ is automatically satisfied. Thus, the existence of a morphism with associated pair $(\alpha, \delta)$ is equivalent to $|A| \ge k_\beta + k_\delta$ (or $k_\alpha + k_\gamma \ge |A| + 1$).
	
	The resulting combinatorial type has $\{h_\alpha, h_\delta\}$ meeting at $\Aq A q _u$ and $\{h_\beta, h_\gamma\}$ meeting at $\Aq A q _v$. When $\delta = 2$, this contains the edge $h_2$, giving type $H_{2, \alpha}$ (Type~I if $\alpha = 3$, Type~II if $\alpha = 4$); when $\delta = 5$, we have $\gamma = 2$, so the pair containing $h_2$ is $\{h_\beta, h_2\}$, giving type $H_{2, \beta}$ (Type~I if $\beta = 3$, Type~II if $\beta = 4$).

	The four possible pairs $(\alpha, \delta) \in \{3, 4\} \times \{2, 5\}$ are:
	\begin{itemize}
		\item $(4, 2)$: Here $\beta = 3, \gamma = 5$. The condition $k_4 + k_5 \ge |A| + 1$ is always satisfied, yielding a morphism of Type~II.
		\item $(3, 5)$: Here $\beta = 4, \gamma = 2$. The condition $k_3 + k_2 \ge |A| + 1$ never holds, since $k_3 \le k_4$ and $k_2 \le k_5$ would make $\sum k_i \ge 2|A| + 2$, contradicting~\eqref{eq-tri}.
		\item $(3, 2)$: Here $\beta = 4, \gamma = 5$. The condition $k_4 + k_2 \le |A|$ yields a morphism of Type~I.
		\item $(4, 5)$: Here $\beta = 3, \gamma = 2$. The condition $k_4 + k_2 \ge |A| + 1$ yields a morphism of Type~I.
	\end{itemize}
	The conditions for $(3, 2)$ and $(4, 5)$ are complementary and mutually exclusive, so exactly one morphism of Type~I is realized in every case.

	\paragraph*{Subcase \{v3-nd4-1\} ($k_4 + k_2 \le |A|$):}
	Type~I is realized by the pair $(3,2)$. See the diagrams below.
	
	\vspace{1em}
	\noindent  
	\begin{minipage}[t]{.23\textwidth}
		\centering
		\begin{overpic}{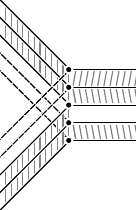} 
			\put (32,26) {\scalebox{0.8}{$w_0$}} 
			\put (32,70) {\scalebox{0.8}{$A$}} 
			\put (-9,7) {\scalebox{0.8}{$e_4$}}
			\put (-9,92) {\scalebox{0.8}{$e_3$}}
			\put (66,57) {\scalebox{0.8}{$e_5$}}
			\put (66,36) {\scalebox{0.8}{$e_2$}}		 
		\end{overpic}
		
		\vspace{1em}  
		$\dtmor_0$
	\end{minipage}	
	\begin{minipage}[t]{.23\textwidth} 
		\centering
		\begin{overpic}{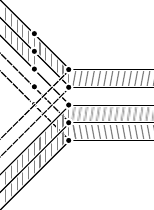} 
			\put (29,23) {\scalebox{0.8}{$A_v^{(1)}$}}
			\put (13,88) {\scalebox{0.8}{$A_u^{(1)}$}}
		\end{overpic}
		
		\vspace{1em}
		$\Aq \dtmor 1$\\Type I
	\end{minipage}\hspace{0.7em}
	\begin{minipage}[t]{.23\textwidth}
		\centering     
		\begin{overpic}{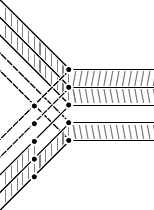} 
			\put (16,8) {\scalebox{0.8}{$A_u^{(2)}$}}
			\put (32,70) {\scalebox{0.8}{$A_v^{(2)}$}}
		\end{overpic}
		
		\vspace{1em}
		$\Aq \dtmor 2$ \\Type II
	\end{minipage}\hspace{0.7em}
	\begin{minipage}[t]{.23\textwidth} 
		\centering
		\begin{overpic}{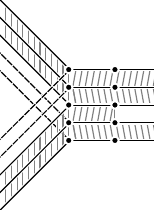}
			\put (50,22) {\scalebox{0.8}{$A_u^{(3)}$}}
			\put (31,70) {\scalebox{0.8}{$A_v^{(3)}$}}
		\end{overpic} 
		
		\vspace{1em}  
		$\Aq \dtmor 3$ \\Type III
	\end{minipage}
	
	\vspace{1em}

	\paragraph*{Subcase \{v3-nd4-2\} ($k_4 + k_2 \ge |A| + 1$):}
	Type~I is realized by the pair $(4,5)$. See the diagrams below.
	
	\vspace{1em}
	\noindent  
	\begin{minipage}[t]{.23\textwidth}
		\centering
		\begin{overpic}{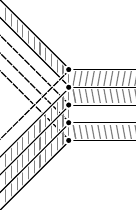} 
			\put (28.5,22) {\scalebox{0.8}{$w_0$}} 
			\put (29,76) {\scalebox{0.8}{$A$}} 
			\put (-9,10) {\scalebox{0.8}{$e_4$}}
			\put (-9,95) {\scalebox{0.8}{$e_3$}}
			\put (66,56) {\scalebox{0.8}{$e_5$}}
			\put (66,36) {\scalebox{0.8}{$e_2$}}		 
		\end{overpic}
		
		\vspace{1em}  
		$\dtmor_0$
	\end{minipage}	
	\begin{minipage}[t]{.23\textwidth} 
		\centering
		\begin{overpic}{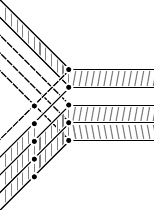} 
			\put (14,8) {\scalebox{0.7}{$A_u^{(1)}$}}
			\put (33,69) {\scalebox{0.7}{$A_v^{(1)}$}}
		\end{overpic}
		
		\vspace{1em}
		$\Aq \dtmor 1$\\Type I
	\end{minipage}\hspace{0.7em}
	\begin{minipage}[t]{.23\textwidth}
		\centering     
		\begin{overpic}{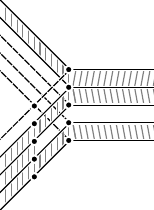} 
			\put (14,8) {\scalebox{0.7}{$A_u^{(2)}$}}
			\put (32,69) {\scalebox{0.7}{$A_v^{(2)}$}}
		\end{overpic} 
		
		\vspace{1em}
		$\Aq \dtmor 2$ \\Type II
	\end{minipage}\hspace{0.7em}
	\begin{minipage}[t]{.23\textwidth} 
		\centering
		\begin{overpic}{\figsdir/217.pdf}
			\put (50,22) {\scalebox{0.7}{$A_u^{(3)}$}}
			\put (31,70) {\scalebox{0.7}{$A_v^{(3)}$}}
		\end{overpic} 
		
		\vspace{1em}  
		$\Aq \dtmor 3$ \\Type III
	\end{minipage}

	\vspace{1em}

	Thus, in both subcases there is exactly one full-dimensional morphism of Type~I, Type~II, and Type~III, respectively.

	\subsection{Valency-2 limits: Case \{v2-nd4\}} \label{subsec-case-v2}
	Let $\vale w_0 = 2$. The setup parallels the valency-2 case in \cite[Section~\ref{I-sub-deg2nd3}]{dv20}: we say $\Aq \dtmor q$ has Base~I or Base~II according to whether its base tree is $T_{\varnothing}$ (subdividing the edge with two divalent vertices $u, v$) or $T_2$ (introducing a trivalent branch point), respectively. Base~I corresponds to a partition $[e_\alpha, e_\beta; e_\gamma, e_\delta]$ of the four non-dangling edges incident to $A$ with $|e_\alpha| = |e_\beta|$ and $|e_\gamma| = |e_\delta|$. Base~II corresponds to a transition between the relations above $t_2$ and $t_3$ via two local changes (each being a class splitting or merging). In both bases, $r_0(A) = 4 - \vale w_0 = 2$ by Lemma~\ref{lemma-above-w0}.
	
	By the no-return condition, the four non-dangling edges incident to $A$ in $G_0$ lie above the two edges $t_2, t_3$ of $T_0$ incident to $w_0$. There are two possible distributions:
	\begin{enumerate}[(A)]
		\item \textbf{Two edges above $t_2$, two edges above $t_3$:} Case~\{v2-nd4-t3\}.
		\item \textbf{Three edges above $t_2$, one edge above $t_3$:} Case~\{v2-nd4-t2\}.
	\end{enumerate}
	We label the edges of $G_0$ and $T_0$ such that $\min(k_2, k_3, k_4, k_5) = k_2$, the edges $e_3, e_4$ lie above $t_2$ with $k_3 \le k_4$, and $e_5$ lies above $t_3$.

	\subsubsection*{Configuration A: Two edges above $t_2$, two edges above $t_3$ (Case \{v2-nd4-t3\})}
	Here $e_2$ lies above $t_3$, so $e_3, e_4$ lie above $t_2$ and $e_2, e_5$ lie above $t_3$. Applying Lemma~\ref{lem-rphi-nd} to $A$ gives
	\[
	r_0(A) = 2 = 2 + 2|A| - (k_2 + k_3 + k_4 + k_5),
	\]
	whence $k_2 + k_3 + k_4 + k_5 = 2|A|$. Since $k_3 + k_4 \le |A|$ and $k_2 + k_5 \le |A|$, both must hold with equality:
	\[
	k_3 + k_4 = |A| \quad \text{and} \quad k_2 + k_5 = |A|.
	\]
	Because $k_2 = \min(k_i)$ and $k_3 \le k_4$, we obtain the chain of inequalities
	\[
	k_2 \le k_3 \le k_4 = |A| - k_3 \le |A| - k_2 = k_5.
	\]
	These inequalities determine which transitions occur for Base~I and Base~II. Depending on the equalities among $k_2, k_3, k_4$, there are three subcases:

	\paragraph*{Subcase \{v2-nd4-t3-k2=k4\} ($k_2 = k_4$):}
	Here $k_2 = k_3 = k_4 = k_5 = |A|/2$, so all four edges have identical cardinality. There are two morphisms with Base~I, corresponding to $[e_3, e_2; e_4, e_5]$ and $[e_3, e_5; e_4, e_2]$; these yield $\Aq \dtmor 1$ (Type~I) and $\Aq \dtmor 2$ (Type~II), respectively. There is a unique morphism with Base~II, where two classes merge above $v$, yielding $\Aq \dtmor 3$ (Type~III). See the diagrams below.

	\vspace{1.5em}
	\noindent \begin{minipage}[t]{.21\textwidth}
		\centering
		\begin{overpic}{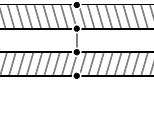}    
			\put (23,20) {\scalebox{0.9}{$t_2$}}
			\put (72,20) {\scalebox{0.9}{$t_3$}}
			\put (45,20) {\scalebox{0.9}{$w_0$}}
			
			\put (-12,69) {\scalebox{0.9}{$e_3$}}	
			\put (-12,36) {\scalebox{0.9}{$e_4$}}	
			\put (101,69) {\scalebox{0.9}{$e_2$}}
			\put (101,36) {\scalebox{0.9}{$e_5$}}
			
			\put (45,82) {\scalebox{1}{$A$}}
		\end{overpic} 
		
		\vspace{1em} 
		$\dtmor_0$
	\end{minipage} \hspace{1em} 
	\begin{minipage}[t]{.21\textwidth}
		\centering 
		\begin{overpic}[scale=1]{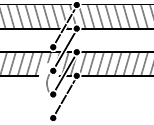}      
			\put (44,7) {\scalebox{0.9}{$t_1$}}   
			\put (30,-5) {\scalebox{0.9}{$u$}}
			\put (52,22) {\scalebox{0.9}{$v$}}
			\put (45,83) {\scalebox{0.9}{$A_1^{(1)}$}}
			\put (52,51) {\scalebox{0.77}{$A_2^{(1)}$}}
			
			\put (100,69) {\scalebox{0.9}{$e_2$}}
			\put (100,36) {\scalebox{0.9}{$e_5$}}	
		\end{overpic}
		
		\vspace{1em}
		$\Aq \dtmor 1$ \\ Type I
	\end{minipage}\hspace{1em}  
	\begin{minipage}[t]{.21\textwidth} 
		\centering
		\begin{overpic}[scale=1]{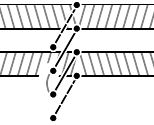}    
			\put (44,7) {\scalebox{0.9}{$t_1$}}   
			\put (30,-5) {\scalebox{0.9}{$u$}}
			\put (52,22) {\scalebox{0.9}{$v$}}
			\put (45,83) {\scalebox{0.9}{$A_2^{(2)}$}}
			\put (52,51) {\scalebox{0.77}{$A_1^{(2)}$}}
			
			\put (100,69) {\scalebox{0.9}{$e_5$}}
			\put (100,36) {\scalebox{0.9}{$e_2$}}		
		\end{overpic}  	
		
		\vspace{1em}
		$\Aq \dtmor 2$ \\ Type II
	\end{minipage}\hspace{1em}
	\begin{minipage}[t]{.21\textwidth} 
		\centering
		\begin{overpic}[scale=1]{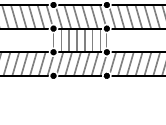}
			\put (45,19) {\scalebox{0.9}{$t_1$}}   
			\put (27,19) {\scalebox{0.9}{$u$}}
			\put (63,19) {\scalebox{0.9}{$v$}}
			\put (27,77) {\scalebox{0.9}{$A_2^{(3)}$}}
			\put (62,77) {\scalebox{0.9}{$A_1^{(3)}$}}
			
			\put (100,65) {\scalebox{0.9}{$e_2$}}
			\put (100,32) {\scalebox{0.9}{$e_5$}}	
		\end{overpic}
		
		\vspace{1em}
		$\Aq \dtmor 3$ \\ Type III
	\end{minipage}

	\vspace{1.5em}

	\paragraph*{Subcase \{v2-nd4-t3-k2=k3\} ($k_2 = k_3 < k_4$):}
	Here $k_2 = k_3 < k_4 = k_5$. There is a unique morphism with Base~I, corresponding to $[e_3, e_2; e_4, e_5]$, which gives $\Aq \dtmor 1$ (Type~I). There are two morphisms with Base~II: the class $e_4$ splits above $u$ because $k_4 > k_2$, yielding $\Aq \dtmor 2$ (Type~II); or two classes merge above $u$, yielding $\Aq \dtmor 3$ (Type~III). See the diagrams below.

	\vspace{1.5em}
	\noindent \begin{minipage}[t]{.21\textwidth}
		\centering
		\begin{overpic}{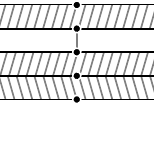}    
			\put (23,20) {\scalebox{0.9}{$t_2$}}
			\put (72,20) {\scalebox{0.9}{$t_3$}}
			\put (45,20) {\scalebox{0.9}{$w_0$}}
			
			\put (-12,83) {\scalebox{0.9}{$e_3$}}	
			\put (-12,45) {\scalebox{0.9}{$e_4$}}	
			\put (101,83) {\scalebox{0.9}{$e_2$}}
			\put (101,45) {\scalebox{0.9}{$e_5$}}
			
			\put (45,98) {\scalebox{1}{$A$}}
		\end{overpic} 
		
		\vspace{1em}
		$\dtmor_0$
	\end{minipage} \hspace{1em} 
	\begin{minipage}[t]{.21\textwidth}
		\centering 
		\begin{overpic}[scale=1]{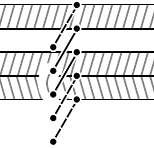}      
			\put (43,10) {\scalebox{0.9}{$t_1$}}   
			\put (31,-5) {\scalebox{0.9}{$u$}}
			\put (52,23) {\scalebox{0.9}{$v$}}
			\put (50,66) {\scalebox{0.8}{$A_2^{(1)}$}}
			\put (48,99) {\scalebox{0.9}{$A_1^{(1)}$}} 
		\end{overpic}
		
		\vspace{1em}
		$\Aq \dtmor 1$ \\ Type I
	\end{minipage}\hspace{1em}  
	\begin{minipage}[t]{.21\textwidth} 
		\centering
		\begin{overpic}[scale=1]{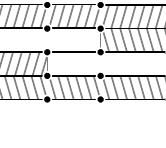}    
			\put (42,20) {\scalebox{0.9}{$t_1$}}   
			\put (25,20) {\scalebox{0.9}{$u$}}
			\put (61,20) {\scalebox{0.9}{$v$}}
			\put (23,61.5) {\scalebox{0.77}{$A_1^{(2)}$}}
			\put (60,91) {\scalebox{0.9}{$A_2^{(2)}$}}		
			\put (101,71) {\scalebox{0.9}{$e_5$}}
			\put (101,34) {\scalebox{0.9}{$e_2$}}
		\end{overpic}  	
		
		\vspace{1em}
		$\Aq \dtmor 2$ \\ Type II
	\end{minipage}\hspace{1em}
	\begin{minipage}[t]{.21\textwidth} 
		\centering
		\begin{overpic}[scale=1]{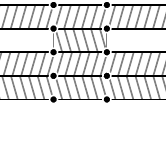}
			\put (45,20) {\scalebox{0.9}{$t_1$}}   
			\put (29,20) {\scalebox{0.9}{$u$}}
			\put (63,20) {\scalebox{0.9}{$v$}}
			\put (28,91) {\scalebox{0.9}{$A_2^{(3)}$}}
			\put (62,91) {\scalebox{0.9}{$A_1^{(3)}$}}
		\end{overpic}
		
		\vspace{1em}
		$\Aq \dtmor 3$ \\ Type III
	\end{minipage}

	\vspace{1.5em}

	\paragraph*{Subcase \{v2-nd4-t3-k2<k3\} ($k_2 < k_3 < k_4$):}
	Here $k_2 < k_3 < k_4 < k_5$, so all four cardinalities are strictly distinct. Since no two incident edges on opposite sides have equal cardinality, there are no morphisms with Base~I. There are three morphisms with Base~II: the class $e_4$ splits above $u$ because $k_4 > k_2$, yielding $\Aq \dtmor 1$ (Type~I); the class $e_3$ splits above $u$ because $k_3 > k_2$, yielding $\Aq \dtmor 2$ (Type~II); or two classes merge above $u$, yielding $\Aq \dtmor 3$ (Type~III). See the diagrams below.

	\vspace{1.5em}
	\noindent \begin{minipage}[t]{.21\textwidth}
		\centering
		\begin{overpic}{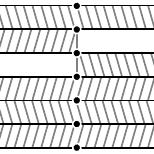}    
			\put (23,-6) {\scalebox{0.9}{$t_2$}}
			\put (72,-6) {\scalebox{0.9}{$t_3$}}
			\put (45,-6) {\scalebox{0.9}{$w_0$}}
			
			\put (101,87) {\scalebox{0.9}{$e_2$}}	 
			\put (101,32) {\scalebox{0.9}{$e_5$}}	
			\put (-12,84) {\scalebox{0.9}{$e_3$}}
			\put (-12,25) {\scalebox{0.9}{$e_4$}}
			
			\put (45,100) {\scalebox{1}{$A$}}
		\end{overpic} 
		
		\vspace{1em}
		$\dtmor_0$
	\end{minipage} \hspace{1em} 
	\begin{minipage}[t]{.21\textwidth} 
		\centering 
		\begin{overpic}[scale=1]{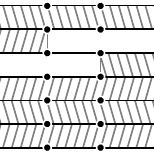}  
			\put (45,-5) {\scalebox{0.9}{$t_1$}}   
			\put (27,-5) {\scalebox{0.9}{$u$}}
			\put (63,-5) {\scalebox{0.9}{$v$}}
			\put (62,69.5) {\scalebox{0.8}{$A_2^{(1)}$}}
			\put (27,101) {\scalebox{0.9}{$A_1^{(1)}$}}   
		\end{overpic}
		
		\vspace{1em}
		$\Aq \dtmor 1$ \\ Type I
	\end{minipage}\hspace{1em}  
	\begin{minipage}[t]{.21\textwidth} 
		\centering
		\begin{overpic}[scale=1]{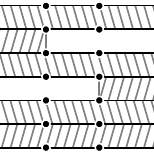}    
			\put (45,-5) {\scalebox{0.9}{$t_1$}}   
			\put (27,-5) {\scalebox{0.9}{$u$}}
			\put (63,-5) {\scalebox{0.9}{$v$}}
			\put (62,69.5) {\scalebox{0.8}{$A_2^{(2)}$}}
			\put (27,101) {\scalebox{0.9}{$A_1^{(2)}$}}		
		\end{overpic}  	
		
		\vspace{1em}
		$\Aq \dtmor 2$ \\ Type II
	\end{minipage}\hspace{1em}
	\begin{minipage}[t]{.21\textwidth} 
		\centering
		\begin{overpic}[scale=1]{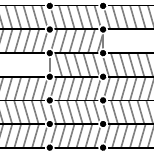}
			\put (45,-5) {\scalebox{0.9}{$t_1$}}   
			\put (27,-5) {\scalebox{0.9}{$u$}}
			\put (63,-5) {\scalebox{0.9}{$v$}}
			\put (27,101) {\scalebox{0.9}{$A_2^{(3)}$}}
			\put (62,101) {\scalebox{0.9}{$A_1^{(3)}$}}
		\end{overpic}
		
		\vspace{1em}
		$\Aq \dtmor 3$ \\ Type III
	\end{minipage}

	\vspace{1.5em}

	\subsubsection*{Configuration B: Three edges above $t_2$, one edge above $t_3$ (Case \{v2-nd4-t2\})}
	\paragraph*{Case \{v2-nd4-t2\} ($e_2$ above $t_2$):}
	Here $e_2, e_3, e_4$ lie above $t_2$, and $e_5$ lies above $t_3$. Since $e_5$ is the unique edge above $t_3$, we have $|A| = k_5$, which forces $k_2 \le k_3 \le k_4 < k_5$. Applying Lemma~\ref{lem-rphi-nd} to $A$ gives
	\[
	r_0(A) = 2 = 2 + 2|A| - (k_2 + k_3 + k_4 + k_5),
	\]
	whence $|A| = k_5 = k_2 + k_3 + k_4$.
	
	Since only one edge lies above $t_3$, no Base~I morphism exists. There are three morphisms with Base~II, corresponding to the three ways two classes can merge above $u$:
	$e_2$ and $e_3$ merge (yielding $\Aq \dtmor 1$ of Type~I);
	$e_2$ and $e_4$ merge (yielding $\Aq \dtmor 2$ of Type~II);
	or $e_3$ and $e_4$ merge (yielding $\Aq \dtmor 3$ of Type~III).
	In each case, a second merge above $v$ combines the result with the remaining class to produce $e_5$. See the diagrams below.

	\vspace{1.5em} 
	\noindent \begin{minipage}[t]{.21\textwidth}
		\centering
		\begin{overpic}{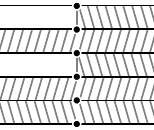}    
			\put (23,-6) {\scalebox{0.9}{$t_2$}}
			\put (72,-6) {\scalebox{0.9}{$t_3$}}
			\put (45,-6) {\scalebox{0.9}{$w_0$}}
			
			\put (-12,79) {\scalebox{0.9}{$e_2$}}
			\put (-12,56) {\scalebox{0.9}{$e_3$}}	
			\put (-12,20) {\scalebox{0.9}{$e_4$}}	
			\put (101,40) {\scalebox{0.9}{$e_5$}}
			
			\put (45,85) {\scalebox{1}{$A$}}
		\end{overpic} 
		
		\vspace{1em}
		$\dtmor_0$ 
	\end{minipage} \hspace{1em} 
	\begin{minipage}[t]{.21\textwidth} 
		\centering
		\begin{overpic}[scale=1]{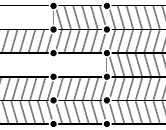}
			\put (45,-5) {\scalebox{0.9}{$t_1$}}   
			\put (27,-5) {\scalebox{0.9}{$u$}}
			\put (63,-5) {\scalebox{0.9}{$v$}}
			\put (27,80) {\scalebox{0.9}{$A_1^{(1)}$}}
			\put (62,80) {\scalebox{0.9}{$A_2^{(1)}$}} 
		\end{overpic}
		
		\vspace{1em}
		$\Aq \dtmor 1$ \\ Type I
	\end{minipage} \hspace{1em}  
	\begin{minipage}[t]{.21\textwidth} 
		\centering
		\begin{overpic}[scale=1]{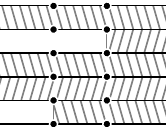}    
			\put (45,-5) {\scalebox{0.9}{$t_1$}}   
			\put (27,-5) {\scalebox{0.9}{$u$}}
			\put (63,-5) {\scalebox{0.9}{$v$}}
			\put (27,50.5) {\scalebox{0.77}{$A_1^{(2)}$}}
			\put (62,80) {\scalebox{0.9}{$A_2^{(2)}$}} 		
		\end{overpic}  	
		
		\vspace{1em}
		$\Aq \dtmor 2$ \\ Type II
	\end{minipage}\hspace{1em}
	\begin{minipage}[t]{.21\textwidth}
		\centering 
		\begin{overpic}[scale=1]{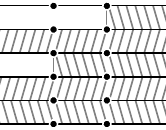}  
			\put (45,-5) {\scalebox{0.9}{$t_1$}}   
			\put (27,-5) {\scalebox{0.9}{$u$}}
			\put (63,-5) {\scalebox{0.9}{$v$}}
			\put (27,64.5) {\scalebox{0.77}{$A_2^{(3)}$}}
			\put (62,80) {\scalebox{0.9}{$A_1^{(3)}$}} 
		\end{overpic}
		
		\vspace{1em}
		$\Aq \dtmor 3$ \\ Type III 
	\end{minipage}

	\vspace{1.5em}

	Thus, in every subcase of $\vale w_0 = 2$, there is exactly one full-dimensional morphism of Type~I, Type~II, and Type~III specializing to $\dtmor_0$.

	\paragraph*{Conclusion of the proof of Proposition~\ref{prop-signed-mult}(2):}
	Combining the results of Subsections~\ref{subsec-case-v4}, \ref{subsec-case-v3}, and~\ref{subsec-case-v2}, we conclude that for every codimension-1 limit $\dtmor_0$ with non-trivalent combinatorial type $H_0$:
	\begin{itemize}
		\item If $\vale w_0 = 4$, there are $\min(k_2 - 1, \, |A| - k_5) + 1$ full-dimensional morphisms specializing to $\dtmor_0$ of Type~I, Type~II, and Type~III, respectively.
		\item If $\vale w_0 = 3$, there is exactly 1 full-dimensional morphism specializing to $\dtmor_0$ of Type~I, Type~II, and Type~III, respectively.
		\item If $\vale w_0 = 2$, there is exactly 1 full-dimensional morphism specializing to $\dtmor_0$ of Type~I, Type~II, and Type~III, respectively.
	\end{itemize}
	In all cases, the number of full-dimensional morphisms specializing to $\dtmor_0$ is the same for each of the three combinatorial types. By Lemma~\ref{lm:change-comb-type}, their signed multiplicities are equal. Therefore, the weighted sum of full-dimensional morphisms of each combinatorial type across the wall $\mH_0$ is balanced, establishing Proposition~\ref{prop-signed-mult}(2). This completes the proof of Theorem~\ref{thm} and Corollary~\ref{theorem-gonality}.

\appendix
	\section{Trivalent deformation} \label{appendix-trivalent-deformation}
	
	For the convenience of the reader we review the constructions necessary to prove the balancing condition for $\Pi$ when $\dtmor_0$ is  codimension-1 and $H(\dtmor_0)$ is trivalent. We just list the diagrams, for the proof that these are exhaustive see the last section of \cite{dv20}.
	
	\begin{itemize}[leftmargin=*]
		\item Case \{w3-r1-nd3-t2-(a=k4)\}: 
		
		\noindent\begin{minipage}[t]{.18\textwidth}
			\centering
			\begin{overpic}[scale=0.6]{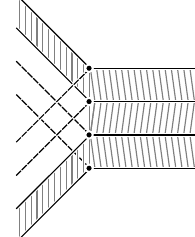}  
				\put (37,74) {\scalebox{0.75}{$A_0$}}
				\put (60,19) {\scalebox{0.75}{$t_4$}}
				\put (26,84) {\scalebox{0.75}{$t_2$}}
				\put (26,8) {\scalebox{0.75}{$t_3$}}
				\put (35,19) {\scalebox{0.75}{$w_0$}}
				
				\put (-2,93) {\scalebox{0.75}{$e_2$}}
				\put (-2,3) {\scalebox{0.75}{$e_3$}} 
				\put (83,48) {\scalebox{0.75}{$e_4$}}  
			\end{overpic}
			
			\vspace{0.5em}
			$\dtmor_0$
		\end{minipage}\hfill
		\begin{minipage}[t]{.18\textwidth}
			\centering
			\begin{overpic}[scale=0.6]{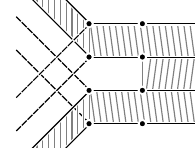}  
				\put (43,4) {\scalebox{0.75}{$v$}}
				\put (70,4) {\scalebox{0.75}{$u$}}
				\put (56,4) {\scalebox{0.75}{$t_1$}}
				\put (70,67) {\scalebox{0.75}{$A^{(1)}$}}  
			\end{overpic}
			
			\vspace{0.5em}
			$\Aq \dtmor 1$
		\end{minipage}\hfill
		\begin{minipage}[t]{.18\textwidth}
			\centering
			\begin{overpic}[scale=0.6]{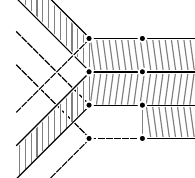} 
				\put (44,12) {\scalebox{0.75}{$v$}}
				\put (70,12) {\scalebox{0.75}{$u$}}
				\put (56,12) {\scalebox{0.75}{$t_1$}}
				\put (43,75) {\scalebox{0.75}{$A^{(2)}$}}
				
				\put (14,93) {\scalebox{0.75}{$e_2$}}
				\put (-2,6) {\scalebox{0.75}{$e_3$}} 
				\put (101,44) {\scalebox{0.75}{$e_4$}}   
			\end{overpic}
			
			\vspace{0.5em}
			$\Aq \dtmor 2$
		\end{minipage}\hfill
		\begin{minipage}[t]{.18\textwidth}
			\centering
			\begin{overpic}[scale=0.6]{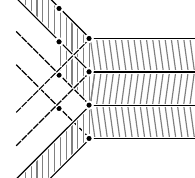} 
				\put (30,89) {\scalebox{0.75}{$u$}}
				\put (46,73) {\scalebox{0.75}{$v$}}
				\put (38,81) {\scalebox{0.75}{$t_1$}}
				\put (41,10) {\scalebox{0.75}{$A^{(3)}$}} 
			\end{overpic}
			
			\vspace{0.5em}
			$\Aq \dtmor 3$
		\end{minipage}\hfill
		\begin{minipage}[t]{.18\textwidth}
			\centering
			\begin{overpic}[scale=0.6]{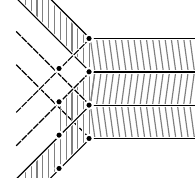}
				\put (31,0) {\scalebox{0.75}{$u$}}
				\put (45,12) {\scalebox{0.75}{$v$}}
				\put (39,6) {\scalebox{0.75}{$t_1$}}
				\put (45,74) {\scalebox{0.75}{$A^{(4)}$}} 
			\end{overpic}
			
			\vspace{0.5em}
			$\Aq \dtmor 4$
		\end{minipage} 
		
		\begin{align*}
		\begin{matrix}
		\bbv_2 			& \bbv_3 		& \bbv_4 				&& \Aq {\bbv_1} 1 	& \Aq {\bbv_1} 2 		 	& \Aq {\bbv_1} 3 	& \Aq {\bbv_1} 4\\
		\frac 1 {k_2} 	& 0		 		& 0 					&& \frac 1 {k_2}	& 0  					 	& \frac 1 {k_2+1}	& 0\\
		0			 	& \frac 1 {k_3} & 0 					&& \frac 1 {k_3}    & 0						 	& 0					& \frac 1 {k_3+1}\\
		0			 	& 0		 		& \frac 1 {k_2 + k_3} 	&& 0  				& \frac 1 {k_2 + k_3 - 1}	& 0					& 0\\
		a_i				& b_i			& c_i					&& c_i				& c_i						& a_i				& b_i\\
		\end{matrix}
		\end{align*}
		
		\[\Aq {\bbv_1} 1 + (k_1+k_2-1)\Aq {\bbv_1} 2 + (k_2+1)\Aq {\bbv_1} 3 + (k_3+1)\Aq {\bbv_1} 4 = (k_2+1)\bbv_2 + (k_3+1)\bbv_3 + (k_2+k_3)\bbv_4  \]
		
		
		\item Case \{w3-r1-nd3-t2-(a>k4)\}: 
		
		\noindent\begin{minipage}{.3\textwidth}
			\centering
			\begin{overpic}[scale=0.9]{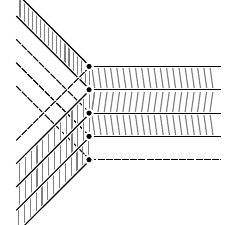} 
				\put (60,18) {\scalebox{1}{$t_4$}}
				\put (26,82) {\scalebox{1}{$t_2$}}
				\put (26,8) {\scalebox{1}{$t_3$}}
				\put (34,18) {\scalebox{1}{$w_0$}}
				\put (36,70) {\scalebox{1}{$A_0$}}
				
				\put (-2,79) {\scalebox{1}{$e_2$}}
				\put (-2,10) {\scalebox{1}{$e_3$}} 
				\put (92,51) {\scalebox{1}{$e_4$}}  
			\end{overpic}
			
			$\dtmor_0$
		\end{minipage}\hspace{0.1em}
		\begin{minipage}{.3\textwidth}
			\centering
			\begin{overpic}[scale=0.9]{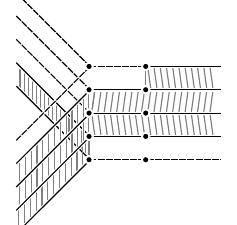} 
				\put (47,18) {\scalebox{1}{$t_1$}}
				\put (60,18) {\scalebox{1}{$u$}}
				\put (36,18) {\scalebox{1}{$v$}}
				\put (36,70) {\scalebox{1}{$A^{(1)}$}} 
			\end{overpic}
			
			$\Aq \dtmor 1$
		\end{minipage}\hspace{0.5em}
		\begin{minipage}{.3\textwidth}
			\centering
			\begin{overpic}[scale=0.9]{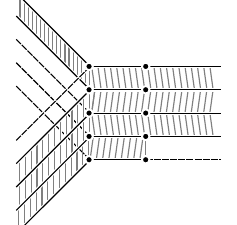} 
				\put (47,18) {\scalebox{1}{$t_1$}}
				\put (60,18) {\scalebox{1}{$u$}}
				\put (36,18) {\scalebox{1}{$v$}}
				\put (36,70) {\scalebox{1}{$A^{(2)}$}}  
			\end{overpic}
			
			$\Aq \dtmor 2$
		\end{minipage}\\
		
		\begin{align*}
		\begin{matrix}
		\bbv_4 			&& \Aq {\bbv_1} 1 	& \Aq {\bbv_1} 2\\
		0		 		&& 0 				&0 		  		\\
		0 	 			&& 0    			&0				\\
		\frac 1 {k_4}	&& \frac 1 {k_4 -1} &\frac 1 {k_4+1}\\
		a_i				&& a_i				& a_i			
		\end{matrix}
		\end{align*}
		
		\[(k_4-1)\Aq {\bbv_1} 1 + (k_4+1)\Aq {\bbv_1} 2  = 2k_4\bbv_4 \]


		\item Case \{w3-r1-nd3-t3\}: 
		
		\noindent		\begin{minipage}[t]{.3\textwidth}
			\centering 
			\begin{overpic}[scale=0.9]{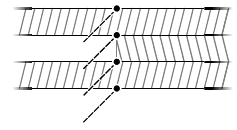}  
				\put (25,9) {\scalebox{1}{$t_3$}}
				\put (62,9) {\scalebox{1}{$t_4$}}
				\put (46,9) {\scalebox{1}{$w_0$}}
				\put (-5,44) {\scalebox{1}{$e_2$}}	
				\put (-5,21) {\scalebox{1}{$e_3$}}	
				\put (98,34) {\scalebox{1}{$e_4$}}	 
				\put (46,55) {\scalebox{1}{$A_0$}}
			\end{overpic}

			$\dtmor_0$
		\end{minipage}\hspace{0.5em}
		\begin{minipage}[t]{.3\textwidth} 
			\centering
			\begin{overpic}[scale=0.9]{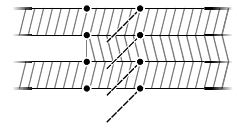}   
				\put (42,10) {\scalebox{1}{$t_1$}}
				\put (33,10) {\scalebox{1}{$u$}}
				\put (57,10) {\scalebox{1}{$v$}}
				\put (34,55) {\scalebox{1}{$A^{(1)}$}}
			\end{overpic}  
			
			$\Aq \dtmor 1$ 
		\end{minipage}\hspace{0.5em}
		\begin{minipage}[t]{.3\textwidth} 
			\centering
			\begin{overpic}[scale=0.9]{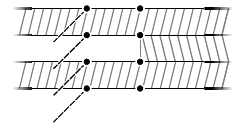}  
				\put (45,10) {\scalebox{1}{$t_1$}}
				\put (35,10) {\scalebox{1}{$v$}}
				\put (57,10) {\scalebox{1}{$u$}}
				\put (56,55) {\scalebox{1}{$A^{(2)}$}}
			\end{overpic}
			
			$\Aq \dtmor 2$
		\end{minipage}
		\begin{align*}
		\begin{matrix}
		\bbv_3 			&\bbv_4 				& \Aq {\bbv_1} 1 	  & \Aq {\bbv_1} 2 \\
		\frac 1 {k_2}	& 0						& 0				      & \frac 1 {k_2}  \\
		\frac 1 {k_3} 	& 0						& 0					  & \frac 1 {k_2}  \\
		0				& \frac 1 {k_2 + k_3}	& \frac 1 {k_2 + k_3} & 0			   \\
		a_i				& b_i					& a_i				  & b_i			
		\end{matrix}
		\end{align*}
		\[\Aq {\bbv_1} 1 + \Aq {\bbv_1} 2  = \bbv_3 + \bbv_4 \]

		
		\item Case \{w3-r1-nd2\}: 
		
		\noindent\begin{minipage}[t]{.3\textwidth} 
			\centering 
			\begin{overpic}[scale=0.9]{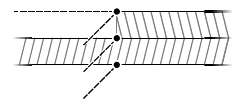} 
				\put (25,9) {\scalebox{1}{$t_3$}}
				\put (62,9) {\scalebox{1}{$t_4$}}
				\put (46,9) {\scalebox{1}{$w_0$}}
				\put (-2,20) {\scalebox{1}{$e$}}
				\put (98, 27) {\scalebox{1}{$e'$}}	
				\put (46,43) {\scalebox{1}{$A_0$}}
			\end{overpic}
			
			$\dtmor_0$
		\end{minipage}\hspace{0.5em}
		\begin{minipage}[t]{.3\textwidth}
			\centering
			\begin{overpic}[scale=0.9]{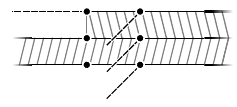}  
				\put (45,10) {\scalebox{1}{$t_1$}}
				\put (35,10) {\scalebox{1}{$u$}}
				\put (57,10) {\scalebox{1}{$v$}}
			\end{overpic}
			
			$\Aq \dtmor 1$
		\end{minipage}\hspace{0.5em}
		\begin{minipage}[t]{.30\textwidth} 
			\centering
			\begin{overpic}[scale=0.9]{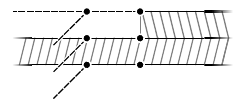} 
				\put (45,10) {\scalebox{1}{$t_1$}}
				\put (35,10) {\scalebox{1}{$v$}}
				\put (57,10) {\scalebox{1}{$u$}}
			\end{overpic}
			
			$\Aq \dtmor 2$
		\end{minipage}
		
		\begin{align*}
		\begin{matrix}
		\bbv_3 			&\bbv_4 				& \Aq {\bbv_1} 1 	  & \Aq {\bbv_1} 2 \\
		\frac 1 {k}		& \frac 1 {k+1}			& \frac 1 {k+1}		  & \frac 1 {k}  \\
		a_i				& b_i					& a_i				  & b_i			
		\end{matrix}
		\end{align*}
		\[\Aq {\bbv_1} 1 + \Aq {\bbv_1} 2  = \bbv_3 + \bbv_4 \]

		\item Case \{w2-r2-nd3-M-11\}:
		
		\noindent\begin{minipage}[t]{.23\textwidth}
			\centering 
			\begin{overpic}[scale=0.63]{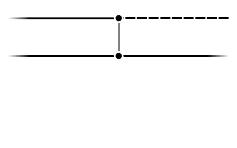}  
				\put (46,29) {\scalebox{1}{$w_0$}} 
				\put (-5,53) {\scalebox{1}{$e_1$}}	
				\put (-5,38) {\scalebox{1}{$e_2$}}	
				\put (98,38) {\scalebox{1}{$e_3$}}   
			\end{overpic}
			
			$\dtmor_0$
		\end{minipage}
		\begin{minipage}[t]{.23\textwidth}
			\centering 
			\begin{overpic}[scale=0.63]{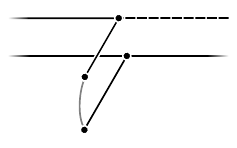}  
				\put (47,16) {\scalebox{1}{$t_1$}}
				\put (52,42) {\scalebox{1}{$A^{(1)}$}} 
				\put (39,4) {\scalebox{1}{$u$}} 
				\put (53,31) {\scalebox{1}{$v$}}   
			\end{overpic}
			
			$\Aq \dtmor 1$
		\end{minipage}
		\begin{minipage}[t]{.23\textwidth} 
			\centering
			\begin{overpic}[scale=0.63]{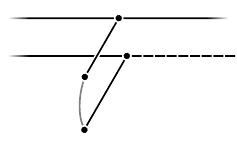}     
				\put (47,16) {\scalebox{1}{$t_1$}}
				\put (48,58) {\scalebox{1}{$A^{(2)}$}}
				\put (39,4) {\scalebox{1}{$u$}} 
				\put (53,31) {\scalebox{1}{$v$}} 
			\end{overpic} 
			
			$\Aq \dtmor 2$
		\end{minipage}
		\begin{minipage}[t]{.23\textwidth} 
			\centering
			\begin{overpic}[scale=0.63]{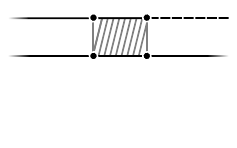}       
				\put (48,31) {\scalebox{1}{$t_1$}}
				\put (37,31) {\scalebox{1}{$u$}} 
				\put (59,31) {\scalebox{1}{$v$}} 
				\put (36,58) {\scalebox{1}{$A^{(3)}$}}
			\end{overpic}
			
			$\Aq \dtmor 3$
		\end{minipage}
		
		\begin{align*}
		\begin{matrix}
		\bbv_2 			&\bbv_3 				&& \Aq {\bbv_1} 1 	  & \Aq {\bbv_1} 2  & \Aq {\bbv_1} 3 \\
		1				& 0						&& 1				      & 0  			& 0\\
		1			 	& 0						&& 0					  & 1  			& 0\\
		0				& 1						&& 0					  & 0			& \frac 1 2	\\
		a_i				& a_i					&& 0					  & 0			& a_i	
		\end{matrix}
		\end{align*}
		\[\Aq {\bbv_1} 1 + \Aq {\bbv_1} 2 + 2\Aq {\bbv_1} 3  = \bbv_2 + \bbv_3. \]
		
		
		\item Case \{w2-r2-nd3-M-1k\}:
		
		\noindent\begin{minipage}[t]{.23\textwidth}
			\centering
			\begin{overpic}[scale=0.63]{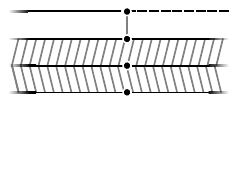}   
				\put (53,29) {\scalebox{1}{$w_0$}} 
				\put (-5,69) {\scalebox{1}{$e_1$}}	
				\put (-5,45) {\scalebox{1}{$e_2$}}	
				\put (98,45) {\scalebox{1}{$e_3$}}  
			\end{overpic} 
			$\dtmor_0$
		\end{minipage}%
		\begin{minipage}[t]{.23\textwidth}
			\centering 
			\begin{overpic}[scale=0.63]{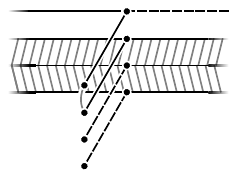} 
				\put (47,15) {\scalebox{1}{$t_1$}}
				\put (52,61) {\scalebox{1}{$A^{(1)}$}} 
				\put (39,2) {\scalebox{1}{$u$}} 
				\put (53,29) {\scalebox{1}{$v$}} 
			\end{overpic}
			
			$\Aq \dtmor 1$
		\end{minipage} 
		\begin{minipage}[t]{.23\textwidth} 
			\centering
			\begin{overpic}[scale=0.63]{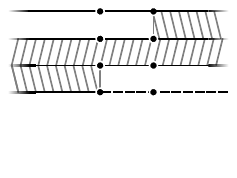}    
				\put (51,28) {\scalebox{1}{$t_1$}}
				\put (40,28) {\scalebox{1}{$u$}}
				\put (63,28) {\scalebox{1}{$v$}}
				\put (62,73) {\scalebox{1}{$A^{(2)}$}}
			\end{overpic}  
			
			$\Aq \dtmor 2$ 
		\end{minipage}
		\begin{minipage}[t]{.23\textwidth} 
			\centering
			\begin{overpic}[scale=0.63]{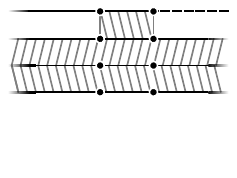}  
				\put (51,28) {\scalebox{1}{$t_1$}}
				\put (40,28) {\scalebox{1}{$u$}}
				\put (63,28) {\scalebox{1}{$v$}}
				\put (39,73) {\scalebox{1}{$A^{(3)}$}}
			\end{overpic}
			
			$\Aq \dtmor 3$ 
		\end{minipage}
		
		\begin{align*}
		\begin{matrix}
		\bbv_2 			&\bbv_3 				&& \Aq {\bbv_1} 1 	  	  & \Aq {\bbv_1} 2  & \Aq {\bbv_1} 3 \\
		1				& 0						&& 1				      & 1  				& 0\\
		\frac 1 k	 	& 0						&& 0					  & \frac 1 {k-1} 	& 0\\
		0				& \frac 1 k				&& 0					  & 0				& \frac 1 {k+1}	\\
		a_i				& a_i					&& 0					  & a_i				& a_i	
		\end{matrix}
		\end{align*}
		\[\Aq {\bbv_1} 1 + (k-1)\Aq {\bbv_1} 2 + (k+1)\Aq {\bbv_1} 3  = k \bbv_2 + k \bbv_3. \]

		
		\item Case \{w2-r2-nd3-M-kk\}: Assume that $k_1 \ge 2$ and $k_2 \ge 2$. Then $k_3 > k_1, k_2$, so Base~I is precluded since $|e_2|\not=|e_3|$ nor $|e_2|\not=|e_4|$. Base~II.2.1.M, Base~II.2.2.M, and Base~II.1.M determine $\Aq \dtmor 1$, $\Aq \dtmor 2$ and $\Aq \dtmor 3$, respectively. See figures and calculations below.
		
		
		\vspace{1em}
		
		\noindent \begin{minipage}[t]{.23\textwidth}
			\centering
			\begin{overpic}[scale=0.63]{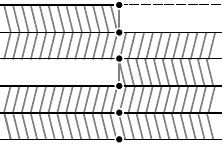}    
				\put (22,-6) {\scalebox{1}{$t_2$}}
				\put (75,-6) {\scalebox{1}{$t_3$}}
				\put (50,-6) {\scalebox{1}{$w_0$}}
				
				\put (-8,49) {\scalebox{1}{$e_1$}}	
				\put (-8,12) {\scalebox{1}{$e_2$}}	
				\put (102,23) {\scalebox{1}{$e_3$}}
				
				\put (50,67) {\scalebox{1}{$A_0$}}
			\end{overpic} 
			
			\vspace{1em}
			
			$\dtmor_0$
		\end{minipage}
		\begin{minipage}[t]{.23\textwidth}
			\centering 
			\begin{overpic}[scale=0.63]{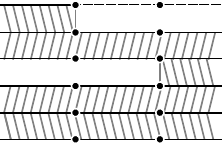}  
				\put (51,-6) {\scalebox{1}{$t_1$}}
				\put (32,-6) {\scalebox{1}{$u$}}
				\put (70,-6) {\scalebox{1}{$v$}}
				\put (62,52) {\scalebox{1}{$\Aq A 1$}} 
			\end{overpic}
			
			\vspace{1em}
			
			$\Aq \dtmor 1$ 
		\end{minipage}%
		\begin{minipage}[t]{.23\textwidth} 
			\centering
			\begin{overpic}[scale=0.63]{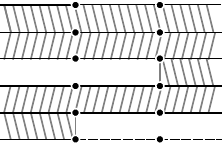}   
				\put (51,-6) {\scalebox{1}{$t_1$}}
				\put (32,-6) {\scalebox{1}{$u$}}
				\put (70,-6) {\scalebox{1}{$v$}}
				\put (65,66) {\scalebox{1}{$\Aq A 2$}}
			\end{overpic}  
			
			\vspace{1em}
			
			$\Aq \dtmor 2$
		\end{minipage}
		\begin{minipage}[t]{.23\textwidth} 
			\centering
			\begin{overpic}[scale=0.63]{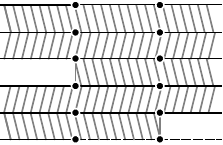} 
				\put (51,-6) {\scalebox{1}{$t_1$}}
				\put (32,-6) {\scalebox{1}{$u$}}
				\put (70,-6) {\scalebox{1}{$v$}}
				\put (30,66) {\scalebox{1}{$\Aq A 3$}}
			\end{overpic}
			
			\vspace{1em}
			
			$\Aq \dtmor 3$
		\end{minipage}
		
		\begin{align*}
		\begin{matrix}
		\bbv_2 			&\bbv_3 					&& \Aq {\bbv_1} 1 	  	  & \Aq {\bbv_1} 2  & \Aq {\bbv_1} 3 \\
		\frac 1 {k_1}	& 0							&& \frac 1 {k_1-1}		  & \frac 1 {k_1}  	& 0\\
		\frac 1 {k_2} 	& 0							&& \frac 1 {k_2}		  & \frac 1 {k_2-1}	& 0\\
		0				& \frac 1 {k_1 + k_2 - 1}	&& 0					  & 0				& \frac 1 {k_1 + k_2}	\\
		a_i				& a_i						&& a_i					  & a_i				& a_i	
		\end{matrix}
		\end{align*}
		\[(k_1-1)\Aq {\bbv_1} 1 + (k_2-1)\Aq {\bbv_1} 2 + (k_1+k_2)\Aq {\bbv_1} 3  = (k_1+k_2-1) \bbv_2 + (k_1+k_2-1)  \bbv_3. \]

		
		\item Case \{w2-r2-nd3-P\}: 
		
		\noindent  \begin{minipage}[t]{.23\textwidth}
			\centering
			\begin{overpic}[scale=0.63]{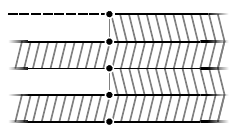}  
				\put (27,-4) {\scalebox{1}{$t_2$}}
				\put (62,-4) {\scalebox{1}{$t_3$}}
				\put (43,-4) {\scalebox{1}{$w_0$}}
				\put (-1,30) {\scalebox{1}{$e_1$}}	
				\put (-1,8)  {\scalebox{1}{$e_2$}}	
				\put (98,25) {\scalebox{1}{$e_3$}}	 
				\put (43,52) {\scalebox{1}{$A_0$}}
			\end{overpic} 
			\vspace{1em}
			
			$\dtmor_0$
		\end{minipage}%
		\begin{minipage}[t]{.23\textwidth}
			\centering 
			\begin{overpic}[scale=0.63]{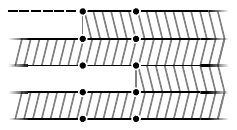}  
				\put (44,-3) {\scalebox{1}{$t_1$}}
				\put (33,-3) {\scalebox{1}{$u$}}
				\put (55,-3) {\scalebox{1}{$v$}}
				\put (54,54) {\scalebox{1}{$A^{(1)}$}} 	
			\end{overpic}\vspace{1em}
			
			$\Aq \dtmor 1$
		\end{minipage}
		\begin{minipage}[t]{.23\textwidth} 
			\centering
			\begin{overpic}[scale=0.63]{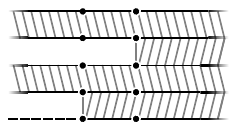}   
				\put (44,-3) {\scalebox{1}{$t_1$}}
				\put (33,-3) {\scalebox{1}{$u$}}
				\put (55,-3) {\scalebox{1}{$v$}}
				\put (32,54) {\scalebox{1}{$A^{(2)}$}}
			\end{overpic}  
			
			\vspace{1em}
			
			$\Aq \dtmor 2$
		\end{minipage}
		\begin{minipage}[t]{.23\textwidth} 
			\centering
			\begin{overpic}[scale=0.63]{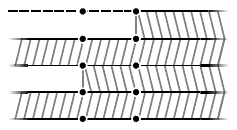} 
				\put (44,-3) {\scalebox{1}{$t_1$}}
				\put (33,-3) {\scalebox{1}{$u$}}
				\put (55,-3) {\scalebox{1}{$v$}}
				\put (54,54) {\scalebox{1}{$A^{(3)}$}}
			\end{overpic}
			
			\vspace{1em}
			
			$\Aq \dtmor 3$
		\end{minipage}
		
		\begin{align*}
		\begin{matrix}
		\bbv_2 			&\bbv_3 					&& \Aq {\bbv_1} 1 	  	  & \Aq {\bbv_1} 2  & \Aq {\bbv_1} 3 \\
		\frac 1 {k_1}	& 0							&& \frac 1 {k_1+1}		  & \frac 1 {k_1}  	& 0\\
		\frac 1 {k_2} 	& 0							&& \frac 1 {k_2}		  & \frac 1 {k_2+1}	& 0\\
		0				& \frac 1 {k_1 + k_2 + 1}	&& 0					  & 0				& \frac 1 {k_1 + k_2}	\\
		a_i				& a_i						&& a_i					  & a_i				& a_i	
		\end{matrix}
		\end{align*}
		\[(k_1+1)\Aq {\bbv_1} 1 + (k_2+1)\Aq {\bbv_1} 2 + (k_1+k_2)\Aq {\bbv_1} 3  = (k_1+k_2+1) \bbv_2 + (k_1+k_2+1)  \bbv_3. \]
		
		
		\item Case \{w2-r1\}: Assume that $r_0(A_0) = 1$. Then $\vale A_0 = 3$. Since $A_0$ satisfies no-return, let $e_2$, $e_3$ be non-dangling edges in $\Neigh {A_0}$, above $t_2$, $t_3$, respectively. Let $e_1$ be the remaining edge of $\Neigh {A_0}$. We may assume without loss of generality that~$e_1$ is above~$t_2$. By the refinement property $k_1 + k_2 = |A_0| = k_3$. Let $\tilde A$ be the unique vertex of $\ndGqAo$ with $\Aq r q(\tilde A) = 1$. If $\tilde A$ is above $u$ (resp. $v$), then the vertices of $\ndGqAo$ above $v$ (resp. $u$) belong to the Case~(r0-nd2) of the local properties (Case~(r0-nd3) would contradict that $v$ (resp. $u$) is divalent). These facts and \cite[Lemma~\ref{I-lemma-vertices-in-GqA0}]{dv20} determine the classes.

		\item Case \{w2-r1-nd3\}: Assume that $\nddeg A_0$ is 3. 
		\vspace{1em}
		
		\noindent    \begin{minipage}[t]{.30\textwidth}
			\centering 
			\begin{overpic}[scale=0.9]{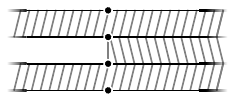}  
				\put (27,-5) {\scalebox{1}{$t_2$}}
				\put (62,-5) {\scalebox{1}{$t_3$}}
				\put (42,-5) {\scalebox{1}{$w_0$}}
				\put (-4,31) {\scalebox{1}{$e_1$}}	
				\put (-4,7) {\scalebox{1}{$e_2$}}	
				\put (96,20) {\scalebox{1}{$e_3$}}	 
				\put (43,42) {\scalebox{1}{$A_0$}}
			\end{overpic}
			
			\vspace{1em}
			
			local part around $A_0$
			\begin{align*}
			&\sigma_0(J_{A_0},2) = \frac {c(e_1)} {k_1}  + \frac {c(e_2)} {k_2}  \\
			&\sigma_0(J_{A_0},3) = \frac {c(e_3)} {k_1 + k_2 } 
			\end{align*}
		\end{minipage}\hspace{0.5em}
		\begin{minipage}[t]{.30\textwidth} 
			\centering
			\begin{overpic}[scale=0.9]{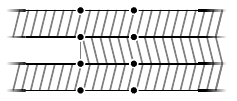} 
				\put (43,-4) {\scalebox{1}{$t_1$}}
				\put (32,-4) {\scalebox{1}{$u$}}
				\put (55,-4) {\scalebox{1}{$v$}}
				\put (31,42) {\scalebox{1}{$A^{(q)}$}}
			\end{overpic}

			\vspace{1em}
			
			$\Aq \dtmor 1 (\tilde A) = u$
			\begin{align*}
			&|e'| = k_3 \\ 
			&\Aq \sigma 1(J_{A_0},1) = \frac {c(e_3)} {k_1 + k_2} 
			\end{align*}
		\end{minipage}\hspace{0.5em}
		\begin{minipage}[t]{.30\textwidth} 
			\centering
			\begin{overpic}[scale=0.9]{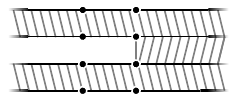} 
				\put (43,-4) {\scalebox{1}{$t_1$}}
				\put (32,-4) {\scalebox{1}{$u$}}
				\put (55,-4) {\scalebox{1}{$v$}}
				\put (54,42) {\scalebox{1}{$A^{(q)}$}}
			\end{overpic}  
			
			\vspace{1em}
			
			$\Aq \dtmor 2 (\tilde A) = v$
			\begin{align*}
			&|e'| = k_1  \quad |e''| = k_2\\
			&\Aq \sigma 2(J_{A_0},1) = \dfrac {c(e_1)} {k_1} + \dfrac {c(e_2)} {k_2}  
			\end{align*}
		\end{minipage}\hspace{1em}
		\vspace{1em} 
		\\
		Thus, \[\Aq \sigma 1(J_{A_0},1) + \Aq \sigma 2(J_{A_0},1) =  \sigma_0(J_{A_0}, 2) + \sigma_0(J_{A_0}, 3).\]
		
		\item Case \{w2-r1-nd2\}: Assume that $\nddeg A_0$ is 2. We may assume that $e_1$ is dangling. So $k_3 = k_2 + 1$. Let $h = h(e_2) = h(e_3)$.
		
		\vspace{1em}
		
		\begin{minipage}[t]{.30\textwidth}
			\centering 
			\begin{overpic}[scale=0.9]{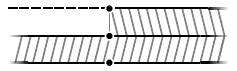}  
				\put (25,-5) {\scalebox{1}{$t_2$}}
				\put (63,-5) {\scalebox{1}{$t_3$}}
				\put (42,-5) {\scalebox{1}{$w_0$}}
				\put (-1,7) {\scalebox{1}{$e_2$}}	
				\put (97,13) {\scalebox{1}{$e_3$}}	 
				\put (43,30) {\scalebox{1}{$A_0$}}
			\end{overpic}
			
			\vspace{1em}
			
			local part around $A_0$
			\begin{align*}
			&\sigma_0(J_{A_0},2) = \dfrac {c_h} {k_1}\\
			&\sigma_0(J_{A_0},3) = \dfrac {c_h} {k_1 + 1} 
			\end{align*}
		\end{minipage}\hspace{0.5em}
		\begin{minipage}[t]{.3\textwidth} 
			\centering
			\begin{overpic}[scale=0.9]{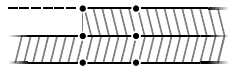} 
				\put (43,-4) {\scalebox{1}{$t_1$}}
				\put (32,-4) {\scalebox{1}{$u$}}
				\put (55,-4) {\scalebox{1}{$v$}}
				\put (31,30) {\scalebox{1}{$A^{(q)}$}}
			\end{overpic}
			
			\vspace{1em}
			
			$\Aq \dtmor 1 (\tilde A) = u$
			\begin{align*}
			&|e'| = k_2 + 1\\ 
			&\Aq \sigma 1(J_{A_0},1)= \dfrac {c_h} {k_2 + 1}  
			\end{align*} 
		\end{minipage}\hspace{0.5em} 
		\begin{minipage}[t]{.3\textwidth} 
			\centering
			\begin{overpic}[scale=0.9]{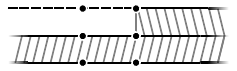}  
				\put (43,-4) {\scalebox{1}{$t_1$}}
				\put (32,-4) {\scalebox{1}{$u$}}
				\put (55,-4) {\scalebox{1}{$v$}}
				\put (54,30) {\scalebox{1}{$A^{(q)}$}}
			\end{overpic}  
			
			\vspace{1em}
			
			$\Aq \dtmor 2 (\tilde A) = v$
			\begin{align*}
			&|e'| = k_2\\
			&\Aq \sigma 2(J_{A_0},1) = \dfrac {c_h} {k_2} 
			\end{align*}
		\end{minipage}
		\vspace{1em} 
		\\
		Thus, \[\Aq \sigma 1(J_{A_0},1) + \Aq \sigma 2(J_{A_0},1) = \sigma_0(J_{A_0}, 2) + \sigma_0(J_{A_0}, 3).\]
		
		\item Proof of \eqref{I-eq-star} for case \{w2-r1\}: There is another vertex $B_0$ above $w_0$ with $r_0(B_0) = 1$. The previous analysis holds for $B_0$, with notation entirely analogous. Note that in $\Aq \dtmor q$, $\Aq \dtmor q (\tilde A) \ne \Aq \dtmor q (\tilde B)$ because $\ch u = \ch v = 1$. So $\Aq \dtmor q (\tilde A)$ determines the \DTmor, and it still holds that \[\Aq \sigma 1(J_{B_0},1) + \Aq \sigma 2(J_{B_0},1) =  \sigma_0(J_{B_0}, 2) + \sigma_0(J_{B_0}, 3).\] This gives the following calculation, which verifies Equation~\eqref{I-eq-star}: 
		\begin{align*}
		\Aq c1 &= \Aq \sigma 1(J_{A_0},1) + \Aq \sigma 1(J_{B_0},1) + s,
		&
		c^{(2)} &= \Aq \sigma 2(J_{A_0},1) + \Aq \sigma 2(J_{B_0},1) + s,
		\end{align*}
		\begin{align} \label{eq-10}
		c^{(1)} + c^{(2)} &= \sigma_0(2) + \sigma_0(3) = 0 + 0 = 0.
		\end{align}
		
	\end{itemize}


\begin{thebibliography}{ABBR15}

\bibitem[ABBR15]{abbr15}
O.~Amini, M.~Baker, E.~Brugall{\'e}, and J.~Rabinoff.
\newblock Lifting harmonic morphisms {I}: Metrized complexes and {B}erkovich skeleta.
\newblock \emph{Res. Math. Sci.}, 2(1):Art.~7, 2015.
\newblock \href{https://arxiv.org/abs/1303.4812}{arXiv:1303.4812}.

\bibitem[ACP15]{acp15}
D.~Abramovich, L.~Caporaso, and S.~Payne.
\newblock The tropicalization of the moduli space of curves.
\newblock \emph{Ann. Sci. \'Ec. Norm. Sup\'er. (4)}, 48(4):765--809, 2015.
\newblock \href{https://arxiv.org/abs/1212.0373}{arXiv:1212.0373}.

\bibitem[BBM11]{bbm11}
B.~Bertrand, E.~Brugall{\'e}, and G.~Mikhalkin.
\newblock Tropical open {H}urwitz numbers.
\newblock \emph{Rend. Semin. Mat. Univ. Padova}, 125:157--171, 2011.
\newblock \href{https://arxiv.org/abs/1005.4628}{arXiv:1005.4628}.

\bibitem[BMV11]{bmv11}
S.~Brannetti, M.~Melo, and F.~Viviani.
\newblock On the tropical {T}orelli map.
\newblock \emph{Adv. Math.}, 226(3):2546--2586, 2011.

\bibitem[BN09]{bn09}
M.~Baker and S.~Norine.
\newblock Harmonic morphisms and hyperelliptic graphs.
\newblock \emph{Int. Math. Res. Not. IMRN}, 2009(15):2914--2955, 2009.

\bibitem[Cap12]{cap12}
L.~Caporaso.
\newblock Geometry of tropical moduli spaces and linkage of graphs.
\newblock \emph{J. Combin. Theory Ser. A}, 119(3):579--598, 2012.

\bibitem[Cap14]{cap14}
L.~Caporaso.
\newblock Gonality of algebraic curves and graphs.
\newblock In \emph{Algebraic and Complex Geometry}, volume~17 of \emph{Springer Proc. Math. Stat.}, pages 77--108. Springer, 2014.
\newblock \href{https://arxiv.org/abs/1201.6246}{arXiv:1201.6246}.

\bibitem[Cas89]{cas89}
G.~Castelnuovo.
\newblock Numero delle serie lineari $\mathfrak{g}^1_n$ contenute sopra una curva di genere $p$.
\newblock \emph{Rend. Circ. Mat. Palermo}, 3:149--179, 1889.

\bibitem[CD18]{cd18}
F.~Cools and J.~Draisma.
\newblock On metric graphs with prescribed gonality.
\newblock \emph{J. Combin. Theory Ser. A}, 156:1--21, 2018.
\newblock \href{https://arxiv.org/abs/1602.05542}{arXiv:1602.05542}.

\bibitem[Cha13]{cha13}
M.~Chan.
\newblock Tropical hyperelliptic curves.
\newblock \emph{J. Algebraic Combin.}, 37(2):331--359, 2013.
\newblock \href{https://arxiv.org/abs/1110.0273}{arXiv:1110.0273}.

\bibitem[CMR16]{cmr16}
R.~Cavalieri, H.~Markwig, and D.~Ranganathan.
\newblock Tropicalizing the space of admissible covers.
\newblock \emph{Math. Ann.}, 364(3--4):1275--1313, 2016.

\bibitem[DV20]{dv20}
J.~Draisma and A.~Vargas.
\newblock Catalan-many tropical morphisms to trees; {P}art~{I}: {C}onstructions.
\newblock \emph{J. Symbolic Comput.}, 104:580--629, 2021.

\bibitem[EH87]{eh87}
D.~Eisenbud and J.~Harris.
\newblock Irreducibility and monodromy of some families of linear series.
\newblock \emph{Ann. Sci. \'Ec. Norm. Sup\'er. (4)}, 20(1):65--87, 1987.

\bibitem[GH80]{gh80}
P.~Griffiths and J.~Harris.
\newblock On the variety of special linear systems on a general algebraic curve.
\newblock \emph{Duke Math. J.}, 47(1):233--272, 1980.

\bibitem[Kem71]{kem71}
G.~Kempf.
\newblock Schubert methods with an application to algebraic curves.
\newblock \emph{Stichting Mathematisch Centrum, Zuivere Wiskunde}, ZW 6/71, 1971.

\bibitem[KL72]{kl72}
S.~L.~Kleiman and D.~Laksov.
\newblock On the existence of special divisors.
\newblock \emph{Amer. J. Math.}, 94(2):431--436, 1972.

\bibitem[Mik07]{mik07}
G.~Mikhalkin.
\newblock Tropical geometry and its applications.
\newblock In \emph{Proceedings of the International Congress of Mathematicians (Madrid, 2006)}, volume~II, pages 827--852. Eur. Math. Soc., Z\"urich, 2007.
\newblock \href{https://arxiv.org/abs/math/0601041}{arXiv:math/0601041}.

\bibitem[Rob24]{rob24}
D.~Robayo.
\newblock A combinatorial extension of tropical cycles.
\newblock \href{https://arxiv.org/abs/2410.23474}{arXiv:2410.23474}, 2024.

\bibitem[Rob25]{rob25}
D.~Robayo.
\newblock Tropical cycles of discrete admissible covers.
\newblock \href{https://arxiv.org/abs/2501.16074}{arXiv:2501.16074}, 2025.

\bibitem[Sta15]{sta15}
R.~P.~Stanley.
\newblock \emph{Catalan Numbers}.
\newblock Cambridge University Press, 2015.

\end{thebibliography}
\end{document}